\documentclass[dvipsnames]{amsart}
\usepackage[utf8]{inputenc}
\usepackage[margin=1.25in]{geometry}
\usepackage{amsmath,mathtools}
\usepackage{amssymb}
\usepackage{amsthm}
\usepackage{amsfonts}
\usepackage{enumerate}
\usepackage{mathdots}
\usepackage[boxsize=0.80cm]{ytableau}
\ytableausetup{centertableaux}
\usepackage{tikz,varwidth}
\usepackage{tikz-cd}
\usepackage{graphics}
\usepackage{url}
\usepackage{calc}

\usepackage{hyperref}
\usepackage{cleveref}

\let\amsamp=&

\usepackage{tikz}
\usetikzlibrary{calc, shapes, backgrounds,arrows,positioning,plotmarks,decorations.pathreplacing}
\tikzset{>=stealth',
  head/.style = {fill = white, text=black},
  plaque/.style = {draw, rectangle, minimum size = 10mm, fill=white}, 
     pil/.style={->,thick},
  junct/.style = {draw,circle,inner sep=0.5pt,outer sep=0pt, fill=black}
  }
\tikzcdset{scale cd/.style={every label/.append style={scale=#1},
    cells={nodes={scale=#1}}}}
\tikzset{
  otimes/.style={
    draw=none,
    every to/.append style={
      edge node={node [sloped, allow upside down, auto=false]{$\otimes$}}}
  }
}

\definecolor{light-gray}{gray}{0.90}
\definecolor{light-blue}{rgb}{0.9, 0.9, 1}

\newcommand{\blue}[1]{\textcolor{blue}{#1}}
\newcommand{\lblue}[1]{\textcolor{Cerulean}{#1}}

\newtheorem{theorem}{Theorem}[section]
\newtheorem{lemma}[theorem]{Lemma}
\newtheorem{proposition}[theorem]{Proposition}
\newtheorem{corollary}[theorem]{Corollary}
\theoremstyle{definition}
\newtheorem{definition}[theorem]{Definition}
\newtheorem{example}[theorem]{Example}
\newtheorem{remark}[theorem]{Remark}

\newtheorem{notation}[theorem]{Notation}

\numberwithin{equation}{section}

\DeclareMathOperator{\BK}{\mathsf{BK}}

\DeclareMathOperator{\ft}{FT}
\DeclareMathOperator{\id}{id}
\DeclareMathOperator{\diagram}{\mathcal{D}}
\DeclareMathOperator{\prom}{\mathsf{prom}}
\DeclareMathOperator{\promotion}{\mathcal{P}}
\DeclareMathOperator{\devacuation}{\mathcal{E}^*}
\DeclareMathOperator{\evacuation}{\mathcal{E}}
\DeclareMathOperator{\Ediagram}{\mathcal{E}-\mathsf{diagram}}
\DeclareMathOperator{\Eddiagram}{\mathcal{E}^*-\mathsf{diagram}}
\DeclareMathOperator{\Pdiagram}{\mathcal{P}-\mathsf{diagram}}
\DeclareMathOperator{\PEdiagram}{\mathcal{P}\mathcal{E}-\mathsf{diagram}}
\DeclareMathOperator{\PM}{\mathbf{M}}
\DeclareMathOperator{\PMP}{\PM_{\promotion}}
\DeclareMathOperator{\PME}{\PM_{\evacuation}}
\DeclareMathOperator{\PMEd}{\PM_{\devacuation}}
\DeclareMathOperator{\PMPE}{\PM_{\promotion\!\evacuation}}
\DeclareMathOperator{\PMr}{\overline{\mathbf{M}}}
\DeclareMathOperator{\std}{\osc}
\DeclareMathOperator{\osc}{\mathsf{osc}}
\DeclareMathOperator{\sgn}{\mathsf{sgn}}
\DeclareMathOperator{\wt}{\mathsf{wt}}
\DeclareMathOperator{\GL}{GL}
\DeclareMathOperator{\SL}{SL}
\DeclareMathOperator{\SO}{SO}
\DeclareMathOperator{\Sp}{Sp}
\DeclareMathOperator{\Spin}{Spin}
\DeclareMathOperator{\Inv}{\mathsf{Inv}}
\DeclareMathOperator{\rev}{\mathsf{rev}}

\DeclareMathOperator{\sort}{\mathsf{sort}}
\DeclareMathOperator{\type}{\mathsf{type}}
\DeclareMathOperator{\toggle}{\mathsf{switch}}
\DeclareMathOperator{\JDT}{\mathsf{jdt}}
\DeclareMathOperator{\hw}{hw}
\DeclareMathOperator{\Aexc}{\mathrm{Aexc}}

\DeclareMathOperator{\rot}{\mathsf{rot}}
\DeclareMathOperator{\refl}{\mathsf{refl}}

\newcommand{\bbb}{\mathsf{b}}

\newcommand{\too}[1]{\stackrel{#1}{\to}}

\AtBeginDocument{%
   \def\MR#1{}
}

\title{Promotion permutations for tableaux}
\author[Gaetz]{Christian Gaetz}
\address[Gaetz]{Department of Mathematics, University of California, Berkeley, CA, USA.}
\email{gaetz@berkeley.edu}

\author[Pechenik]{Oliver Pechenik}
\address[Pechenik]{Department of Combinatorics \& Optimization, University of Waterloo,  ON, Canada.}
\email{oliver.pechenik@uwaterloo.ca}

\author[Pfannerer]{Stephan Pfannerer}  
\address[Pfannerer]{Institute of Discrete Mathematics and Geometry,
Technische Universität Wien, Austria.}
\email{math@pfannerer-mittas.net}

\author[Striker]{Jessica Striker}
\address[Striker]{Department of Mathematics, North Dakota State University, Fargo, ND, USA.}
\email{jessica.striker@ndsu.edu}

\author[Swanson]{Joshua P. Swanson}
\address[Swanson]{Department of Mathematics, University of Southern California, Los Angeles, CA, USA.}
\email{swansonj@usc.edu}

\thanks{Gaetz was partially supported by a Klarman Postdoctoral Fellowship and by an NSF Postdoctoral Research Fellowship (DMS-2103121). Pechenik was partially supported by a Discovery Grant (RGPIN-2021-02391) and Launch Supplement (DGECR-2021-00010) from the Natural Sciences and Engineering Research Council of Canada. Pfannerer was partially supported by  the  Austrian  Science  Fund (FWF) P29275 and is a recipient of a DOC Fellowship of the Austrian Academy of Sciences. Striker was partially supported by a Simons Foundation/SFARI grant (527204, JS) and NSF grant DMS-2247089}

\date{\today}

\begin{document}

\begin{abstract}
  We introduce \textit{fluctuating tableaux}, which subsume many classes of tableaux that have been previously studied, including (generalized) oscillating, vacillating, rational, alternating, standard, and transpose semistandard tableaux. Our main contribution is the introduction of \emph{promotion permutations} and \emph{promotion matrices}, which are new even for standard tableaux. We provide characterizations in terms of Bender--Knuth involutions, \textit{jeu de taquin}, and crystals. We prove key properties in the rectangular case about the behavior of promotion permutations under promotion and evacuation. We also give a full development of the basic combinatorics and representation theory of fluctuating tableaux.

  Our motivation comes from our companion paper \cite{Four-row-paper}, where we use these results in the development of a new rotation-invariant $\SL_4$-\emph{web basis}. Basis elements are given by certain planar graphs and are constructed so that important algebraic operations can be performed diagrammatically. These planar graphs are indexed by fluctuating tableaux, tableau promotion corresponds to graph rotation, and promotion permutations correspond to key graphical information. 

\end{abstract}

\maketitle

\section{Introduction}

Tableaux are elementary, yet powerful, objects. They are defined simply as `numbers in boxes with rules,' yet they are central objects in algebraic combinatorics. Tableaux exhibit excellent properties from many perspectives, including enumerative (hook and Jacobi--Trudi formulas), algebraic (plactic monoid, crystals), dynamical (promotion, evacuation), and representation-theoretic (bases of Specht and Schur modules, etc.). Many of the most important families of tableaux are permuted by the action of a \emph{promotion} operator (see, e.g., \cite{Schutzenberger-promotion, Pechenik, Patrias}), which will be our main focus in this paper. While promotion on tableaux of an arbitrary-shaped partition may have large order, an amazing fact is that promotion on various types of rectangular tableaux has order dividing $n$, the maximum allowed entry. This is not only true for standard and semistandard Young tableaux (see, e.g., \cite{Haiman, Rhoades}), but also for \emph{generalized oscillating} tableaux \cite{Patrias}. 

Our motivation for this project comes from combinatorial representation theory. Specifically, in the companion papers \cite{Four-row-paper,Two-column} we develop new rotation-invariant \emph{web bases} for the Specht modules $\mathsf{S}^{(d^4)}$ and $\mathsf{S}^{(2^d)}$, building on analogous constructions for $\mathsf{S}^{(d^2)}$ (see \cite{Temperley.Lieb,Kung-Rota}) and $\mathsf{S}^{(d^3)}$ \cite{Kuperberg}. Web bases have basis elements given by certain planar graphs and are constructed in such a way that important algebraic operations can be carried out diagrammatically. A guiding principle behind our new construction in \cite{Four-row-paper} is that the long cycle $c = (1\:2\:\cdots\:n) \in \mathfrak{S}_n$ should act (up to signs) on basis elements as the permutation given by rotation of planar graphs. Moreover, the key bijection between web diagrams and tableaux should carry the action of rotation to the promotion action on tableaux. In this paper, we focus solely on tableau combinatorics without reference to webs.

To carry out the above program, we need to develop notions of \emph{promotion permutations} and \emph{promotion matrices}. These ideas are the main contribution of the current paper and are new even in the case of standard tableaux. For the sake of being able to make inductive arguments in \cite{Four-row-paper}, we need however to work in a setting that generalizes Patrias' generalized oscillating tableaux. To avoid the monstrous name `generalized generalized oscillating tableaux,' we propose the new name \emph{fluctuating tableaux} for this class (see \Cref{def:ft}) and work in that generality throughout this paper. As discussed in \Cref{rem:other_types}, fluctuating tableaux subsume many classes of tableaux that have been previously studied, including standard, dual semistandard, oscillating, generalized oscillating, vacillating, rational, and alternating tableaux.

Our main construction associates a \emph{promotion matrix} to a fluctuating tableau of general shape. Indeed, we give six different characterizations of this construction:
\begin{enumerate}[(1)]
    \item Bender--Knuth involution swap positions (see \Cref{def:PM}),
    \item decorated local rule diagrams (see \Cref{prop:PMr_sum}),
    \item row slides in jeu de taquin (see \Cref{prop:promi_alt_def}),
    \item antiexcedance sets (see \Cref{thm:antiexcedance}),
    \item first balance point conditions (see \Cref{prop:first_balance}), and
    \item a Kashiwara crystal raising algorithm (see \Cref{thm:Mbar_crystal}).
\end{enumerate}

When the shape is rectangular of $r$ rows, we further transform the promotion matrix to an $(r-1)$-tuple $(\prom_1,\ldots,\prom_{r-1})$ of \emph{promotion permutations}. The tableau may be uniquely reconstructed from its promotion permutations. For example, for a standard tableau $T$, the values appearing in the first $i$ rows are exactly the antiexcedances of the permutation $\prom_i$.
Our main result, \Cref{thm:prom_perms}, establishes important structural properties of these permutations, which are critical to the companion paper \cite{Four-row-paper} and also likely of independent interest. We also show in \Cref{thm:ft.prom_evac} that promotion has order dividing $n$ even in the rectangular fluctuating setting, unifying various special cases appearing throughout the literature on promotion.

For $r=2$, $\prom_1$ is the involution seen diagrammatically as the noncrossing matching corresponding to the $2$-row standard Young tableaux via the usual Catalan bijection (see, e.g., \cite{ec2,Tymoczko}). The promotion permutation $\prom_1$ of a $3$-row rectangular standard Young tableaux was defined and studied in~\cite{Hopkins-Rubey}, where it was further noted that $\prom_2$ could also be defined, but that it would necessarily be the inverse of $\prom_1$. 
We show in \Cref{thm:prom_perms} these properties hold more generally for rectangular fluctuating tableaux with an arbitrary number $r$ of rows: $\prom_i$ is always the inverse of $\prom_{r-i}$, and, when $r$ is even, $\prom_{r/2}$ is a fixed-point free involution. \Cref{thm:antiexcedance} shows that the antiexcedances of the promotion permutations determine the entries of the corresponding tableau. 

We give an example of promotion permutations for a rectangular standard Young tableau  below. Precise definitions  appear later in the paper.
\ytableausetup{boxsize=0.6cm}
\begin{example}\label{ex:standard_prom}
Consider the standard Young tableau below. Its promotion permutations in one-line notation are shown, with the antiexcedances highlighted in blue \textcolor{Cerulean}{$\blacksquare$}.
\begin{center}
\begin{tikzpicture}
\node (tab) at (0,0) {$E=\scalebox{0.90}{\ytableaushort{1256, 3 47{10}, 8 9 {11} {14}, {12} {13} {15} {16}}}$};
\node [right = 1em of tab] (promperm2)
    {$\prom_2(E) = 14\ \ 9\ \ 16\ \ 15\ \ 11\ \ 8\ \ 13\ \ \lblue6\ \ \lblue2\ \ 12\ \ \lblue5\ \ \lblue{10}\ \ \lblue7\ \ \lblue1\ \ \lblue4\ \ \lblue3 $};
\node [above = .1em of promperm2] (promperm1)  {$\prom_1(E) = 4\ \ 3\ \ 14\ \ 10\ \ 9\ \ 7\ \ 8\ \ 16\ \ 13\ \ 11\ \ 12\ \ \lblue6\ \ \lblue5\ \ 15\ \ \lblue2\ \ \lblue1$};
\node [below = .1em of promperm2] (promperm3)    {$\prom_3(E) = 16\ \ 15\ \ \lblue2\ \ \lblue1\ \ 13\ \ 12\ \ \lblue6\ \ \lblue7\ \ \lblue5\ \ \lblue4\ \ \lblue{10}\ \ \lblue{11}\ \ \lblue9\ \ \lblue3\ \ \lblue{14}\ \ \lblue8$};
\end{tikzpicture}
\end{center}
The reader may check that the permutations $\prom_1(E)$ and $\prom_3(E)$ are inverses of each other and that $\prom_2(E)$ is a fixed-point free involution. Note also that the antiexcedances of $\prom_i(E)$ are exactly the entries in the first $i$ rows of $E$. 
\end{example}
\ytableausetup{boxsize=0.8cm}

\Cref{sec:dihedral} illustrates our main results, \Cref{thm:ft.prom_evac,thm:prom_perms}, by showing in \Cref{cor:dihedral.1} that promotion and evacuation act as rotation and reflection on promotion permutation diagrams. This yields a dihedral model of promotion and evacuation on rectangular fluctuating tableaux (see \Cref{ex:dihedral}). These diagrams are new even for standard tableaux.

\Cref{sec:crystals} gives a crystal-theoretic interpretation of the purely combinatorial results of the preceding sections. The main result (\Cref{thm:Mbar_crystal}) is a crystal raising algorithm to compute promotion permutations and matrices. We also use crystal techniques to prove the balance point characterization, \Cref{prop:first_balance}.

The rest of this paper is devoted to proving many other beautiful properties of tableaux in the fluctuating setting, such as interpretations of promotion via growth diagrams (\Cref{sec:ft:local_rules}), Bender--Knuth involutions (\Cref{sec:BK}), jeu de taquin (\Cref{sec:jdt}), and crystals (\Cref{sec:crystals}). Much of this material is known to experts in the semistandard setting. However, even for semistandard tableaux, it is hard or impossible to find explicit proofs in the literature for many of these facts. We hope that in giving the technicalities needed for the general fluctuating case, our paper can also serve as a useful compendium of these details in more traditional settings.

This paper is organized as follows.
\Cref{sec:ft} defines fluctuating tableaux and relates these tableaux to representation theory. \Cref{sec:ft:local_rules,sec:ft:BK_jdt} give rigorous proofs of many standard facts about tableaux at the level of generality of fluctuating tableaux. 
In  \Cref{sec:prom_stuff,sec:promperm}, we move to ideas that are new even for standard Young tableaux: promotion matrices, promotion permutations, and their properties. \Cref{sec:dihedral} applies promotion permutations to give a dihedral model of promotion and evacuation for rectangular fluctuating tableaux. \Cref{sec:crystals} relates the constructions of this paper to Kashiwara's theory of crystals. 

An extended abstract describing part of this work appears in the proceedings of FPSAC 2023 \cite{fpsac-abstract}.
\section{Fluctuating tableaux}
\label{sec:ft}

Here we define fluctuating tableaux, describe their connections to representation theory, and introduce some associated basic combinatorial notions.

\subsection{Generalized partitions and fluctuating tableaux}\label{subsec:GPart}

A \emph{generalized partition} with $r$ rows is a tuple $\lambda = (\lambda_1,\ldots, \lambda_r) \in \mathbb{Z}^r$ where $\lambda_1 \geq \cdots \geq \lambda_r$. We visualize generalized partitions as \textit{diagrams}, which are semi-infinite collections of \textit{cells} (or \textit{boxes}) as in \Cref{fig:gparex}, namely
  \[ \{ (i,j) \in \mathbb{Z} \times \mathbb{Z} :  1 \leq i \leq r \text{ and } j \leq \lambda_i\} \]
using matrix indexing so that row $1$ is on top.
(Stembridge \cite{Stembridge:rational} refers to generalized partitions as \textit{staircases}; we follow \cite{Patrias} in preferring to avoid potential confusion with \textit{staircase partitions} $(k, k-1, \ldots, 1)$.)

Let $\mathcal{A}_r$ be the collection of subsets of
$\{\pm 1,\ldots,\pm r\}$ whose elements are all of the same sign. We write $\mathbf{e}_i$ for the $i$-th standard basis vector of $\mathbb{Z}^r$. If $S \in \mathcal{A}_r$ is a positive subset of $\{\pm 1,\ldots,\pm r\}$, we define $\mathbf{e}_S = \sum_{i \in S} \mathbf{e}_i$, while if $S$ is a negative subset, we define $\mathbf{e}_S = - \sum_{i \in S} \mathbf{e}_{-i}$.
We say two $r$-row generalized partitions $\lambda, \mu$ \emph{differ by a skew column} if $\lambda=\mu + \mathbf{e}_S$ for some $S \in \mathcal{A}_r$. For $c \geq 0$, we furthermore write $\mu \too{c} \lambda$ if $\lambda$ is obtained from $\mu$ by adding a skew column of $c$ boxes and $\mu \too{-c} \lambda$ if $\lambda$ is obtained from $\mu$ by removing a skew column of $c$ boxes. We often write $-c$ as $\overline{c}$. We further write $\mathbf{c} = (c, \ldots, c)$.

The following is the fundamental combinatorial object considered in this paper.

\newsavebox{\gparex}
\sbox{\gparex}{%
  \begin{tikzcd}
  \begin{tikzpicture}[inner sep=0in,outer sep=0in]
    \node (n) {\begin{varwidth}{5cm}{
    \begin{ytableau}
      \none[\cdots] & \  & \  & \  & \  & \  \\
      \none[\cdots] & \  & \  & \  & \  & \  \\
      \none[\cdots] & \  & \  & \  & \  \\
      \none[\cdots] & \  & \  & \  \\
      \none[\cdots] & \  \\
    \end{ytableau}}\end{varwidth}};
    \draw[line width=4pt,black]
        ($(n.north west) + (4*0.815cm, 0)$)
      --++(2*0.815cm, 0)
      --++(0, -2*0.815cm)
      --++(-1*0.815cm, 0)
      --++(0, -1*0.815cm)
      --++(-1*0.815cm, 0)
      --++(0, -1*0.815cm)
      --++(-2*0.815cm, 0)
      --++(0, -1*0.815cm)
      --++(2*0.815cm, 0);
  \end{tikzpicture}
  \end{tikzcd}
}

\begin{figure}[ht]
  \[
    \begin{tikzcd}
      \lambda = (2, 2, 1, 0, -2) \ar[Leftrightarrow]{r}
        & \scalebox{0.65}{\usebox{\gparex}}
    \end{tikzcd}
  \]
\caption{A generalized partition with $r=5$ rows. The right border has been drawn in bold.}\label{fig:gparex}
\end{figure}

\begin{definition}\label{def:ft}
  An $r$-row \emph{skew fluctuating tableau} of \textit{length} $n$ is a sequence
    \[ T = \lambda^0 \too{c_1} \lambda^1 \too{c_2} \cdots \too{c_n} \lambda^n \]
  of $r$-row generalized partitions such that $\lambda^{i-1}$ and $\lambda^i$ differ by a skew column obtained by adding $c_i$ or removing $-c_i$ cells for all $1 \leq i \leq n$. The partition $\lambda^0$ is called the \textit{initial shape} of $T$ and $\lambda^n$ is called the \emph{final shape} of $T$. The sequence $\underline{c} = (c_1, \ldots, c_n) \in \{0, \pm 1,  \ldots, \pm r\}^n$ is the \textit{type} of $T$. When $\lambda^0 = \mathbf{0} = (0, \dots, 0)$, we drop the adjective ``skew'' and refer to the final shape as simply the \textit{shape}. Let $\mathrm{FT}(r, n, \mu, \lambda, \underline{c})$ be the set of skew fluctuating tableaux with $r$ rows, length $n$, initial shape $\mu$, final shape $\lambda$, and type $\underline{c}$. We will drop some parameters from $\mathrm{FT}(r, n, \mu, \lambda, \underline{c})$ as convenient.
\end{definition}

Building on work of Patrias \cite{Patrias}, we visualize fluctuating tableaux by writing $i$ in the added cells of $\lambda^i - \lambda^{i-1}$ or $\overline{i}$ in the removed cells of $\lambda^{i-1} - \lambda^i$; see \Cref{fig:ft-example} for an example. We indicate the outline of the initial shape with a bold line as in \Cref{fig:gparex}.

\ytableausetup{boxsize=0.85cm}
\newsavebox{\runningT}
\sbox{\runningT}{%
\begin{tikzpicture}[inner sep=0in,outer sep=0in]
\node (n) {\begin{varwidth}{5cm}{
\begin{ytableau}
  \none & *(light-gray)1 & 3\,\overline{7} \\
  \none & *(light-gray)1 & 4\,\overline{5} \\
  \none & *(light-gray)3\,\overline{5}\,6 \\
  \overline{2}\,3 & *(light-gray)\,6 \\
\end{ytableau}}
\end{varwidth}};
\draw[ultra thick,black] ([xshift=0.85cm]n.south west)--([xshift=0.85cm]n.north west);
\end{tikzpicture}
}
\ytableausetup{boxsize=0.8cm}

\begin{figure}
\begin{tikzpicture}
  \node (begin) at (0,0) {$T = 0000 \too{2} 1100 \too{\overline{1}} 110\overline{1} \too{3} 2110 \too{1} 2210 \too{\overline{2}} 2100 \too{2} 2111 \too{\overline{1}} 1111 $};
  \node [below = 1em of begin] (runningT) {\usebox{\runningT}};
\end{tikzpicture}
\caption{The visualization of our running example fluctuating tableau $T$ with $4$ rows, length $7$, shape $\mathbf{1} = (1, 1, 1, 1)$, and type $(2, \overline{1}, 3, 1, \overline{2}, 2, \overline{1})$. Cells in the (final) shape are in light grey. The thick line indicates the outline of the initial shape $\mathbf{0} = (0, 0, 0, 0)$.}\label{fig:ft-example}
\end{figure}

Our main applications of fluctuating tableaux involve the natural fluctuating analogue of rectangular standard tableaux.

\begin{definition}
  A fluctuating tableau is \emph{rectangular} if its shape $\lambda$ is a \emph{generalized rectangle}, meaning $\lambda_1=\cdots=\lambda_r$. We say a fluctuating tableau of type $(c_1,\ldots,c_n)$ is of \emph{oscillating type}  
  if each $c_j\in \{\pm 1\}$.
\end{definition}

\begin{remark}
  Fluctuating tableaux whose type consists of non-negative integers are \emph{row-strict} or ``transpose semistandard'' tableaux. Standard tableaux are fluctuating tableaux of type $(1, \ldots, 1)$. In the context of semistandard tableaux, ``type'' is often referred to as ``content.'' For the purposes of promotion matrices and permutations, the reader who is not interested in the full generality of fluctuating tableaux will lose little by considering the transpose semistandard case.

  Fluctuating tableaux of oscillating type are referred to as \emph{up-down staircase tableaux} in \cite{Stembridge:rational} and as \emph{generalized oscillating tableaux} in \cite{Patrias}. We use the terminology from \cite{Patrias}, but drop the word `generalized' for brevity. In \cite{Patrias-Pechenik},  the rectangular $r=3$ cases with type consisting of $1$'s and $2$'s are called \emph{Russell tableaux} after \cite{Russell}. Allowing skew columns to be added or removed is essential our work in \cite{Four-row-paper} on applications of fluctuating tableaux to webs and the representation category of $\SL_4$. 
\end{remark}

\subsection{Fluctuating tableaux and \texorpdfstring{$\GL_r(\mathbb{C})$}{GLrC} representation theory}\label{sec:GL_r_rep_theory}

The irreducible rational representations of $\GL_r(\mathbb{C})$ are naturally indexed by $r$-row generalized partitions (see, e.g.,~\cite[\S2]{Stembridge:rational}). Write $V(\lambda)$ for the irreducible associated to $\lambda$. Consider the tuples $\omega_i = (1^i, 0^{r-i})$ and $\overline{\omega}_i = \omega_{\overline{i}} = (0^{r-i}, \overline{1}^i)$, which correspond to adding or removing a column of size $i$.  These tuples also correspond to exterior powers,
  \[ V(\omega_i) \cong \bigwedge\nolimits^i V \qquad\text{and}\qquad V(\overline{\omega}_i) \cong \bigwedge\nolimits^i V^* \qquad(0 \leq i \leq r), \]
where $V \cong \mathbb{C}^r$ is the defining representation and $V^*$ is the dual of $V$. In particular, $V(\mathbf{1}) = V(\omega_r) \cong \det$ and $V(\mathbf{\overline{1}}) = V(\omega_r^*) \cong \det^* \cong 1/\det = \det^{-1}$. As a shorthand, we write $\bigwedge^{-i} V \coloneqq \bigwedge^i V^*$.

Given a type $\underline{c} = (c_1, \ldots, c_n)$, let
\begin{equation}\label{eq:V_c_wedge}
  \bigwedge\nolimits^{\underline{c}} V = \bigwedge\nolimits^{c_1} V \otimes \cdots \otimes \bigwedge\nolimits^{c_n} V.
\end{equation}

The dual Pieri rule gives the following. A crystal-theoretic proof is described in \Cref{sec:crystals.tableaux}.

\begin{theorem}\label{thm:ft.irreps}
  The multiplicity of the irreducible $\GL_r(\mathbb{C})$ representation $V(\lambda)$ in $\bigwedge^{\underline{c}} V$ is the number of $r$-row fluctuating tableaux of shape $\lambda$ and type $\underline{c}$.
\end{theorem}

\begin{example}
  Recall that $V \otimes V^*$ may be identified with the set of $r \times r$ matrices. Thus $\bigwedge^{(1, -1)} V$ is the adjoint representation of $\GL_r(\mathbb{C})$. When $\underline{c} = (1, -1, 1, -1, \ldots, 1, -1)$ has length $2k$, $\bigwedge\nolimits^{\underline{c}} V$ is the $k$th tensor power of the adjoint representation. Stembridge \cite[\S4]{Stembridge:rational} refers to the corresponding fluctuating tableaux as \emph{alternating tableaux}, since they alternately add and remove single cells. The $\GL_r(\mathbb{C})$-irreducible decomposition of $(V \otimes V^*)^{\otimes k}$ may be computed by enumerating all alternating tableaux of length $2k$ according to their final shape.
\end{example}

Let $\mathfrak{S}_n$ denote the symmetric group of permutations of $[n] \coloneqq \{1, \dots, n\}$.
The factors in the tensor products $\bigwedge\nolimits^{\underline{c}} V$ may be permuted without changing the isomorphism type. Consequently, we have the following symmetry in fluctuating tableaux.

\begin{corollary}\label{cor:ft.rearrangement}
  For any permutation $\sigma \in \mathfrak{S}_n$, 
    \[ \#\ft(r, n, \lambda, (c_1, \ldots, c_n)) = \#\ft(r, n, \lambda, (c_{\sigma_1}, \ldots, c_{\sigma_n})). \]
\end{corollary}

We give a direct combinatorial proof of \Cref{cor:ft.rearrangement} in \Cref{sec:BK} by defining Bender--Knuth involutions on (skew) fluctuating tableaux. A crystal-theoretic argument is discussed in \Cref{sec:crystals.tableaux}.

It is well-known that the representations $\bigwedge\nolimits^{\underline{c}} V$ contain all of the irreducible rational representations $V(\lambda)$ of $\GL_r(\mathbb{C})$. For many purposes, one may study the $\bigwedge\nolimits^{\underline{c}} V$ instead of the $V(\lambda)$. To make this statement precise, we first need an additional definition.

\begin{definition}
  The \emph{extremal fluctuating tableau} of type $\underline{c}$ is the unique fluctuating tableau obtained by starting at $\varnothing$ and successively adding top-justified columns of size $c_i$ if $c_i \geq 0$ or removing bottom-justified columns of size $-c_i$ if $c_i \leq  0$. 
\end{definition}

\begin{example}
  The extremal fluctuating tableau of type $(2, \overline{1}, 3, 2, \overline{2}, 2, \overline{1})$ has shape $(4, 4, 0, \overline{3})$:
\ytableausetup{boxsize=0.85cm}
\newsavebox{\runningTss}
\sbox{\runningTss}{%
\begin{tikzpicture}[inner sep=0in,outer sep=0in]
\node (n) {\begin{varwidth}{5cm}{
\begin{ytableau}
  \none & \none & \none & 1 & 3 & 4 & 6 \\
  \none & \none & \none & 1 & 3 & 4 & 6 \\
  \none & \none & \none & 3\,\overline{5} \\
  \overline{7} & \overline{5} & \overline{2} \\
\end{ytableau}}
\end{varwidth}};
\draw[ultra thick,black] ([xshift=0.88cm*3]n.south west)--([xshift=0.88cm*3]n.north west);
\end{tikzpicture}
}
\ytableausetup{boxsize=0.8cm}

\begin{center}
\begin{tikzpicture}
  \node [below = 1em of begin] (runningTss) {\usebox{\runningTss}};
\end{tikzpicture}.
\end{center}
\end{example}

The shape of the extremal fluctuating tableau of type $\underline{c}$ is the generalized partition
  \[ \omega_{\underline{c}} = \sum_{i=1}^n \omega_{c_i}. \]
Moreover, every generalized partition is of the form $\omega_{\underline{c}}$ for some $\underline{c}$.
Define \emph{lexicographic order} on tuples of integers by $\underline{c} <_{\mathrm{lex}} \underline{d}$ if the leftmost nonzero entry of $\underline{d} - \underline{c}$ is positive.
It is straightforward to see that the extremal fluctuating tableau of type $\underline{c}$ is the unique fluctuating tableau of type $\underline{c}$ and shape $\omega_{\underline{c}}$, and that all other fluctuating tableaux of type $\underline{c}$ have smaller final shape. In particular, we have the following.
\begin{proposition}
  There is a unique copy of $V(\omega_{\underline{c}})$ in $\bigwedge\nolimits^{\underline{c}} V$. Each other $V(\lambda)$ in $\bigwedge\nolimits^{\underline{c}} V$ has $\lambda <_{\mathrm{lex}} \omega_{\underline{c}}$. 
\end{proposition}

\subsection{Fluctuating tableaux and \texorpdfstring{$\SL_r(\mathbb{C})$}{SLrC} representation theory}

We now summarize the role of fluctuating tableaux in the representation theory of $\SL_r(\mathbb{C})$. The rational representations are the same as for $\GL_r(\mathbb{C})$, except that $\det \cong 1$, and more generally $V(\lambda) \cong V(\mu)$ if and only if $\lambda = \mu + c \cdot \mathbf{1}$ for some $c \in \mathbb{Z}$. Consequently, we say the \emph{weight} of a fluctuating tableau is the image of the shape in the quotient $\mathbb{Z}^r/\mathbf{1}\mathbb{Z}$. Note that the weight $\mathbf{0}$ fluctuating tableaux are precisely the rectangular fluctuating tableaux.

Let $\Inv_{\SL_r}(W)$ denote the subspace of $\SL_r$-invariant elements of the representation $W$. Equivalently, this is the isotypic component of $W$ with trivial $\SL_r$-action. By \Cref{thm:ft.irreps} and the preceding paragraph, we have the following.

\begin{proposition}
  The $r$-row rectangular fluctuating tableaux of type $\underline{c}$ index a basis of
    \[ \Inv_{\SL_r}(\bigwedge\nolimits^{\underline{c}} V). \]
\end{proposition}

In particular, rectangular fluctuating tableaux with $r$ rows and type $(1^n)$ are precisely $r \times (n/r)$ rectangular standard tableaux, which index a basis for $\Inv_{\SL_r}(V^{\otimes{n}})$. This algebra of invariants may also be identified with a graded piece of the homogeneous coordinate ring of a Grassmannian with respect to its Pl\"ucker embedding.

\subsection{Lattice words}\label{sec:lattice}
Recall from \Cref{subsec:GPart}, $\mathcal{A}_r$ is the alphabet 
consisting of subsets of
$\{\pm 1,\ldots,\pm r\}$ whose elements are all of the same sign.
The \emph{lattice word} associated to a skew fluctuating tableau $T \in \ft(r, n)$ is the word $L = L(T) = w_1 \dots w_n$ in $\mathcal{A}_r$, 
  where $\lambda^i = \lambda^{i-1} + \mathbf{e}_{w_i}$. 
We may recover $T$ from $L(T)$ and the initial shape $\lambda^0$, so in the non-skew case we sometimes identify $T$ and $L(T)$. In our running example of \Cref{fig:ft-example}, $L(T) = \{12\}\overline{4}\{134\}2\{\overline{3}\overline{2}\}\{34\}\overline{1}$. Here we have suppressed commas between elements and we have suppressed the curly braces for singletons.

A word $w$ is a lattice word of a (non-skew) fluctuating tableau if and only if for every prefix $w_1 \dots w_k$ and every $1 \leq a \leq b \leq r$, we have
\begin{equation}\label{eq:lattice_inequalities}
  (\mathbf{e}_{w_1} + \cdots + \mathbf{e}_{w_k})_a \geq (\mathbf{e}_{w_1} + \cdots + \mathbf{e}_{w_k})_b.
\end{equation}
More concretely, in each prefix we require the number of $a$'s minus the number of $\overline{a}$'s to be weakly greater than the number of $b$'s minus the number of $\overline{b}$'s. A fluctuating tableau is rectangular if and only if equality holds in \eqref{eq:lattice_inequalities} for all $a,b$ with $k=n$; in this case, we call $L$ \textit{balanced}.

\subsection{Three fundamental involutions}\label{sec:involutions}

We now define three fundamental involutions on fluctuating tableaux. They are responsible for a pervasive $4$-fold symmetry in many constructions. For example, the $14$ growth rules of Khovanov--Kuperberg \cite{Khovanov-Kuperberg} in fact may be grouped into six orbits under these symmetries, with five of size $2$ and one of size $4$. See \cite[\S3]{Roby.Sottile.Stroomer.West} for a closely related $4$-fold symmetry.

\begin{definition}\label{def:involutions}
  For any tuple $\alpha = (\alpha_1, \ldots, \alpha_r)$, set $\rev(\alpha) \coloneqq (\alpha_r, \ldots, \alpha_1)$. Let
  \begin{align*}
    \tau\left(\mu \too{c} \lambda\right)
      &\coloneqq \quad \lambda \too{\overline{c}} \mu \\
    \varpi\left(\mu \too{c} \lambda\right)
      &\coloneqq \quad \rev(-\mu) \too{\overline{c}} \rev(-\lambda) \\
    \varepsilon\left(\mu \too{c} \lambda\right)
      &\coloneqq \quad \rev(-\lambda) \too{c} \rev(-\mu).
  \end{align*}
\end{definition}

We call $\tau$ \textit{time reversal}. Our notation for $\varpi$ (``\texttt{varpi}'') is inspired by the natural involution $\omega$ on symmetric functions that exchanges elementary and complete homogeneous symmetric functions; $\varpi$ corresponds to dualization on the level of rational $\GL_r(\mathbb{C})$-representations. We will later relate $\varepsilon$ to evacuation on rectangular tableaux. The following properties are clear.

\begin{lemma}\label{lem:Klein4}
  The operators $\tau,\varpi, \varepsilon$ satisfy 
  \begin{itemize}
  \item $\tau^2 = \varpi^2 = \varepsilon^2 = \id$;
  \item $\tau, \varpi, \varepsilon$ pairwise commute; and
    \item each is the composite of the other two.
  \end{itemize}
  That is, they form a representation of a Klein $4$-group.
\end{lemma}

We extend the definitions of $\tau, \varpi, \varepsilon$ in the obvious way to apply to a fluctuating tableau $T$, obtaining fluctuating tableaux $\tau(T), \varpi(T)$, and $\varepsilon(T)$. Note that the initial and final shapes of a fluctuating tableau are swapped by $\tau$. In the rectangular case, we may subtract $\mathbf{c}$ throughout to ensure the result begins with $\mathbf{0}$, in which case $\tau$ sends rectangular fluctuating tableaux with final shape $\mathbf{c}$ to rectangular fluctuating tableaux with final shape $-\mathbf{c}$. For $\varpi$, the initial and final shapes are not swapped, but they are replaced by their negative reversals, so $\varpi$ sends rectangular fluctuating tableaux of final shape $\mathbf{c}$ to rectangular fluctuating tableaux of final shape $-\mathbf{c}$. The involution $\varepsilon$ is the composite of these two and hence is an involution on rectangular fluctuating tableaux with final shape $\mathbf{c}$. Such constant shifts do not materially affect any of our arguments, so we do not mention them further.

\begin{example}\label{ex:operator_tableaux}
    Letting $T$ be the rectangular fluctuating tableau in \Cref{fig:ft-example}, we have 

\[
\ytableausetup{boxsize=0.85cm}
\tau(T) = \begin{tikzpicture}[inner sep=0in,outer sep=0in, baseline={([yshift=-.8ex]current bounding box.center)}]
\node (t) {\begin{varwidth}{5cm}{
\begin{ytableau}
 \none & *(light-gray) \overline{7} & 1 \, \overline{5} & \none \\
  \none & *(light-gray) \overline{7} & 3\, \overline{4} & \none \\
  \none & *(light-gray) \overline{2}\, 3 \, \overline{5} &  \none \\
  \overline{5} \, 6 & *(light-gray) \overline{2} & \none \\
\end{ytableau}}
\end{varwidth}};
\draw[ultra thick,black] ([xshift=1.7cm]t.south west)--([xshift=1.7cm]t.north west);
\end{tikzpicture}
\quad
\varpi(T) = \begin{tikzpicture}[inner sep=0in,outer sep=0in, baseline={([yshift=-.8ex]current bounding box.center)}]
\node (vp) {\begin{varwidth}{5cm}{
\begin{ytableau}
 \none & *(light-gray)  \overline{6} & 2 \, \overline{3} & \none \\
  \none & *(light-gray) \overline{3} \, 5 \, \overline{6} & \none & \none \\
  \overline{4} \, 5 & *(light-gray) \overline{1} &  \none \\
  \overline{3} \, 7 & *(light-gray) \overline{1} &  \none \\
\end{ytableau}}
\end{varwidth}};
\draw[ultra thick,black] ([xshift=1.7cm]vp.south west)--([xshift=1.7cm]vp.north west);
\end{tikzpicture}
\quad 
\varepsilon(T) = \begin{tikzpicture}[inner sep=0in,outer sep=0in, baseline={([yshift=-.8ex]current bounding box.center)}]
\node (ve) {\begin{varwidth}{5cm}{
\begin{ytableau}
  \none & *(light-gray) 2 & 5 \, \overline{6} \\
  \none & *(light-gray) 2 \, \overline{3} \, 5 & \none \\
  \overline{3} \, 4 & *(light-gray) 7 \\
  \overline{1} \, 5 & *(light-gray) 7 \\
\end{ytableau}}
\end{varwidth}};
\draw[ultra thick,black] ([xshift=0.85cm]ve.south west)--([xshift=0.85cm]ve.north west);
\end{tikzpicture}
\ytableausetup{boxsize=0.8cm}.
\]
\end{example}

The effects of $\tau, \varpi, \varepsilon$ on lattice words are as follows.

\begin{definition}
  Consider a word $L = w_1 \dots w_n$ in the alphabet $\mathcal{A}_r$. Define
  \begin{enumerate}[(i)]
    \item $\tau(L) = \overline{w_n} \dots \overline{w_1}$;
    \item $\varpi(L) = \varpi(w_1) \dots \varpi(w_n)$, where $\varpi(w)$ replaces each element $i$ in $w$ with $-\sgn(i)(r-|i|+1)$; and
    \item $\varepsilon(L) = \varepsilon(w_n) \dots \varepsilon(w_1)$, where $\varepsilon(w)$ replaces each element $i$ in $w$ with $\sgn(i)(r-|i|+1)$.
  \end{enumerate}
\end{definition}

It is easy to see that 
  \[ L(\tau(T)) = \tau(L(T)), \quad L(\varpi(T)) = \varpi(L(T)), \quad \text{and } L(\varepsilon(T)) = \varepsilon(L(T)).\]

\begin{example}
    Given the fluctuating tableau $T$ in \Cref{fig:ft-example}, the lattice word is
    \[ L(T)= \{12\}\overline{4}\{134\}2\{\overline{3}\overline{2}\}\{34\}\overline{1}. \]
    We then have
        \begin{align*}
        \tau(L(T)) &= 1\{\overline{4}\overline{3}\}\{23\}\overline{2}\{\overline{4}\overline{3}\overline{1}\}4\{\overline{2}\overline{1}\}, \\
        \varpi(L(T)) &= \{\overline{4}\overline{3}\}1\{\overline{421}\}\overline{3}\{23\}\{\overline{2}\overline{1}\}4, \text{ and} \\
        \varepsilon(L(T)) &= \overline{4}\{12\}\{\overline{3}\overline{2}\}3\{124\}\overline{1}\{34\}.
        \end{align*}
        Note that these lattice words correspond to the tableaux depicted in \Cref{ex:operator_tableaux}.
\end{example}

When $T$ is rectangular, we declare the initial shape of $\tau(T)$ and $\varepsilon(T)$ to be $\mathbf{0}$, so the resulting rectangular fluctuating tableaux are precisely encoded by $\tau(L(T))$ and $\varepsilon(L(T))$.

\subsection{Oscillization}\label{sec:oscillization}
It will sometimes be convenient to reduce to the oscillating case. The following operation is the natural extension of standardization to fluctuating tableaux. For our purposes, standardization additionally involves applying the ``switch'' maps which will be encountered shortly in \Cref{def:toggle}.

\begin{definition}\label{def:std}
    Consider $\mu \too{c} \lambda$. Let $\lambda = \mu + \mathbf{e}_{\{i_1 < \cdots < i_c\}}$. 
The \emph{oscillization} of $\mu \too{c} \lambda$ is the sequence
  \[ \std(\mu \too{c} \lambda) \coloneqq \mu \to \mu^1 \to \cdots \to \mu^{|c| - 1} \to \lambda, \]
where
\[
\mu^j =  \mu^{j-1} + \mathbf{e}_{i_j} 
\]
with $\mu^0 \coloneqq \mu$ and $\mu^{|c|} \coloneqq \lambda$.
\end{definition}

Given a fluctuating tableau $T = \lambda^0 \to \lambda^1 \to \cdots \to \lambda^n$, let $\std(T)$ be the concatenation of $\std(\lambda^{i-1} \to \lambda^i)$ for $1 \leq i \leq n$. 
A fluctuating tableau is in the image of oscillization if and only if it is of oscillating type.
Oscillization may be described easily in terms of lattice words. To compute $L(\std(T))$, we simply erase the curly braces from $L(T)$, where elements of each letter are written in increasing order, keeping the same initial and final shape. In our running example from \Cref{fig:ft-example}, $\std(T) = 12\overline{4}1342\overline{3}\overline{2}34\overline{1}$. Compare \Cref{fig:ft-example} to \Cref{fig:ft-std-example} for a pictorial interpretation.

\ytableausetup{boxsize=0.85cm}
\newsavebox{\runningStdT}
\sbox{\runningStdT}{%
\begin{tikzpicture}[inner sep=0in,outer sep=0in]
\node (n) {\begin{varwidth}{5cm}{
\begin{ytableau}
  \none & *(light-gray)1 & 4\,\overline{12} \\
  \none & *(light-gray)2 & 7\,\overline{9} \\
  \none & *(light-gray)5\,\overline{8}\,10 \\
  \overline{3}\,6 & *(light-gray)11 \\
\end{ytableau}}\end{varwidth}};
\draw[ultra thick,black] ([xshift=0.85cm]n.south west)--([xshift=0.85cm]n.north west);
\end{tikzpicture}
}
\ytableausetup{boxsize=0.8cm}

\begin{figure}[h]
\begin{tikzpicture}
  \node (begin) at (0,0) {$\std(T) = $};
  \node [right = 0.0cm of begin] {\usebox{\runningStdT}};
\end{tikzpicture}
\caption{The oscillization of the fluctuating tableau $T$ from \Cref{fig:ft-example}.}\label{fig:ft-std-example}
\end{figure}

\section{Local rules, promotion, and evacuation}
\label{sec:ft:local_rules}
There are many ways one may define promotion: via Bender--Knuth involutions, jeu de taquin, growth rules, or cactus group actions based on Lusztig's involution. We show all of these definitions are equivalent for fluctuating tableaux. These equivalences are known to experts for the main special cases, such as for semistandard tableaux; such experts can probably safely skip this section and the next one, as the extension to fluctuating tableaux behaves mostly as expected. 

We begin in this section with the most combinatorially simple definition of promotion and evacuation using a direct description of local rules, borrowed from van Leeuwen~\cite{vanLeeuwen}. In \Cref{sec:ft:BK_jdt}, we then define the corresponding Bender--Knuth involutions tableau-theoretically and describe jeu de taquin on fluctuating tableaux. We show there that all three combinatorial definitions of promotion are 
equivalent. 

Our emphasis on local rule diagrams has many similarities with the approaches taken in special cases by Chmutov--Glick--Pylyavskyy~\cite[\S3]{Chmutov.Glick.Pylyavskyy} and Westbury~\cite[\S5-6]{Westbury}, which in turn built on work of Lenart~\cite{Lenart} and Henriques--Kamnitzer~\cite{Henriques.Kamnitzer}; see also, \cite[\S2]{Bloom.Pechenik.Saracino} and \cite{Patrias}. 

\subsection{Local rules}
Fomin's \textit{growth diagrams} (see, e.g., \cite[Appendix~1]{ec2}) can be used to define promotion, evacuation, and jeu de taquin for (semi)standard tableaux. In \cite[Rule 4.1.1]{vanLeeuwen}, van Leeuwen introduced certain explicit \textit{local rules} generalizing those used to construct Fomin's growth diagrams. Here, we apply van Leeuwen's construction in the context of fluctuating tableaux, where we allow both addition and removal of skew columns. The approach we take here is intended to be as combinatorially direct and self-contained as possible, with no reference to Littelmann paths, dominant weights, etc.

Given a tuple $\alpha \in \mathbb{Z}^r$, let $\sort(\alpha)$ denote the generalized partition that is the weakly decreasing rearrangement of $\alpha$. 

\begin{definition}
  A \textit{local rule diagram} is a square
  \begin{equation}\label{eq:local_rule}
  \begin{tikzcd}
  \lambda \rar{d}
    & \nu \\
  \kappa \uar{c} \ar{r}[swap]{d}
    & \mu \ar{u}[swap]{c}
  \end{tikzcd}
  \end{equation}
  where
  \begin{equation}\label{eq:local_rule.sort}
    \mu = \sort(\nu + \kappa - \lambda) \qquad\text{and}\qquad
      \lambda = \sort(\nu+\kappa-\mu),
  \end{equation}
  with pointwise addition and subtraction.
  
  A \textit{local rule} fills in a missing lower right or upper left corner in \eqref{eq:local_rule} with $\mu$ or $\lambda$ as determined by \eqref{eq:local_rule.sort}. See \Cref{fig:local.rule-ex} for an application of local rules. We will shortly see that the two conditions of \eqref{eq:local_rule.sort} imply each other, so local rules result in local rule diagrams.
\end{definition}

\begin{figure}[ht]
  \begin{equation*}
  \begin{tikzcd}
  220\overline{1}\overline{1} \rar{3}
    & 32000 \\
  11\overline{1}\overline{1}\overline{2} \uar{4} \ar[dashed]{r}[swap]{3}
    & \sort(21\overline{1}0\overline{1}) = 210\overline{1}\overline{1} \ar[dashed]{u}[swap]{4}
  \end{tikzcd}
  \end{equation*}
  \caption{An example of using a local rule to fill in the lower right corner of a local rule diagram. Here, note that $21\overline{1}0\overline{1} = 32000 + 11\overline{1}\overline{1}\overline{2} - 220\overline{1}\overline{1}$.}\label{fig:local.rule-ex}
\end{figure}

In the standard case, the local rules are uniquely determined by the fact that there are at most two ways to complete \eqref{eq:local_rule}, since rank two intervals in Young's lattice have either $1$ or $2$ intermediate elements. Given $\kappa \too{1} \lambda \too{1} \nu$, we have $\mu = \lambda$ in the first case and $\mu \neq \lambda$ in the second.

We may typically reduce to the semistandard case using the following operators, which are combinatorial shadows of the $\GL_r(\mathbb{C})$-isomorphisms $\bigwedge^c V \cong \bigwedge^{r-c} V^* \otimes \det$.

\begin{definition}\label{def:toggle}
  The \emph{switch} involution is given by
  \[
    \toggle(\mu \too{c} \lambda)
      = \mu \xrightarrow{-\sgn(c) (r-|c|)} \lambda -\sgn(c) \mathbf{1}. 
  \]
\end{definition}

\begin{example}\label{ex:switch}
    Consider $2210 \too{\overline{2}} 2100$. Applying $\toggle$ gives $2210 \too{2} 3211$. Note that $\toggle$ has replaced taking away $2$ boxes with instead adding the complementary $2$ boxes.
    Similarly, we have $\toggle(1100 \too{\overline{1}} 110 \overline{1}) = 1100 \too{3} 2210$.
\end{example}

\begin{lemma}\label{lem:local.rules}
  Applying local rules results in local rule diagrams. That is, if $\kappa \too{c} \lambda \too{d} \nu$ is given and $\mu = \sort(\nu + \kappa - \lambda)$, then $\lambda = \sort(\nu+\kappa-\mu)$.
\end{lemma}

\begin{proof}
  Say $\lambda = \kappa + \mathbf{e}_S$ and $\nu = \lambda + \mathbf{e}_T$. Then $\nu + \kappa - \lambda = \kappa + \mathbf{e}_T$. Using $\toggle$ operators, we may assume without loss of generality that $c, d \geq 0$ and $S, T \subset [r]$. Let $\Delta = (S \cup T) - (S \cap T)$ be the symmetric difference. Define an equivalence relation on $\Delta$ where $i \equiv j$ if and only if $\kappa_i = \kappa_j$. Since $\lambda$ and $\nu$ are generalized partitions, the equivalence classes of $\Delta$ are of the form $\{i+1, \ldots, i+a+b\}$, for some $a,b \in \mathbb{N}$, where $\{i+1, \ldots, i+a\} \subseteq S$ and $\{i+a+1, \ldots, i+a+b\} \subseteq T$.
  
  Let $U$ be the union of $S \cap T$ and intervals $\{i+1, \ldots, i+b\}$ for each equivalence class of $\Delta$, and let $V$ be the union of $S \cap T$ and the intervals $\{i+b+1, \ldots, i+b+a\}$, so that $\mathbf{e}_S + \mathbf{e}_T = \mathbf{e}_U + \mathbf{e}_V$. We now find that $\sort(\kappa+\mathbf{e}_T) = \kappa+\mathbf{e}_U$ and $\sort(\kappa+\mathbf{e}_V) = \kappa+\mathbf{e}_S$. Hence, $\mu = \sort(\nu + \kappa - \lambda) = \sort(\kappa + \mathbf{e}_T) = \kappa + \mathbf{e}_U$, $\nu = \mu + \mathbf{e}_V$, $\kappa \too{d} \mu \too{c} \nu$, and $\sort(\nu + \kappa - \mu) = \sort(\kappa + \mathbf{e}_V) = \kappa + \mathbf{e}_S = \lambda$.
\end{proof}

The $\tau, \varpi, \varepsilon$ involutions of \Cref{def:involutions} may also be applied to local rule diagrams.

\newsavebox{\toeA}
\begin{lrbox}{\toeA}%
  \begin{tikzcd}
  \lambda \rar{d}
    & \nu \\
  \kappa \uar{c} \ar{r}[swap]{d}
    & \mu \ar{u}[swap]{c}
  \end{tikzcd}
\end{lrbox}
\newsavebox{\toeB}
\begin{lrbox}{\toeB}%
  \begin{tikzcd}
  \lambda \rar[leftarrow]{\overline{d}}
    & \nu \\
  \kappa \uar[leftarrow]{\overline{c}} \ar[leftarrow]{r}[swap]{\overline{d}}
    & \mu \ar[leftarrow]{u}[swap]{\overline{c}}
  \end{tikzcd}
\end{lrbox}
\newsavebox{\toeBrev}
\begin{lrbox}{\toeBrev}%
  \begin{tikzcd}
  \mu \rar{\overline{d}}
    & \kappa \\
  \nu \uar{\overline{c}} \ar{r}[swap]{\overline{d}}
    & \lambda \ar{u}[swap]{\overline{c}}
  \end{tikzcd}
\end{lrbox}
\newsavebox{\toeC}
\begin{lrbox}{\toeC}%
  \begin{tikzcd}
  \rev(-\lambda) \rar[leftarrow]{d}
    & \rev(-\nu) \\
  \rev(-\kappa) \uar[leftarrow]{c} \ar[leftarrow]{r}[swap]{d}
    & \rev(-\mu) \ar[leftarrow]{u}[swap]{c}
  \end{tikzcd}
\end{lrbox}
\newsavebox{\toeCrev}
\begin{lrbox}{\toeCrev}%
  \begin{tikzcd}
  \rev(-\mu) \rar{d}
    & \rev(-\kappa) \\
  \rev(-\nu) \uar{c} \ar{r}[swap]{d}
    & \rev(-\lambda) \ar{u}[swap]{c}
  \end{tikzcd}
\end{lrbox}
\newsavebox{\toeD}
\begin{lrbox}{\toeD}%
  \begin{tikzcd}
  \rev(-\lambda) \rar{\overline{d}}
    & \rev(-\nu) \\
  \rev(-\kappa) \uar{\overline{c}} \ar{r}[swap]{\overline{d}}
    & \rev(-\mu) \ar{u}[swap]{\overline{c}}
  \end{tikzcd}
\end{lrbox}

\begin{lemma}
  Applying $\tau, \varpi, \varepsilon$ to local rule diagrams results in local rule diagrams:
  \[
  \begin{tikzcd}
    \usebox{\toeA} \rar[Leftrightarrow]{\tau} \dar[Leftrightarrow,swap]{\varepsilon} \ar[Leftrightarrow]{dr}{\varpi}
      & \usebox{\toeB} \dar[Leftrightarrow]{\varepsilon} \\
    \usebox{\toeC} \rar[Leftrightarrow,swap]{\tau}
      & \usebox{\toeD}
  \end{tikzcd}.
  \]
\end{lemma}
\begin{proof}
  The result is clear for $\tau$. For $\varpi$, it reduces to the fact that $\sort(-\rev(\alpha)) = \rev(-\sort(\alpha))$. Finally, by \Cref{lem:Klein4}, $\varepsilon$ is the composite of $\tau$ and $\varpi$.
\end{proof}

For later use, it will be convenient to consistently orient growth diagrams with edges going north or east. Consequently when applying $\tau$ or $\varepsilon$, we rotate all diagrams by $180^\circ$:
 \[ 
  \tau\left(\usebox{\toeA}\right)
    = \usebox{\toeBrev} \qquad\text{and}\qquad
  \varepsilon\left(\usebox{\toeA}\right)
    = \usebox{\toeCrev}
  \]

\subsection{Promotion and evacuation diagrams}
\label{sec:prom_evac_diag}

Given a skew fluctuating tableau, we use local rules to fill certain diagrams, allowing us to compute promotion, evacuation, and dual evacuation. 

\begin{definition}
  The \textit{promotion diagram} of the skew fluctuating tableau
    \[ T = \lambda^0 \too{c_1} \lambda^1 \too{c_2} \cdots \too{c_n} \lambda^n \]
  is obtained by applying the local rules to recursively fill the bottom row of the diagram
  \begin{equation*}
  \Pdiagram(T) \coloneqq
  \begin{tikzcd}[column sep=scriptsize]
  \lambda^{00} \rar{c_1} \ar[equal]{dr}
    \arrow[rrr, start anchor=west, end anchor=east, no head, yshift=1em, decorate, decoration={brace}, "T" above=3pt]
    & \lambda^{01} \ar{r}[name=Ua]{c_2}
    & \cdots \ar{r}[name=Ub]{c_n}
    & \lambda^{0n} \ar[equal]{dr} \\
  \ 
    & \lambda^{11} \ar[dashed]{r}[name=Da]{c_2} \uar{c_1}
      \arrow[rrr, start anchor=west, end anchor=east, no head, yshift=-1em, decorate, decoration={brace, mirror}, "\promotion(T)" below=3pt]
    & \cdots \ar[dashed]{r}[name=Db]{c_n} \uar[dashed]{c_1}
    & \lambda^{1n} \rar[dashed]{c_1} \uar[dashed]{c_1}
    & \lambda^{1,n+1} \\
  \end{tikzcd}
  \end{equation*}
  from left to right, subject to the boundary conditions
    \[ \lambda^{0i} = \lambda^i, \lambda^{11} = \lambda^0, \text{ and } \lambda^{1, n+1} = \lambda^n. \]
  The \textit{promotion} of $T$ is the bottom row of $\Pdiagram(T)$:
    \[ \promotion(T) \coloneqq \lambda^{11} \too{c_2} \cdots \too{c_n} \lambda^{1n} \too{c_1} \lambda^{1, n+1}. \]
\end{definition}

\newsavebox{\ediag}
\begin{lrbox}{\ediag}%
  \begin{tikzcd}[column sep=scriptsize]
    \lambda^{00} \rar{c_1} \ar[equal]{dr} \arrow[rrr, start anchor=west, end anchor=east, no head, yshift=1em, decorate, decoration={brace}, "T" above=3pt]
      & \lambda^{01} \rar{c_2}
      & \cdots \rar{c_n}
      & \lambda^{0n} \\
    \ 
      & \lambda^{11} \rar[dashed]{c_2} \uar{c_1} \ar[equal]{dr}
      & \cdots \rar[dashed]{c_n} \uar[dashed]{c_1}
      & \lambda^{1n} \uar[dashed]{c_1} \\
    \ 
      & \ 
      & \ddots \rar[dashed]{c_n} \uar[dashed]{c_2} \ar[equal]{dr}
      & \vdots \uar[dashed]{c_2} \\
    \ 
      & \ 
      & \ 
      & \lambda^{nn} \uar[dashed]{c_n}
  \end{tikzcd}
\end{lrbox}

\newsavebox{\eddiag}
\begin{lrbox}{\eddiag}%
  \begin{tikzcd}[column sep=scriptsize]
    \lambda^{-n,0} \\
    \vdots \uar[dashed]{c_1} \rar[dashed]{c_1}
      & \ddots \ar[equal]{ul} \\
    \lambda^{-1,0} \rar[dashed]{c_1} \uar[dashed]{c_{n-1}}
      & \cdots \rar[dashed]{c_{n-1}} \uar[dashed]{c_{n-1}}
      & \lambda^{-1,n-1} \ar[equal]{ul} \\
    \lambda^{00} \uar[dashed]{c_n} \rar{c_1}
        \arrow[rrr, start anchor=west, end anchor=east, no head, yshift=-1em, decorate, decoration={brace,mirror}, "T" below=3pt]
      & \cdots \uar[dashed]{c_n} \rar{c_{n-1}}
      & \lambda^{0,n-1} \uar{c_n} \rar{c_n}
      & \lambda^{0,n} \ar[equal]{ul} \\
  \end{tikzcd}
\end{lrbox}

\begin{definition}
  The \textit{evacuation diagram} and \textit{dual evacuation diagram} of the skew fluctuating tableau
    \[ T = \lambda^0 \too{c_1} \lambda^1 \too{c_2} \cdots \too{c_n} \lambda^n \]
  are obtained by recursively applying local rules to fill the triangles
  \begin{equation*}
    \begin{tikzcd}[row sep=tiny]
      \Ediagram(T) =
        & \Eddiagram(T) = \\
      \usebox{\ediag}
        & \usebox{\eddiag}
    \end{tikzcd}
  \end{equation*}
  from upper left to lower right for $\Ediagram(T)$ and from lower right to upper left for $\Eddiagram(T)$, subject to the boundary conditions
    \[ \lambda^{0i} = \lambda^i, \lambda^{ii} = \lambda^0, \text{ and } \lambda^{i, n+i} = \lambda^n. \]
  The \textit{evacuation} of $T$ is the right column of $\Ediagram(T)$:
    \[ \evacuation(T) = \lambda^{nn} \too{c_n} \cdots \too{c_2} \lambda^{1n} \too{c_1} \lambda^{0n}.\]
  Similarly, the \textit{dual evacuation} of $T$ is the left column of $\Eddiagram(T)$:
    \[ \devacuation(T) = \lambda^{00} \too{c_n} \lambda^{-1, 0} \too{c_{n-1}} \cdots \too{c_1} \lambda^{-n,0}. \]
\end{definition}

We combine promotion, evacuation, and dual evacuation diagrams in the following 
parallelogram.

\begin{definition}
  The \textit{promotion-evacuation diagram} of the skew fluctuating tableau
    \[ T = \lambda^0 \too{c_1} \lambda^1 \too{c_2} \cdots \too{c_n} \lambda^n \]
  is obtained by recursively applying local rules to fill the parallelogram
  \begin{equation*}
  \PEdiagram(T) \coloneqq
  \begin{tikzcd}[column sep=scriptsize]
    \lambda^{00} \rar{c_1} \ar[equal]{dr} \arrow[rrr, start anchor=west, end anchor=east, no head, yshift=1em, decorate, decoration={brace}, "T" above=3pt]
      & \lambda^{01} \rar{c_2}
      & \cdots \rar{c_n}
      & \lambda^{0n} \ar[equal]{dr} \\
    \ 
      & \lambda^{11} \rar[dashed]{c_2} \uar{c_1} \ar[equal]{dr}
      & \cdots \rar[dashed]{c_n} \uar[dashed]{c_1}
      & \lambda^{1n} \rar[dashed]{c_1} \uar[dashed]{c_1}
      & \lambda^{1,n+1} \ar[equal]{dr} \\
    \ 
      & \ 
      & \ddots \rar[dashed]{c_n} \uar[dashed]{c_2} \ar[equal]{dr}
      & \vdots \rar[dashed]{c_1} \uar[dashed]{c_2}
      & \vdots \rar[dashed]{c_2} \uar[dashed]{c_2}
      & \ddots \ar[equal]{dr} \\
    \ 
      & \ 
      & \ 
      & \lambda^{nn} \rar[dashed]{c_1} \uar[dashed]{c_n}
      & \lambda^{n,n+1} \rar[dashed]{c_2} \uar[dashed]{c_n}
      & \cdots \rar[dashed]{c_n} \uar[dashed]{c_n}
      & \lambda^{n,2n}
  \end{tikzcd}
  \end{equation*}
  from upper left to lower right subject to the boundary conditions
    \[ \lambda^{0i} = \lambda^i,  \lambda^{ii} = \lambda^0, \text{ and } \lambda^{i,n+i} = \lambda^n. \]
\end{definition}

Note that the $i$th power of promotion of $T$ is the $i$th row of $\PEdiagram(T)$, indexing from $0$ at the top. In particular, the bottom row of $\PEdiagram(T)$ is $\promotion^n(T)$. The promotion-evacuation diagram is the concatenation of evacuation and dual evacuation diagrams:
\begin{equation}\label{eq:pediagram.sum}
  \PEdiagram(T) = \Ediagram(T) \mathbin\Vert \Eddiagram(\promotion^n(T)).
\end{equation}

\begin{example}\label{ex:PEdiagram_running}
  For the fluctuating tableau $T$ of \Cref{fig:ft-example}, $\PEdiagram(T)$ is
  \[
  \begin{tikzcd}[column sep=tiny, scale cd=0.7]
    0000 \rar
      & 1100 \rar
      & 110\overline{1} \rar
      & 2110 \rar
      & 2210 \rar
      & 2100 \rar
      & 2111 \rar
      & 1111 \\
    \ 
      & 0000 \rar \uar
      & 000\overline{1} \rar \uar
      & 1100 \rar \uar
      & 2100 \rar \uar
      & 200\overline{1} \rar \uar
      & 2100 \rar \uar
      & 1100 \rar \uar
      & 1111 \\
    \ 
      & \ 
      & 0000 \rar \uar
      & 1110 \rar \uar
      & 2110 \rar \uar
      & 210\overline{1} \rar \uar
      & 2200 \rar \uar
      & 2100 \rar \uar
      & 2111 \rar \uar
      & 1111 \\
    \ 
      & \ 
      & \ 
      & 0000 \rar \uar
      & 1000 \rar \uar
      & 10\overline{1}\overline{1} \rar \uar
      & 110\overline{1} \rar \uar
      & 100\overline{1} \rar \uar
      & 1100 \rar \uar
      & 1000 \rar \uar
      & 1111 \\
    \ 
      & \ 
      & \ 
      & \ 
      & 0000 \rar \uar
      & 00\overline{1}\overline{1} \rar \uar
      & 100\overline{1} \rar \uar
      & 10\overline{1}\overline{1} \rar \uar
      & 110\overline{1} \rar \uar
      & 100\overline{1} \rar \uar
      & 1110 \rar \uar
      & 1111 \\
    \ 
      & \ 
      & \ 
      & \ 
      & \ 
      & 0000 \rar \uar
      & 1100 \rar \uar
      & 110\overline{1} \rar \uar
      & 211\overline{1} \rar \uar
      & 210\overline{1} \rar \uar
      & 2210 \rar \uar
      & 2211 \rar \uar
      & 1111 \\
    \ 
      & \ 
      & \ 
      & \ 
      & \ 
      & \ 
      & 0000 \rar \uar
      & 000\overline{1} \rar \uar
      & 110\overline{1} \rar \uar
      & 11\overline{1}\overline{1} \rar \uar
      & 2100 \rar \uar
      & 2110 \rar \uar
      & 1100 \rar \uar
      & 1111 \\
    \ 
      & \ 
      & \ 
      & \ 
      & \ 
      & \ 
      & \ 
      & 0000 \rar \uar
      & 1100 \rar \uar
      & 110\overline{1} \rar \uar
      & 2110 \rar \uar
      & 2210 \rar \uar
      & 2100 \rar \uar
      & 2111 \rar \uar
      & 1111. \\
  \end{tikzcd}
  \]
  In particular,
  \begin{align*}
    T
      &= 0000 \to 1100 \to 110\overline{1} \to 2110 \to 2210 \to 2100 \to 2111 \to 1111, \\
    \promotion(T)
      &= 0000 \to 000\overline{1} \to 1100 \to 2100 \to 200\overline{1} \to 2100 \to 1100 \to 1111, \\
    \evacuation(T)
      &= 0000 \to 000\overline{1} \to 110\overline{1} \to 10\overline{1}\overline{1} \to 100\overline{1} \to 2100 \to 1100 \to 1111, \text{ and } \\
    \devacuation(T)
      &= \evacuation(T).
  \end{align*}
  In terms of lattice words,
  \begin{align*}
      L(T)
        &= \{12\}\overline{4}\{134\}2\{\overline{3}\overline{2}\}\{34\}\overline{1}, \\
      L(\promotion(T))
        &= \overline{4}\{124\}1\{\overline{4}\overline{2}\}\{24\}\overline{1}\{34\}, \text{ and } \\
      L(\evacuation(T))
        &= \overline{4}\{12\}\{\overline{3}\overline{2}\}3\{124\}\overline{1}\{34\} \\
        &= L(\devacuation(T)).
  \end{align*}
  In this particular case, $\promotion^n(T) = T$, $\evacuation(T) = \devacuation(T)$, and $L(\evacuation(T)) = \varepsilon(L(T))$. As we will see, these properties are equivalent in general and moreover always hold when $T$ is rectangular.
\end{example}

We will refer to promotion diagrams, evacuation diagrams, promotion-evacuation diagrams, etc., as \emph{growth diagrams}.

\begin{remark}
  The literature is inconsistent regarding the definition of ``promotion.'' If we replace $\Pdiagram(T)$ with the dual notion as in $\Eddiagram(T)$, the corresponding dual version of promotion is $\promotion^{-1}$. Some sources hence use ``promotion'' to refer to $\promotion^{-1}$. 
\end{remark}

The following properties are straightforward to verify from the definitions and the symmetry of local rules.

\begin{lemma}\label{lem:12_relations}
As operators on fluctuating tableaux, we have the following:

  \[
  \begin{aligned}[c]
    \operatorname{\tau} \circ \Ediagram
      &= \Eddiagram \circ \operatorname{\tau} \\
    \operatorname{\varepsilon} \circ \Ediagram
      &= \Eddiagram \circ \operatorname{\varepsilon} \\
    \operatorname{\varpi} \circ \Ediagram
      &= \Ediagram \circ \operatorname{\varpi} \\
  \end{aligned}
  \qquad
  \begin{aligned}[c]
    \operatorname{\tau} \circ \Eddiagram
      &= \Ediagram \circ \operatorname{\tau} \\
    \operatorname{\varepsilon} \circ \Eddiagram
      &= \Ediagram \circ \operatorname{\varepsilon} \\
    \operatorname{\varpi} \circ \Eddiagram
      &= \Eddiagram \circ \operatorname{\varpi} \\
  \end{aligned}
  \]

  \[
  \begin{aligned}[c]
    \operatorname{\tau} \circ \Pdiagram
      &= \Pdiagram \circ \operatorname{\tau} \circ \promotion \\
    \operatorname{\varepsilon} \circ \Pdiagram
      &= \Pdiagram \circ \operatorname{\varepsilon} \circ \promotion \\
    \operatorname{\varpi} \circ \Pdiagram
      &= \Pdiagram \circ \operatorname{\varpi} \\
  \end{aligned}
  \qquad
  \begin{aligned}[c]
    \operatorname{\tau} \circ \PEdiagram
      &= \PEdiagram \circ \evacuation \circ \devacuation \circ \operatorname{\tau} \\
    \operatorname{\varepsilon} \circ \PEdiagram
      &= \PEdiagram \circ \evacuation \circ \devacuation \circ \operatorname{\varepsilon} \\
    \operatorname{\varpi} \circ \PEdiagram
      &= \PEdiagram \circ \operatorname{\varpi}. \\
  \end{aligned}
  \]

\end{lemma}

For standard tableaux, the following is \cite[Theorem~2.1]{Stanley-promotion-evacuation}. The proof for fluctuating tableaux is essentially the same and we omit it.
\begin{lemma}\label{lem:dihedral}
On length $n$ skew fluctuating tableaux, we have the following:
\begin{enumerate}[(i)]
  \item $\promotion$ is invertible;
  \item $\evacuation \circ \evacuation = \id$, $\devacuation \circ \devacuation = \id$;
  \item $\devacuation \circ \evacuation = \promotion^n$, $\evacuation \circ \devacuation = \promotion^{-n}$; and
  \item $\promotion \circ \evacuation = \evacuation \circ \promotion^{-1}$, $\promotion \circ \devacuation = \devacuation \circ \promotion^{-1}$.
\end{enumerate}
In particular, $\promotion$ and $\evacuation$ give a representation of the infinite dihedral group, as do $\promotion$ and $\devacuation$.
\end{lemma}

We also record the following consequence of \Cref{lem:12_relations}.

\begin{lemma}\label{lem:invs.PEEd}
On skew fluctuating tableaux, we have the following:
\begin{enumerate}[(i)]
  \item $\operatorname{\varpi}$ commutes with each of $\promotion, \evacuation$, and $\devacuation$;
  \item $\operatorname{\tau} \circ \promotion = \promotion^{-1} \circ \operatorname{\tau}$ and $\operatorname{\tau} \circ \evacuation = \devacuation \circ \operatorname{\tau}$;
  \item $\operatorname{\varepsilon} \circ \promotion = \promotion^{-1} \circ \operatorname{\varepsilon}$ and $\operatorname{\varepsilon} \circ \evacuation = \devacuation \circ \operatorname{\varepsilon}$.
\end{enumerate}
\end{lemma}

Finally, the following lemma will be useful to us in studying rectangular fluctuating tableaux.

\begin{lemma}\label{lem:E.Ed.epsilon}
  Let $T$ be a length $n$ skew fluctuating tableaux. The following are equivalent:
  \begin{enumerate}[(i)]
    \item $\evacuation(T) = \varepsilon(T)$,
    \item $\devacuation(T) = \varepsilon(T)$.
  \end{enumerate}
  Moreover, they imply the following, which are also equivalent:
  \begin{enumerate}[(a)]
    \item $\evacuation(T) = \devacuation(T)$,
    \item $\promotion^n(T) = T$.
  \end{enumerate}
\end{lemma}
\begin{proof}
  The equivalence of (i) and (ii) follows from $\operatorname{\varepsilon} \circ \evacuation = \devacuation \circ \operatorname{\varepsilon}$ and the fact that all three of these operations are involutions (see \Cref{lem:Klein4,lem:dihedral}). Clearly (i) and (ii) imply (a). The equivalence of (a) and (b) follows from $\devacuation \circ \evacuation = \promotion^n$ (see \Cref{lem:dihedral}).
\end{proof}

\section{Bender--Knuth involutions and jeu de taquin}
\label{sec:ft:BK_jdt}

Our next goal is to encode the local rules in combinatorial manipulations on fluctuating tableaux. In \Cref{sec:prom_stuff}, we will use this description via jeu de taquin to define the main new objects of interest in this paper, \emph{promotion matrices} and \emph{promotion permutations}.

\subsection{Bender--Knuth involutions via local rules}\label{sec:BK}

Here we introduce Bender--Knuth involutions for fluctuating tableaux in terms of local rules and give their basic properties. For the case of semistandard tableaux, these involutions were first given in \cite{Bender.Knuth}. More precisely, our Bender--Knuth involutions are the transposes of the usual ones from \cite{Bender.Knuth}, since fluctuating tableaux generalize \emph{transpose} semistandard tableaux.

\begin{definition}
 For $1 \leq i \leq n-1$, the $i$th \textit{Bender--Knuth involution} $\BK_i$ on skew fluctuating tableaux is given by
    \begin{align*}
    \BK_i(\lambda^0 \to \cdots \to \lambda^{i-1} \to &\lambda^i \to \lambda^{i+1} \to \cdots \to \lambda^n) \\
    = \lambda^0 \to \cdots \to \lambda^{i-1} \to &\mu^i \to \lambda^{i+1} \to \cdots \to \lambda^n, \qquad 
  \end{align*}
  where $\mu^i = \sort(\lambda^{i+1} + \lambda^{i-1} - \lambda^i)$.
  Pictorially, we have
  \begin{equation}\label{eq:BKi}
  \begin{tikzcd}
    \ 
      & \ 
      & \ 
      & \lambda^i \ar{dr}{c_{i+1}} \ar[Rightarrow]{dd}{\BK_i}
      & \ 
      & \ 
      & \ \\
    \lambda^0 \rar{c_1}
      & \cdots \rar{c_{i-1}}
      & \lambda^{i-1} \ar{ur}{c_i} \ar[dashed]{dr}[swap]{c_{i+1}}
      & \ 
      & \lambda^{i+1} \rar{c_{i+2}}
      & \cdots \rar{c_n}
      & \lambda^n \\
    \ 
      & \ 
      & \ 
      & \mu^i \ar[dashed]{ur}[swap]{c_i}
      & \ 
      & \ 
      & \ \\
  \end{tikzcd}.
  \end{equation}

\end{definition}

We also extend the $\toggle$ operator to skew fluctuating tableaux. Recall that these operators toggle signs in the type of the fluctuating tableaux.

\begin{definition}
  Given a skew fluctuating tableau of length $n$ and $1 \leq i \leq n$, define $\toggle_i(T)$ by replacing the $i$th step with its $\toggle$:
  \begin{align*}
    \toggle_i(\lambda^0 \xrightarrow{c_1} \cdots \xrightarrow{c_{i-1}} \lambda^{i-1} \xrightarrow{c_i} &\lambda^i \xrightarrow{c_{i+1}} \lambda^{i+1} \xrightarrow{c_{i+2}} \cdots \xrightarrow{c_n} \lambda^n) \\
    = \lambda^0 \xrightarrow{c_1} \cdots \xrightarrow{c_{i-1}} \lambda^{i-1} \xrightarrow{\overline{r-c_i}} &\lambda^i - \sgn(c_i)  \mathbf{1} \xrightarrow{c_{i+1}} \lambda^{i+1} - \sgn(c_i)  \mathbf{1} \xrightarrow{c_{i+2}} \cdots \xrightarrow{c_n} \lambda^n - \sgn(c_i)  \mathbf{1}.
  \end{align*}
\end{definition}

\begin{example}
    The fluctuating tableau $T$ of \Cref{fig:ft-example} is shown below, center. In \Cref{ex:switch}, we computed the local action of $\toggle_5$ and $\toggle_2$. Pictorially, we have
    \[
   \begin{tikzpicture}
  \node (begin) at (0,0)  {\begin{tikzpicture}[inner sep=0in,outer sep=0in]
\node (n) {\begin{varwidth}{5cm}{
\begin{ytableau}
  \none & *(light-gray)1 & 3\,\overline{7} \\
  \none & *(light-gray)1 & 4\,\overline{5} \\
  \none & *(light-gray)3\,\overline{5}\,6 \\
  \overline{2}\,3 & *(light-gray)\,6 \\
\end{ytableau}}
\end{varwidth}};
\draw[ultra thick,black] ([xshift=0.85cm]n.south west)--([xshift=0.85cm]n.north west);
\end{tikzpicture}};
  \node[left = -1em of begin] (5arrow) {\scalebox{1.2}{$\xLeftrightarrow{\toggle_5}$}};
  \node [left = -1em of 5arrow] (switch5)  {\begin{tikzpicture}[inner sep=0in,outer sep=0in]
\node (n) {\begin{varwidth}{5cm}{
\begin{ytableau}
  \none & *(light-gray)1 & *(light-gray)3 & 5\,\overline{7} \\
  \none & *(light-gray)1 & *(light-gray)4 \\
  \none & *(light-gray)3 & *(light-gray)6 \\
  \overline{2}\,3 & *(light-gray)5 & *(light-gray)6 \\
\end{ytableau}}
\end{varwidth}};
\draw[ultra thick,black] ([xshift=0.85cm]n.south west)--([xshift=0.85cm]n.north west);
\end{tikzpicture}};
 \node[right = 1em of begin] (2arrow) {\scalebox{1.2}{$\xLeftrightarrow{\toggle_2}$}};
  \node [right = -1em of 2arrow] (switch2) {\begin{tikzpicture}[inner sep=0in,outer sep=0in]
\node (n) {\begin{varwidth}{5cm}{
\begin{ytableau}
  \none & *(light-gray)1 & *(light-gray)2 & 3\,\overline{7} \\
  \none & *(light-gray)1 & *(light-gray)2 & 4\,\overline{5} \\
  \none & *(light-gray)2 & *(light-gray)3\,\overline{5}\,6 \\
 \none & *(light-gray)3 & *(light-gray)6 \\
\end{ytableau}}
\end{varwidth}};
\draw[ultra thick,black] ([xshift=0.85cm]n.south west)--([xshift=0.85cm]n.north west);
\end{tikzpicture}};
\end{tikzpicture}.
\]
See also \Cref{fig:BK-toggle-example} for further examples.
\end{example}

The following are direct consequences of the definitions and the symmetry of local rule diagrams. All of the proofs are similar and straightforward, so we mostly omit them.

\begin{lemma}
  On skew fluctuating tableaux of length $n$, $\BK_i$ and $\toggle_j$ are involutions for which
  \begin{equation}\label{eq:BK.toggle}
    \BK_i \circ \toggle_i = \toggle_{i+1} \circ \BK_i \qquad (1 \leq i \leq n-1).
  \end{equation}
  Moreover,
  \begin{enumerate}[(i)]
    \item $\BK_i \circ \BK_j = \BK_j \circ \BK_i$ if $|i-j| > 1$;
    \item $\BK_i \circ \toggle_j = \toggle_j \circ \BK_i$ for all $1 \leq i \leq n-1$, $1 \leq j \leq n$ with $j \neq i, i+1$;
    \item $\toggle_i \circ \toggle_j = \toggle_j \circ \toggle_i$ if $i \neq j$.
  \end{enumerate}
\end{lemma}

\begin{lemma}
  On length $n$ fluctuating tableaux, we have:
  \begin{enumerate}[(i)]
    \item $\promotion = \BK_{n-1} \circ \cdots \circ \BK_1$; \label{item1}
    \item $\evacuation = \BK_1 \circ (\BK_2 \circ \BK_1) \circ \cdots \circ (\BK_{n-1} \circ \cdots \circ \BK_1)$; \label{item2}
    \item $\devacuation = (\BK_{n-1} \circ  \cdots \circ \BK_1) \circ \cdots \circ (\BK_{n-1} \circ \BK_{n-2}) \circ \BK_{n-1}$. \label{item3}
  \end{enumerate}
\end{lemma}

\begin{proof}
  The following diagram shows (\ref{item1}), by successively applying the Bender--Knuth involutions to the fluctuating tableau $T$, producing the promotion diagram of $T$.
  \begin{equation*}
  \begin{tikzcd}
    \lambda^{00} \rar
    \ar[equal]{dr}
      \ar[->,blue,rounded corners,to path= {
           ([yshift=0em,xshift=0.2em]\tikzcdmatrixname-1-1.north)
        -- ([yshift=1.6em,xshift=0.2em]\tikzcdmatrixname-1-1.north)
        -- ([yshift=1.6em,xshift=0.2em]\tikzcdmatrixname-1-2.north west)node[above]{$T$}
        -- ([yshift=1.6em]\tikzcdmatrixname-1-5.north)
        }]{}
      & \lambda^{01} \rar
      & \lambda^{02} \rar
      & \cdots \vphantom{\lambda^{11}} \rar
      & \lambda^{0n} 
      \ar[equal]{dr}\\
    \ 
      & \lambda^{11} \rar \uar
        \ar[->,blue,rounded corners,to path= {
           ([yshift=0em,xshift=0.2em]\tikzcdmatrixname-2-2.west)
        -- ([yshift=1.2em,xshift=0.2em]\tikzcdmatrixname-1-2.north west)
        -- ([yshift=1.2em]\tikzcdmatrixname-1-5.north)
        }]{}
        \ar[->,orange,dashed,rounded corners,to path= {
           ([yshift=0.0em,xshift=0.0em]\tikzcdmatrixname-2-2.south)
        -- ([yshift=0.0em,xshift=0.2em]\tikzcdmatrixname-2-3.south west)
        -- ([yshift=0.8em,xshift=0.2em]\tikzcdmatrixname-1-3.north west)node[below right,xshift=2em]{$\BK_1(T)$}
        -- ([yshift=0.8em]\tikzcdmatrixname-1-5.north)
        }]{}
        \ar[->,red,dash pattern=on 5pt off 1pt on 2pt off 2pt,rounded corners,to path= {
           ([yshift=-0.4em,xshift=0.0em]\tikzcdmatrixname-2-2.south)
        -- ([yshift=-0.4em,xshift=0.2em]\tikzcdmatrixname-2-5.south west)
        -- ([yshift=0.4em,xshift=0.2em]\tikzcdmatrixname-1-5.north west)
        -- ([yshift=0.4em]\tikzcdmatrixname-1-5.north)
        }]{}
        \ar[->,red,dash pattern=on 5pt off 1pt on 2pt off 2pt,rounded corners,to path= {
           ([yshift=-0.8em,xshift=0.0em]\tikzcdmatrixname-2-2.south)node[below right, xshift=1em]{$\BK_{n-1} \circ \cdots \circ \BK_1(T)$}
        -- ([yshift=-0.8em,xshift=0.2em]\tikzcdmatrixname-2-6.south)
        -- ([yshift=0em,xshift=0.2em]\tikzcdmatrixname-2-6.south)
        }]{}
      & \lambda^{12} \rar \uar
      & \cdots \vphantom{\lambda^{11}} \rar \uar
      & \lambda^{1n} \rar \uar
      & \lambda^{1, n+1}
  \end{tikzcd}
  \end{equation*}
  For a similar perspective on these diagrams, see \cite[$\mathsection 6$]{Speyer}.
\end{proof}

\begin{lemma}\label{lem:BK.toggle.PEEd}
  On length $n$ fluctuating tableaux with $r$ rows, we have:
  \begin{enumerate}[(i)]
    \item For $1 \leq i \leq n-1$,
    \begin{align*}
      \BK_i \circ \operatorname{\varpi} &= \operatorname{\varpi} \circ \BK_i, \\
      \BK_i \circ \operatorname{\tau} &= \operatorname{\tau} \circ \BK_{n-i}, \text{ and } \\
      \BK_i \circ \operatorname{\varepsilon} &= \operatorname{\varepsilon} \circ \BK_{n-i}.
    \end{align*}
    \item For $1 \leq i \leq n$,
    \begin{align*}
      \toggle_i \circ \operatorname{\varpi} &= \operatorname{\varpi} \circ \toggle_i, \\
      \toggle_i \circ \operatorname{\tau} &= \operatorname{\tau} \circ \toggle_{n+1-i} + \sgn(c_{n+1-i}) \mathbf{1}, \text{ and } \\
      \toggle_i \circ \operatorname{\varepsilon} &= \operatorname{\varepsilon} \circ \toggle_{n+1-i} + \sgn(c_{n+1-i})  \mathbf{1}.
    \end{align*}
    \item When $i$ is taken modulo $n$,
    \begin{align*}
      \toggle_i \circ \promotion &= \promotion \circ \toggle_{i+1}, \\
      \toggle_i \circ \evacuation &= \evacuation \circ \toggle_{n+1-i}, \text{ and } \\
      \toggle_i \circ \devacuation &= \devacuation \circ \toggle_{n+1-i}.
    \end{align*}
  \end{enumerate}
\end{lemma}

\subsection{Bender--Knuth involutions via tableaux}\label{sec:BK.comb}

We next give a more direct, tableaux-theoretic description of $\BK_i$. We broadly follow Stembridge's account of the classical Bender--Knuth involutions on semistandard tableaux~\cite{Stembridge}, extended to fluctuating tableaux. See \Cref{fig:BK-toggle-example} and \Cref{ex:running.BK_P} for examples.

\begin{definition} \label{def:free_etc}
  Let $T \in \ft(n, \underline{c})$ and fix $1 \leq i \leq n-1$. We call certain cells \textit{free}, \textit{forced}, \textit{moving}, or \textit{open} as follows.
  \begin{itemize}
    \item If $c_i \cdot c_{i+1} \geq 0$, call a cell $\bbb$ \textit{free} if it contains exactly one of $i, i+1, \overline{i},$ or $ \overline{i+1}$ and no other cell in $\bbb$'s row contains any of $i, i+1, \overline{i},$ or $ \overline{i+1}$.
    \item If $c_i \cdot c_{i+1} \leq 0$, call cells containing exactly one of $i, i+1, \overline{i}$, or $\overline{i+1}$ \textit{forced}. Call a cell \emph{moving} if it contains both $i$ and $\overline{i+1}$ or both $\overline{i}$ and $i+1$.
    Additionally, for each row $R$ which does not contain any of $i, i+1, \overline{i},$ or $\overline{i+1}$, we identify a cell $\bbb_R$ in $R$ and call it \textit{open}. Let $j$ be the largest absolute value of an entry in $R$ less than $i$, if any exist.
    \begin{itemize}
      \item If $c_i \geq 0$, let $\bbb_R$ be the cell immediately right of the cell containing $j$, or the cell containing $\overline{j}$, or if $j$ does not exist then let $\bbb_R$ be the cell immediately right of the rightmost cell of the initial shape in $R$.
      \item If $c_i \leq 0$, let $\bbb_R$ be the cell containing $j$, or the cell immediately left of the cell containing $\overline{j}$, or otherwise the rightmost cell of the initial shape in $R$.
    \end{itemize}
  \end{itemize}
\end{definition}

See \Cref{fig:BK-toggle-example} for some examples.

\ytableausetup{boxsize=0.85cm}
\newsavebox{\BKtogAa}
\sbox{\BKtogAa}{%
\begin{tikzcd}
\begin{tikzpicture}[inner sep=0in,outer sep=0in]
\node (n) {\begin{varwidth}{5cm}{
\begin{ytableau}
  \none & \none & \  & 1 & 2 & 3 \\
  \none & \none & \  & *(light-blue)1 & 3\,\overline{4} \\
  \none & \none & \  & *(light-blue)1 \\
  \none & \none & \  & *(light-blue)2 \\
  \none & \none & 1  & 2 \\
  \none & \  \\
\end{ytableau}}
\end{varwidth}};
\draw[ultra thick,black]
    ($(n.north west) + (2*0.87cm, 0)$)
  --++(1*0.87cm, 0)
  --++(0, -4*0.87cm)
  --++(-1*0.87cm, 0)
  --++(0, -1*0.87cm)
  --++(-1*0.87cm, 0)
  --++(0, -1*0.87cm)
  --++(1*0.87cm, 0);
\end{tikzpicture}
\end{tikzcd}
}
\newsavebox{\BKtogAb}
\sbox{\BKtogAb}{%
\begin{tikzcd}
\begin{tikzpicture}[inner sep=0in,outer sep=0in]
\node (n) {\begin{varwidth}{5cm}{
\begin{ytableau}
  \none & \none & \  & 1 & 2 & 3 \\
  \none & \none & \  & *(light-blue)1 & 3\,\overline{4} \\
  \none & \none & \  & *(light-blue)2 \\
  \none & \none & \  & *(light-blue)2 \\
  \none & \none & 1  & 2 \\
  \none & \  \\
\end{ytableau}}
\end{varwidth}};
\draw[ultra thick,black]
    ($(n.north west) + (2*0.87cm, 0)$)
  --++(1*0.87cm, 0)
  --++(0, -4*0.87cm)
  --++(-1*0.87cm, 0)
  --++(0, -1*0.87cm)
  --++(-1*0.87cm, 0)
  --++(0, -1*0.87cm)
  --++(1*0.87cm, 0);
\end{tikzpicture}
\end{tikzcd}
}
\newsavebox{\BKtogBa}
\sbox{\BKtogBa}{%
\begin{tikzcd}
\begin{tikzpicture}[inner sep=0in,outer sep=0in]
\node (n) {\begin{varwidth}{5cm}{
\begin{ytableau}
  \none & \none & \none & \  & *(pink) 2 & 3 \\
  \none & \none & \none & *(green)\  & 3\,\overline{4} \\
  \none & \none & \none & *(green)\  \\
  \none & \none & \none & *(lime)\overline{1}\,2 \\
  \none & \none & \none & *(pink) 2 \\
  \none & *(pink) \overline{1} & \  \\
\end{ytableau}}
\end{varwidth}};
\draw[ultra thick,black]
    ($(n.north west) + (3*0.87cm, 0)$)
  --++(1*0.87cm, 0)
  --++(0, -4*0.87cm)
  --++(-1*0.87cm, 0)
  --++(0, -1*0.87cm)
  --++(-1*0.87cm, 0)
  --++(0, -1*0.87cm)
  --++(1*0.87cm, 0);
\end{tikzpicture}
\end{tikzcd}
}
\newsavebox{\BKtogBb}
\sbox{\BKtogBb}{%
\begin{tikzcd}
\begin{tikzpicture}[inner sep=0in,outer sep=0in]
\node (n) {\begin{varwidth}{5cm}{
\begin{ytableau}
  \none & \none & \none & \  & *(pink) 1 & 3 \\
  \none & \none & \none & \  & *(lime)1\,\overline{2}\,3\,\overline{4} \\
  \none & \none & \none & \  & *(green)\  \\
  \none & \none & \none & \  & *(green)\  \\
  \none & \none & \none & *(pink) 1 \\
  \none & *(pink) \overline{2} & \  \\
\end{ytableau}}
\end{varwidth}};
\draw[ultra thick,black]
    ($(n.north west) + (3*0.87cm, 0)$)
  --++(1*0.87cm, 0)
  --++(0, -4*0.87cm)
  --++(-1*0.87cm, 0)
  --++(0, -1*0.87cm)
  --++(-1*0.87cm, 0)
  --++(0, -1*0.87cm)
  --++(1*0.87cm, 0);
\end{tikzpicture}
\end{tikzcd}
}
\newsavebox{\BKtogCa}
\sbox{\BKtogCa}{%
\begin{tikzcd}
\begin{tikzpicture}[inner sep=0in,outer sep=0in]
\node (n) {\begin{varwidth}{5cm}{
\begin{ytableau}
  \none & \none & \  & *(pink) 1 & 3 & \none \\
  \none & \none & \  & *(lime)1\,\overline{2}\,3\,\overline{4} \\
  \none & \none & \  & *(lime)1\,\overline{2} \\
  \none & \none & \  & *(green)\  \\
  \none & \none & *(pink) 1 \\
  *(pink) \overline{2} & \  \\
\end{ytableau}}
\end{varwidth}};
\draw[ultra thick,black]
    ($(n.north west) + (2*0.87cm, 0)$)
  --++(1*0.87cm, 0)
  --++(0, -4*0.87cm)
  --++(-1*0.87cm, 0)
  --++(0, -1*0.87cm)
  --++(-1*0.87cm, 0)
  --++(0, -1*0.87cm)
  --++(1*0.87cm, 0);
\end{tikzpicture}
\end{tikzcd}
}
\newsavebox{\BKtogCb}
\sbox{\BKtogCb}{%
\begin{tikzcd}
\begin{tikzpicture}[inner sep=0in,outer sep=0in]
\node (n) {\begin{varwidth}{5cm}{
\begin{ytableau}
  \none & \none & \  & *(pink) 2 & 3 & \none \\
  \none & \none & *(green)\  & 3\,\overline{4} \\
  \none & \none & *(lime)\overline{1}\,2 \\
  \none & \none & *(lime)\overline{1}\,2 \\
  \none & \none & *(pink) 2 \\
  *(pink) \overline{1} & \  \\
\end{ytableau}}
\end{varwidth}};
\draw[ultra thick,black]
    ($(n.north west) + (2*0.87cm, 0)$)
  --++(1*0.87cm, 0)
  --++(0, -4*0.87cm)
  --++(-1*0.87cm, 0)
  --++(0, -1*0.87cm)
  --++(-1*0.87cm, 0)
  --++(0, -1*0.87cm)
  --++(1*0.87cm, 0);
\end{tikzpicture}
\end{tikzcd}
}
\newsavebox{\BKtogDa}
\sbox{\BKtogDa}{%
\begin{tikzcd}
\begin{tikzpicture}[inner sep=0in,outer sep=0in]
\node (n) {\begin{varwidth}{5cm}{
\begin{ytableau}
  \none & \none & \none & \  & 3 & \none\\
  \none & \none & \none & *(light-blue)\overline{2}\,3\,\overline{4} \\
  \none & \none & \none & *(light-blue)\overline{2} \\
  \none & \none & \none & *(light-blue)\overline{1} \\
  \none & \none & \none \\
  \overline{2} & \overline{1} & \  \\
\end{ytableau}}
\end{varwidth}};
\draw[ultra thick,black]
    ($(n.north west) + (3*0.87cm, 0)$)
  --++(1*0.87cm, 0)
  --++(0, -4*0.87cm)
  --++(-1*0.87cm, 0)
  --++(0, -1*0.87cm)
  --++(-1*0.87cm, 0)
  --++(0, -1*0.87cm)
  --++(1*0.87cm, 0);
\end{tikzpicture}
\end{tikzcd}
}
\newsavebox{\BKtogDb}
\sbox{\BKtogDb}{%
\begin{tikzcd}
\begin{tikzpicture}[inner sep=0in,outer sep=0in]
\node (n) {\begin{varwidth}{5cm}{
\begin{ytableau}
  \none & \none & \none & \  & 3 & \none \\
  \none & \none & \none & *(light-blue)\overline{2}\,3\,\overline{4} \\
  \none & \none & \none & *(light-blue)\overline{1} \\
  \none & \none & \none & *(light-blue)\overline{1} \\
  \none & \none & \none \\
  \overline{2} & \overline{1} & \  \\
\end{ytableau}}
\end{varwidth}};
\draw[ultra thick,black]
    ($(n.north west) + (3*0.87cm, 0)$)
  --++(1*0.87cm, 0)
  --++(0, -4*0.87cm)
  --++(-1*0.87cm, 0)
  --++(0, -1*0.87cm)
  --++(-1*0.87cm, 0)
  --++(0, -1*0.87cm)
  --++(1*0.87cm, 0);
\end{tikzpicture}
\end{tikzcd}
}
\ytableausetup{boxsize=0.8cm}

\begin{figure}[htb]
  \[
    \begin{tikzcd}
      \scalebox{0.5}{\usebox{\BKtogAa}} \rar[leftrightarrow]{\BK_1} \dar[Leftrightarrow,swap]{\toggle_2}
        & \scalebox{0.5}{\usebox{\BKtogAb}} \dar[Leftrightarrow]{\toggle_1}
        & \scalebox{0.5}{\usebox{\BKtogBa}} \rar[leftrightarrow]{\BK_1} \dar[Leftrightarrow,swap]{\toggle_2}
        & \scalebox{0.5}{\usebox{\BKtogBb}} \dar[Leftrightarrow]{\toggle_1} \\
      \scalebox{0.5}{\usebox{\BKtogCa}} \rar[leftrightarrow]{\BK_1}
        & \scalebox{0.5}{\usebox{\BKtogCb}}
        & \scalebox{0.5}{\usebox{\BKtogDa}} \rar[leftrightarrow]{\BK_1}
        & \scalebox{0.5}{\usebox{\BKtogDb}} \\
    \end{tikzcd}
  \]
\caption{Interactions between the $\BK_i$ involutions and the $\toggle_j$ involutions. Free cells are highlighted in light blue \textcolor{light-blue}{$\blacksquare$}. Forced cells are pink \textcolor{pink}{$\blacksquare$}, moving cells are light green \textcolor{lime}{$\blacksquare$}, and open cells are darker green \textcolor{green}{$\blacksquare$}.}\label{fig:BK-toggle-example}
\end{figure}

\begin{definition}
  For $1 \leq i \leq n-1$, the $i$th \emph{Bender--Knuth involution} $\BK_i$ on skew fluctuating tableaux of length $n$ is defined combinatorially as follows.
  \begin{itemize}
    \item If $c_i \cdot c_{i+1} \geq 0$, the free cells in each column form a connected segment, with $a$ copies of $i$ or $\overline{i+1}$ at the top and $b$ copies of $i+1$ or $\overline{i}$ at the bottom. Apply $\BK_i$ by replacing each such segment with $b$ copies of $i$ or $\overline{i+1}$ at the top and $a$ copies of $i+1$ or $\overline{i}$ at the bottom, leaving all other entries of $T$ unchanged.
    \item If $c_i \cdot c_{i+1} \leq  0$, the collection of moving and open cells in each column forms a single connected segment.
    \begin{itemize}
      \item If $c_i > 0$, such a segment has $a$ moving cells above  $b$ open cells. In this case, apply $\BK_i$ by first moving all $a$ copies of $i, \overline{i+1}$ from moving cells to the bottom $a$ cells immediately left of the segment and then replacing these labels with $\overline{i}, i+1$.
      \item If $c_i < 0$, such a segment has $a$ open cells above $b$ moving cells. In this case, apply $\BK_i$ by first moving all $b$ copies of $\overline{i}, i+1$ from moving cells to the top $b$ cells immediately right of the segment and then replacing these labels with $i, \overline{i+1}$.
    \end{itemize}
    Finally, in each forced cell, replace $i$ with $i+1$, $i+1$ with $i$, $\overline{i}$ with $\overline{i+1}$, and $\overline{i+1}$ with $\overline{i}$.
  \end{itemize}
\end{definition}

\begin{lemma}
  The combinatorial $\BK_i$ involutions are well-defined and agree with the local rule definition.
\end{lemma}

\begin{proof}
  When $c_i \cdot c_{i+1} \geq 0$, it is straightforward to see that the combinatorial description agrees with the local rule description in the proof of \Cref{lem:local.rules}. When $c_i \cdot c_{i+1} \leq 0$, the combinatorial description agrees with that obtained by applying  $\toggle$ operators to the first case as in \eqref{eq:BK.toggle}.
\end{proof}

\newsavebox{\BKa}
\sbox{\BKa}{%
  \begin{tikzpicture}[inner sep=0in,outer sep=0in]
    \node (n) {\begin{varwidth}{5cm}{
    \begin{ytableau}
      \none & *(pink)1 & 3\overline{7} \\
      \none & *(pink)1 & 4\overline{5} \\
      \none & *(green) 3\overline{5}6 \\
      *(pink) \overline{2}3 & 6 \\
    \end{ytableau}}\end{varwidth}};
    \draw[ultra thick,black] ([xshift=0.8cm]n.south west)--([xshift=0.8cm]n.north west);
  \end{tikzpicture}
}
\newsavebox{\BKb}
\sbox{\BKb}{%
  \begin{tikzpicture}[inner sep=0in,outer sep=0in]
    \node (n) {\begin{varwidth}{5cm}{
    \begin{ytableau}
      \none & 2 & 3\overline{7} \\
      \none & *(light-blue)2 & 4\overline{5} \\
      \none & *(light-blue)3\overline{5}6 \\
     *(light-blue) \overline{1}3 & 6 \\
    \end{ytableau}}\end{varwidth}};
    \draw[ultra thick,black] ([xshift=0.8cm]n.south west)--([xshift=0.8cm]n.north west);
  \end{tikzpicture}
}
\newsavebox{\BKc}
\sbox{\BKc}{%
  \begin{tikzpicture}[inner sep=0in,outer sep=0in]
    \node (n) {\begin{varwidth}{5cm}{
    \begin{ytableau}
      \none & 2 & *(light-blue) 3\overline{7} \\
      \none & 2 & *(light-blue) 4\overline{5} \\
      \none & *(light-blue)3\overline{5}6 \\
      \overline{1}2 & 6 \\
    \end{ytableau}}\end{varwidth}};
    \draw[ultra thick,black] ([xshift=0.8cm]n.south west)--([xshift=0.8cm]n.north west);
  \end{tikzpicture}
}
\newsavebox{\BKd}
\sbox{\BKd}{%
  \begin{tikzpicture}[inner sep=0in,outer sep=0in]
    \node (n) {\begin{varwidth}{5cm}{
    \begin{ytableau}
      \none & 2 & 3\overline{7} &  *(green)\ \\
      \none & 2 & *(lime) 4\overline{5} \\
      \none & *(lime) 4\overline{5}6 \\
      \overline{1}2 & *(green) 6 \\
    \end{ytableau}}\end{varwidth}};
    \draw[ultra thick,black] ([xshift=0.8cm]n.south west)--([xshift=0.8cm]n.north west);
  \end{tikzpicture}
}
\newsavebox{\BKe}
\sbox{\BKe}{%
  \begin{tikzpicture}[inner sep=0in,outer sep=0in]
    \node (n) {\begin{varwidth}{5cm}{
    \begin{ytableau}
      \none & 2 & 3\overline{7} \\
      \none & *(light-blue)2\overline{4}5 \\
      \none & *(light-blue)6 \\
      \overline{1}2\overline{4}5 & 6 \\
    \end{ytableau}}\end{varwidth}};
    \draw[ultra thick,black] ([xshift=0.8cm]n.south west)--([xshift=0.8cm]n.north west);
  \end{tikzpicture}
}
\newsavebox{\BKf}
\sbox{\BKf}{%
  \begin{tikzpicture}[inner sep=0in,outer sep=0in]
    \node (n) {\begin{varwidth}{5cm}{
    \begin{ytableau}
      \none & 2 & *(pink) 3\overline{7} \\
      \none & 2\overline{4}5 & *(green)\ \\
      \none & *(pink)6 \\
      \overline{1}2\overline{4}5 & *(pink)6 \\
    \end{ytableau}}\end{varwidth}};
    \draw[ultra thick,black] ([xshift=0.8cm]n.south west)--([xshift=0.8cm]n.north west);
  \end{tikzpicture}
}
\newsavebox{\BKg}
\sbox{\BKg}{%
  \begin{tikzpicture}[inner sep=0in,outer sep=0in]
    \node (n) {\begin{varwidth}{5cm}{
    \begin{ytableau}
      \none & 2 & 3\overline{6} \\
      \none & 2\overline{4}5 \\
      \none & 7 \\
      \overline{1}2\overline{4}5 & 7 \\
    \end{ytableau}}\end{varwidth}};
    \draw[ultra thick,black] ([xshift=0.8cm]n.south west)--([xshift=0.8cm]n.north west);
  \end{tikzpicture}
}

\begin{example}\label{ex:running.BK_P}
Using the fluctuating tableau $T$ from \Cref{fig:ft-example}, we compute $\promotion(T)$ with the following composite of Bender--Knuth involutions:
\[
\begin{tikzpicture}
\node (begin) at (0,0) {$T = $};
\node [right = 0.0cm of begin] (0)  {\scalebox{0.85}{\usebox{\BKa}}};
\node [right = 0.5cm of 0] (1) {\scalebox{0.85}{\usebox{\BKb}}};
\node [right = 0.5cm of 1] (2) {\scalebox{0.85}{\usebox{\BKc}}};
\node [right = 0.5cm of 2] (3) {\scalebox{0.85}{\usebox{\BKd}}};
\node [below = 0.5cm of 3] (4) {\scalebox{0.85}{\usebox{\BKe}}};
\node [left = 0.5cm of 4] (5) {\scalebox{0.85}{\usebox{\BKf}}};
\node [left = 0.5cm of 5] (6) {\scalebox{0.85}{\usebox{\BKg}}};
\node [left = 0.0cm of 6] (end) {$\promotion(T)=$};
\draw[pil] (0) -- (1) node[midway,above] {$\BK_1$};
\draw[pil] (1) -- (2) node[midway,above] {$\BK_2$};
\draw[pil] (2) -- (3) node[midway,above] {$\BK_3$};
\draw[pil] (3) -- (4) node[midway,right] {$\BK_4$};
\draw[pil] (4) -- (5) node[midway,above] {$\BK_5$};
\draw[pil] (5) -- (6) node[midway,above] {$\BK_6$};
\end{tikzpicture}.
\]
Here, we have used the same color-coding as in \Cref{fig:BK-toggle-example}.
\end{example}

\subsection{Oscillization and local rules}

It is sometimes useful to oscillize growth diagrams. We begin by oscillizing local rules as follows. Recall $\std(\mu \too{c} \lambda)$ from \Cref{def:std}.

\begin{definition}
  The \emph{oscillization} of the local rule diagram \eqref{eq:local_rule} is the diagram:
  \begin{equation}\label{eq:local_rule.std}
  \std\left(
    \begin{tikzcd}
    \lambda \rar{d}
      & \nu \\
    \kappa \uar{c} \ar{r}[swap]{d}
      & \mu \ar{u}[swap]{c}
    \end{tikzcd}\right)
    \coloneqq
    \begin{tikzcd}
      \lambda \rar
        & \lambda^1 \rar
        & \cdots \rar
        & \lambda^{|d|-1} \rar
        & \nu \\
      \kappa^{|c|-1} \rar \uar
        & \cdot \rar \uar
        & \cdots \rar \uar
        & \cdot \rar \uar
        & \nu^{|c|-1} \uar \\
      \vdots \rar \uar
        & \vdots \rar \uar
        & \iddots \rar \uar
        & \vdots \rar \uar
        & \vdots \uar \\
      \kappa^1 \rar \uar
        & \cdot \rar \uar
        & \cdots \rar \uar
        & \cdot \rar \uar
        & \nu^1 \uar \\
      \kappa \rar \uar
        & \mu^1 \rar \uar
        & \cdots \rar \uar
        & \mu^{|d|-1} \rar \uar
        & \mu \uar
    \end{tikzcd}
  \end{equation}
  where the edges on the boundary of the large rectangle are $\std(\kappa \too{d} \mu)$, $\std(\lambda \too{d} \nu)$, $\std(\kappa \too{c} \lambda)$, and $\std(\mu \too{c} \nu)$ and the cells are filled using local rules.
\end{definition}

Note that we may fill the right-hand side of \eqref{eq:local_rule.std} from the upper left or from the lower right, or some mixture of the two. Every edge in the right-hand side of \eqref{eq:local_rule.std} either adds or removes a single cell.

In \Cref{lem:standard_welldef} we show the oscillization of a local rule diagram is well-defined. The proof uses the following lemma.

\begin{lemma}\label{lem:local_rule.refinement}
  Consider the following refinement of the local rule diagram \eqref{eq:local_rule}:
  \begin{equation}\label{eq:2X3}
  \begin{tikzcd}
  \lambda \rar{d_1}
    & \beta \rar{d_2}
    & \nu \\
  \kappa \uar{c} \ar{r}[swap]{d_1}
    & \alpha \uar{c} \ar{r}[swap]{d_2}
    & \mu \ar{u}[swap]{c}
  \end{tikzcd}.
  \end{equation}
  Here, $\kappa \too{d_1} \alpha \too{d_2} \mu$ and $\lambda \too{d_1} \beta \too{d_2} \nu$ have been obtained from $\std(\kappa \too{d} \mu)$ and $\std(\lambda \too{d} \nu)$ by collapsing edges. Then both squares in \eqref{eq:2X3} are local rule diagrams.
\end{lemma}
\begin{proof}
  Let $T = \kappa \too{c} \lambda \too{d} \nu$ and $S = \BK_1(T) = \kappa \too{d} \mu \too{c} \nu$. Further let $T' = \kappa \too{c} \lambda \too{d_1} \beta \too{d_2} \nu$ and $S' = \kappa \too{d_1} \alpha \too{d_2} \mu \too{c} \nu$. The lemma follows from showing that $S' = \BK_2 \circ \BK_1(T')$. This may be verified directly from the combinatorial description of the Bender--Knuth involutions. Using $\toggle$ involutions as necessary, we may reduce to the case that $c, d \geq 0$ and the free cells of $T$ form a single column with $c$ copies of $1$ above $d$ copies of $2$. Then, $T'$ has $c$ copies of $1$ above $d_1$ copies of $2$ above $d_2$ copies of $3$, so that $\BK_2 \circ \BK_1(T')$ has $d_1$ copies of $1$ above $d_2$ copies of $2$ above $c$ copies of $3$. On the other hand, $S$ has $d$ copies of $1$ above $c$ copies of $2$, so $S'$ also has $d_1$ copies of $1$ above $d_2$ copies of $2$ above $c$ copies of $3$.
\end{proof}

\begin{lemma}\label{lem:standard_welldef}
  The oscillization of a local rule diagram is well-defined.
\end{lemma}
\begin{proof}
  Filling the right-hand side of \eqref{eq:local_rule.std} results in an upper-left-justified collection $A$ of squares and a lower-right-justified collection $B$ of squares that meet along a path traveling by north and east steps from the lower-left corner to the upper-right corner. We must show $A$ and $B$ agree on this path. By extending $A$ to the full $|c| \times |d|$ rectangle and using symmetry of local rules, it suffices to take $B = \varnothing$. That is, given the boundary edges $\std(\kappa \to \lambda)$ and $\std(\lambda \to \nu)$, we must show the remaining edges are $\std(\kappa \to \mu)$ and $\std(\mu \to \nu)$. We may now induct on $|c|+|d|$ by \Cref{lem:local_rule.refinement}. The base case $|c|=|d|=1$ is the fundamental symmetry from \Cref{lem:local.rules}.
\end{proof}

\begin{definition}
  Let
  \begin{equation}
  \std\left(
    \begin{tikzcd}
    \lambda \rar{d}
      & \nu \\
    \ 
      & \lambda \ar{u}[swap]{d}
    \end{tikzcd}\right)
    \coloneqq
    \Ediagram(\std(\lambda \too{d} \nu))
  \end{equation}
  and
  \begin{equation}
  \std\left(
    \begin{tikzcd}
    \lambda \\
    \kappa \uar{c} \ar{r}[swap]{c}
      & \lambda
    \end{tikzcd}\right)
    \coloneqq
    \Eddiagram(\std(\kappa \too{c} \lambda)).
  \end{equation}
\end{definition}

\begin{definition}
  Let $\std(\Pdiagram(T))$, $\std(\Ediagram(T))$, $\std(\Eddiagram(T))$, and $\std(\PEdiagram(T))$ be obtained by oscillizing each component.
\end{definition}

The lemma below follows by \Cref{lem:standard_welldef}. 
\begin{lemma}
  The operation $\std$ commutes with $\Pdiagram$, $\Ediagram$, $\Eddiagram$, and $\PEdiagram$.
\end{lemma}

\subsection{Jeu de taquin}\label{sec:jdt}

In a local rule diagram
  \[
  \begin{tikzcd}
  \lambda \rar{d}
    & \nu \\
  \kappa \uar{c} \ar{r}[swap]{d}
    & \mu \ar{u}[swap]{c}
  \end{tikzcd}
  \]
we may mark the $|c|$ added or removed cells in the fluctuating tableaux $\kappa \too{c} \lambda$ and $\mu \too{c} \nu$ with $\bullet$'s if $c \geq 0$ or $\overline{\bullet}$'s if $c \leq 0$; see \Cref{ex:bullets}. Note that we may recover the full local rule diagram from its top row $\lambda \too{d} \nu$ together with the $\bullet$ or $\overline{\bullet}$ markings. Roughly speaking, when computing $\promotion(T)$, instead of directly computing $\BK_1(T), \BK_2 \circ \BK_1(T), \ldots$, we may track the positions of the $\bullet$'s or $\overline{\bullet}$'s in the top row. To formalize this, we introduce \textit{jeu de taquin} for fluctuating tableaux. We again begin with a definition via local rules, and then give a tableau-theoretic description, generalizing the description in the generalized oscillating case from~\cite{Patrias}.

\newsavebox{\JDTLRa}
\sbox{\JDTLRa}{%
  \begin{tikzpicture}[inner sep=0in,outer sep=0in]
    \node (n) {\begin{varwidth}{5cm}{
    \begin{ytableau}
      \none & \none & \none & \bullet \\
      \none & \none & \none & \bullet \\
      \none & \none & \none & \bullet \\
      \none & \none & \none \\
      \bullet \\
    \end{ytableau}}\end{varwidth}};
    \draw[ultra thick,black]
        ($(n.north west) + (2*0.815cm, 0)$)
      --++(0.815cm, 0)
      --++(0, -4*0.815cm)
      --++(-3*0.815cm, 0)
      --++(0, -0.815cm)
      --++(2*0.815cm, 0);
  \end{tikzpicture}
}
\newsavebox{\JDTLRb}
\sbox{\JDTLRb}{%
  \begin{tikzpicture}[inner sep=0in,outer sep=0in]
    \node (n) {\begin{varwidth}{5cm}{
    \begin{ytableau}
      \none & \none & \none & \ & \bullet \\
      \none & \none & \none & \ & \bullet \\
      \none & \none & \none & \ \\
      \none & \none & \none & \bullet \\
      \bullet \\
    \end{ytableau}}\end{varwidth}};
    \draw[ultra thick,black]
        ($(n.north west) + (2*0.815cm, 0)$)
      --++(0.815cm, 0)
      --++(0, -4*0.815cm)
      --++(-3*0.815cm, 0)
      --++(0, -0.815cm)
      --++(2*0.815cm, 0);
  \end{tikzpicture}
}
\newsavebox{\JDTLRc}
\sbox{\JDTLRc}{%
  \begin{tikzpicture}[inner sep=0in,outer sep=0in]
    \node (n) {\begin{varwidth}{5cm}{
    \begin{ytableau}
      \none & \none & \none & \none \\
      \none & \none & \none \\
      \none & \none & \none \\
      \none & \none & \none \\
      \none \\
    \end{ytableau}}\end{varwidth}};
    \draw[ultra thick,black]
        ($(n.north west) + (2*0.815cm, 0)$)
      --++(0.815cm, 0)
      --++(0, -4*0.815cm)
      --++(-3*0.815cm, 0)
      --++(0, -0.815cm)
      --++(2*0.815cm, 0);
  \end{tikzpicture}
}
\newsavebox{\JDTLRd}
\sbox{\JDTLRd}{%
  \begin{tikzpicture}[inner sep=0in,outer sep=0in]
    \node (n) {\begin{varwidth}{5cm}{
    \begin{ytableau}
      \none & \none & \none & \ \\
      \none & \none & \none & \ \\
      \none & \none & \none & \ \\
      \none & \none & \none \\
      \none \\
    \end{ytableau}}\end{varwidth}};
    \draw[ultra thick,black]
        ($(n.north west) + (2*0.815cm, 0)$)
      --++(0.815cm, 0)
      --++(0, -4*0.815cm)
      --++(-3*0.815cm, 0)
      --++(0, -0.815cm)
      --++(2*0.815cm, 0);
  \end{tikzpicture}.
}

\begin{example}\label{ex:bullets}
  In the following local rule associated to $1111\overline{2} \too{4} 2221\overline{1} \too{3} 3322\overline{1}$, we mark the $4$ added cells in the vertical arrows with $\bullet$'s:
  \[
  \begin{tikzcd}
    {\scalebox{0.5}{\usebox{\JDTLRa}}} \rar{3}
      & {\scalebox{0.5}{\usebox{\JDTLRb}}} \\
    {\scalebox{0.5}{\usebox{\JDTLRc}}} \ar{r}[swap]{3} \uar{4}
      & {\scalebox{0.5}{\usebox{\JDTLRd}}} \ar{u}[swap]{4} \\
  \end{tikzcd}
  \]
\end{example}

\begin{definition}
  Given a fluctuating tableau $T = \lambda^0 \to \cdots \to \lambda^n$, let $\JDT_i(T)$ be the diagram of $\BK_i \circ \cdots \circ \BK_1(T)$, where the first $i$ steps are labeled $\pm 2, \ldots, \pm (i+1)$'s, the $(i+1)$st step is labeled with $\bullet$'s or $\overline{\bullet}$'s, and the remaining steps are labeled with $\pm (i+2), \ldots, \pm n$'s.
\end{definition}

From this definition, we see immediately that $\promotion(T)$ is obtained from $\JDT_{n-1}(T)$ by replacing $\bullet$ or $\overline{\bullet}$ with $n+1$ or $\overline{n+1}$, respectively, and then decreasing the absolute value of all entries by $1$.

\newsavebox{\JDTa}
\sbox{\JDTa}{%
  \begin{tikzpicture}[inner sep=0in,outer sep=0in]
    \node (n) {\begin{varwidth}{5cm}{
    \begin{ytableau}
      \none & *(light-gray)\bullet & 3\overline{7} \\
      \none & *(light-gray)\bullet & 4\overline{5} \\
      \none & *(light-gray)3\overline{5}6 \\
      \overline{2}3 & *(light-gray)6 \\
    \end{ytableau}}\end{varwidth}};
    \draw[ultra thick,black] ([xshift=0.8cm]n.south west)--([xshift=0.8cm]n.north west);
  \end{tikzpicture}
}
\newsavebox{\JDTb}
\sbox{\JDTb}{%
  \begin{tikzpicture}[inner sep=0in,outer sep=0in]
    \node (n) {\begin{varwidth}{5cm}{
    \begin{ytableau}
      \none & *(light-gray)3 & \bullet\overline{7} \\
      \none & *(light-gray)3 & 4\overline{5} \\
      \none & *(light-gray)\bullet\overline{5}6 \\
      \overline{2}3 & *(light-gray)6 \\
    \end{ytableau}}\end{varwidth}};
    \draw[ultra thick,black] ([xshift=0.8cm]n.south west)--([xshift=0.8cm]n.north west);
  \end{tikzpicture}
}
\newsavebox{\JDTc}
\sbox{\JDTc}{%
  \begin{tikzpicture}[inner sep=0in,outer sep=0in]
    \node (n) {\begin{varwidth}{5cm}{
    \begin{ytableau}
      \none & *(light-gray)3 & 4\overline{7} \\
      \none & *(light-gray)3 & \bullet\overline{5} \\
      \none & *(light-gray)\bullet\overline{5}6 \\
      \overline{2}3 & *(light-gray)6 \\
    \end{ytableau}}\end{varwidth}};
    \draw[ultra thick,black] ([xshift=0.8cm]n.south west)--([xshift=0.8cm]n.north west);
  \end{tikzpicture}
}
\newsavebox{\JDTd}
\sbox{\JDTd}{%
  \begin{tikzpicture}[inner sep=0in,outer sep=0in]
    \node (n) {\begin{varwidth}{5cm}{
    \begin{ytableau}
      \none & *(light-gray)3 & 4\overline{7} \\
      \none & *(light-gray)3\overline{5}\bullet \\
      \none & *(light-gray)6 \\
      \overline{2}3\overline{5}\bullet & *(light-gray)6 \\
    \end{ytableau}}\end{varwidth}};
    \draw[ultra thick,black] ([xshift=0.8cm]n.south west)--([xshift=0.8cm]n.north west);
  \end{tikzpicture}
}
\newsavebox{\JDTe}
\sbox{\JDTe}{%
  \begin{tikzpicture}[inner sep=0in,outer sep=0in]
    \node (n) {\begin{varwidth}{5cm}{
    \begin{ytableau}
      \none & *(light-gray)3 & 4\overline{7} \\
      \none & *(light-gray)3\overline{5}6 \\
      \none & *(light-gray)\bullet \\
      \overline{2}3\overline{5}6 & *(light-gray)\bullet \\
    \end{ytableau}}\end{varwidth}};
    \draw[ultra thick,black] ([xshift=0.8cm]n.south west)--([xshift=0.8cm]n.north west);
  \end{tikzpicture}
}
\newsavebox{\JDTf}
\sbox{\JDTf}{%
  \begin{tikzpicture}[inner sep=0in,outer sep=0in]
    \node (n) {\begin{varwidth}{5cm}{
    \begin{ytableau}
      \none & *(light-gray)3 & 4\overline{7} \\
      \none & *(light-gray)3\overline{5}6 \\
      \none & *(light-gray)8 \\
      \overline{2}3\overline{5}6 & *(light-gray)8 \\
    \end{ytableau}}\end{varwidth}};
    \draw[ultra thick,black] ([xshift=0.8cm]n.south west)--([xshift=0.8cm]n.north west);
  \end{tikzpicture}
}
\newsavebox{\JDTg}
\sbox{\JDTg}{%
  \begin{tikzpicture}[inner sep=0in,outer sep=0in]
    \node (n) {\begin{varwidth}{5cm}{
    \begin{ytableau}
      \none & *(light-gray)2 & 3\overline{6} \\
      \none & *(light-gray)2\overline{4}5 \\
      \none & *(light-gray)7 \\
      \overline{1}2\overline{4}5 & *(light-gray)7 \\
    \end{ytableau}}\end{varwidth}};
    \draw[ultra thick,black] ([xshift=0.8cm]n.south west)--([xshift=0.8cm]n.north west);
  \end{tikzpicture}
}

\begin{example}\label{ex:JDT_prom}
For our fluctuating tableau $T$ from \Cref{fig:ft-example}, first draw $\PEdiagram(T)$ and record the location of the added cells in the top row using two $\bullet$'s as in \Cref{ex:bullets}.
\[
\begin{tikzcd}[column sep=small]
  \varnothing\vphantom{\scalebox{0.4}{\ytableaushort{\none\none\none,\none,\none,\none}}} \rar
    & {\scalebox{0.4}{\ytableaushort{\none\bullet\none,\none\bullet,\none,\none}}} \rar
    & {\scalebox{0.4}{\ytableaushort{\none\bullet\none,\none\bullet,\none,\ }}} \rar
    & {\scalebox{0.4}{\ytableaushort{\none\ \bullet,\none\ \bullet,\none,\none}}} \rar
    & {\scalebox{0.4}{\ytableaushort{\none\ \ ,\none\ \bullet,\none\bullet,\none}}} \rar
    & {\scalebox{0.4}{\ytableaushort{\none\ \ ,\none\bullet,\none,\bullet}}} \rar
    & {\scalebox{0.4}{\ytableaushort{\none\ \ ,\none\ ,\none\bullet,\none\bullet}}} \rar
    & {\scalebox{0.4}{\ytableaushort{\none\ \none,\none\ ,\none\bullet,\none\bullet}}} \\
  \ 
    & \varnothing\vphantom{\scalebox{0.4}{\ytableaushort{\none\none\none,\none,\none,\none}}} \rar \uar
      \ar[->,dashed,rounded corners,to path= {
             ([yshift=0.0em,xshift=0.0em]\tikzcdmatrixname-2-2.south)node[above,xshift=2.3em]{$\overline{2}$}
          -- ([yshift=0.0em,xshift=0.0em]\tikzcdmatrixname-2-3.south)node[above,xshift=2.3em]{$3$}
          -- ([yshift=0.0em,xshift=0.0em]\tikzcdmatrixname-2-4.south)node[above,xshift=2.3em]{$4$}
          -- ([yshift=0.0em,xshift=0.8em]\tikzcdmatrixname-2-5.south west)node[left,xshift=1em,yshift=5.2em]{$\bullet$}
          -- ([yshift=0.2em,xshift=0.8em]\tikzcdmatrixname-1-5.north west)node[above right,xshift=2em]{$\BK_3 \circ \BK_2 \circ \BK_1(T)$}node[below right,xshift=3em]{$\overline{5}$}
          -- ([yshift=0.2em]\tikzcdmatrixname-1-6.north)node[below right,xshift=2em]{$6$}
          -- ([yshift=0.2em]\tikzcdmatrixname-1-7.north)node[below right,xshift=2em]{$\overline{7}$}
          -- ([yshift=0.2em]\tikzcdmatrixname-1-8.north)
          }]{}
    & {\scalebox{0.4}{\ytableaushort{\none\none\none,\none,\none,\ }}} \rar \uar
    & {\scalebox{0.4}{\ytableaushort{\none\ \none,\none\ ,\none,\none}}} \rar \uar
    & {\scalebox{0.4}{\ytableaushort{\none\ \ ,\none\ ,\none,\none}}} \rar \uar
    & {\scalebox{0.4}{\ytableaushort{\none\ \ ,\none,\none,\ }}} \rar \uar
    & {\scalebox{0.4}{\ytableaushort{\none\ \ ,\none\ ,\none,\none}}} \rar \uar
    & {\scalebox{0.4}{\ytableaushort{\none\ \none,\none\ ,\none,\none}}} \rar \uar
    & {\scalebox{0.4}{\ytableaushort{\none\ \none,\none\ ,\none\ ,\none\ }}} \\
\end{tikzcd}.
\]
The dashed path represents the fluctuating tableau $\BK_3 \circ \BK_2 \circ \BK_1(T)$. We encode this path in the following tableau $\JDT_3(T)$ using the labels $\overline{2}, 3, 4, \bullet, \overline{5}, 6, \overline{7}$:
  \[ \usebox{\JDTc}. \]
Putting these calculations together, we compute $\promotion(T)$ as follows:
\[
\begin{tikzpicture}
\node (begin) at (0,0) {$T = $};
\node [right = 0.0cm of begin] (0)  {\scalebox{0.8}{\usebox{\runningT}}};
\node [right = 0.5cm of 0] (1) {\scalebox{0.8}{\usebox{\JDTa}}};
\node [right = 0.5cm of 1] (2) {\scalebox{0.8}{\usebox{\JDTa}}};
\node [right = 0.5cm of 2] (3) {\scalebox{0.8}{\usebox{\JDTb}}};
\node [right = 0.5cm of 3] (4) {\scalebox{0.8}{\usebox{\JDTc}}};
\node [below = 0.5cm of 4] (5) {\scalebox{0.8}{\usebox{\JDTd}}};
\node [left = 0.5cm of 5] (6) {\scalebox{0.8}{\usebox{\JDTe}}};
\node [left = 0.5cm of 6] (7) {\scalebox{0.8}{\usebox{\JDTf}}};
\node [left = 0.5cm of 7] (8) {\scalebox{0.8}{\usebox{\JDTg}}};
\node [left = 0.0cm of 8] (end) {$\promotion(T)=$};
\draw[pil] (0) -- (1) node[midway,above] {$\JDT_0$};
\draw[pil] (1) -- (2) node[midway,above] {$\JDT_1$};
\draw[pil] (2) -- (3) node[midway,above] {$\JDT_2$};
\draw[pil] (3) -- (4) node[midway,above] {$\JDT_3$};
\draw[pil] (4) -- (5) node[midway,right] {$\JDT_4$};
\draw[pil] (5) -- (6) node[midway,above] {$\JDT_5$};
\draw[pil] (6) -- (7) node[midway,above] {$\JDT_6$};
\draw[pil] (7) -- (8) node[midway,above] {$-1$};
\end{tikzpicture}.
\]
\end{example}

We may encode the changes in the positions of the $\bullet$'s or $\overline{\bullet}$'s using the following combinatorial notion, which computes $\JDT_{i-1}(T)$ from $\JDT_{i-2}(T)$ roughly by swapping $\bullet < i$ for $i < \bullet$. Recall the notion of open cells from \Cref{def:free_etc}.

\begin{definition}
  Suppose $T$ is a fluctuating tableau labeled by $2 < \cdots < i-1 < \bullet < i < \cdots < n$ and their negatives. The \textit{jeu de taquin slides} for $T$ are given by the following rules. See \Cref{fig:JDT-rules} for schematic diagrams.
  \begin{enumerate}[(a)]
  \item $\bullet$'s first move right to swap places with $i$'s, then move down as far as possible by swapping places with $i$'s.
  \item $\overline{\bullet}$'s first move left to swap places with $\overline{i}$'s, then move up as far as possible by swapping places with $\overline{i}$'s.
  \item $\bullet\overline{i}$ pairs move down as far as possible into open cells, and then they move left one column before becoming $\overline{i}\bullet$ pairs.
  \item $\overline{\bullet}i$ pairs move up as far as possible into open cells, and then move right one column before becoming $i\overline{\bullet}$ pairs.
  \end{enumerate}
  All other entries are left unchanged.
\end{definition}

\newsavebox{\JDTppaa}
\sbox{\JDTppaa}{%
  \begin{tikzpicture}[inner sep=0in,outer sep=0in]
    \node (n) {\begin{varwidth}{5cm}{
    \begin{ytableau}
      \bullet & i \\
    \end{ytableau}}\end{varwidth}};
    \draw[ultra thick,black] ([xshift=0*0.8cm]n.south west)--([xshift=0*0.8cm]n.north west);
  \end{tikzpicture}
}
\newsavebox{\JDTppab}
\sbox{\JDTppab}{%
  \begin{tikzpicture}[inner sep=0in,outer sep=0in]
    \node (n) {\begin{varwidth}{5cm}{
    \begin{ytableau}
       i & \bullet \\
    \end{ytableau}}\end{varwidth}};
    \draw[ultra thick,black] ([xshift=0*0.8cm]n.south west)--([xshift=0*0.8cm]n.north west);
  \end{tikzpicture}
}
\newsavebox{\JDTppba}
\sbox{\JDTppba}{%
  \begin{tikzpicture}[inner sep=0in,outer sep=0in]
    \node (n) {\begin{varwidth}{5cm}{
    \begin{ytableau}
       \none & \none \\
    \end{ytableau}}\end{varwidth}};
    \draw[ultra thick,black] ([xshift=0*0.8cm]n.south west)--([xshift=0*0.8cm]n.north west);
  \end{tikzpicture}
}
\newsavebox{\JDTppbb}
\sbox{\JDTppbb}{%
  \begin{tikzpicture}[inner sep=0in,outer sep=0in]
    \node (n) {\begin{varwidth}{5cm}{
    \begin{ytableau}
       \none & \none \\
    \end{ytableau}}\end{varwidth}};
    \draw[ultra thick,black] ([xshift=0*0.8cm]n.south west)--([xshift=0*0.8cm]n.north west);
  \end{tikzpicture}
}
\newsavebox{\JDTppca}
\sbox{\JDTppca}{%
  \begin{tikzpicture}[inner sep=0in,outer sep=0in]
    \node (n) {\begin{varwidth}{5cm}{
    \begin{ytableau}
       \bullet & \none \\
       \bullet & \none \\
       i & \none \\
       i & \none \\
       i & \none \\
    \end{ytableau}}\end{varwidth}};
    \draw[ultra thick,black] ([xshift=0*0.8cm]n.south west)--([xshift=0*0.8cm]n.north west)--([xshift=1*0.8cm]n.north west);
  \end{tikzpicture}
}
\newsavebox{\JDTppcb}
\sbox{\JDTppcb}{%
  \begin{tikzpicture}[inner sep=0in,outer sep=0in]
    \node (n) {\begin{varwidth}{5cm}{
    \begin{ytableau}
       i & \none \\
       i & \none \\
       i & \none \\
       \bullet & \none \\
       \bullet & \none \\
    \end{ytableau}}\end{varwidth}};
    \draw[ultra thick,black] ([xshift=0*0.8cm]n.south west)--([xshift=0*0.8cm]n.north west)--([xshift=1*0.8cm]n.north west);
  \end{tikzpicture}
}

\newsavebox{\JDTnnaa}
\sbox{\JDTnnaa}{%
  \begin{tikzpicture}[inner sep=0in,outer sep=0in]
    \node (n) {\begin{varwidth}{5cm}{
    \begin{ytableau}
      \overline{i} & \overline{\bullet} \\
    \end{ytableau}}\end{varwidth}};
    \draw[ultra thick,black] ([xshift=2*0.8cm]n.south west)--([xshift=2*0.8cm]n.north west);
  \end{tikzpicture}
}
\newsavebox{\JDTnnab}
\sbox{\JDTnnab}{%
  \begin{tikzpicture}[inner sep=0in,outer sep=0in]
    \node (n) {\begin{varwidth}{5cm}{
    \begin{ytableau}
       \overline{\bullet} & \overline{i} \\
    \end{ytableau}}\end{varwidth}};
    \draw[ultra thick,black] ([xshift=2*0.8cm]n.south west)--([xshift=2*0.8cm]n.north west);
  \end{tikzpicture}
}
\newsavebox{\JDTnnba}
\sbox{\JDTnnba}{%
  \begin{tikzpicture}[inner sep=0in,outer sep=0in]
    \node (n) {\begin{varwidth}{5cm}{
    \begin{ytableau}
       \none & \none \\
    \end{ytableau}}\end{varwidth}};
    \draw[ultra thick,black] ([xshift=0*0.8cm]n.south west)--([xshift=0*0.8cm]n.north west);
  \end{tikzpicture}
}
\newsavebox{\JDTnnbb}
\sbox{\JDTnnbb}{%
  \begin{tikzpicture}[inner sep=0in,outer sep=0in]
    \node (n) {\begin{varwidth}{5cm}{
    \begin{ytableau}
       \none & \none \\
    \end{ytableau}}\end{varwidth}};
    \draw[ultra thick,black] ([xshift=0*0.8cm]n.south west)--([xshift=0*0.8cm]n.north west);
  \end{tikzpicture}
}
\newsavebox{\JDTnnca}
\sbox{\JDTnnca}{%
  \begin{tikzpicture}[inner sep=0in,outer sep=0in]
    \node (n) {\begin{varwidth}{5cm}{
    \begin{ytableau}
       \overline{i} & \none \\
       \overline{i} & \none \\
       \overline{i} & \none \\
       \overline{\bullet} & \none \\
       \overline{\bullet} & \none \\
    \end{ytableau}}\end{varwidth}};
    \draw[ultra thick,black] ([xshift=0*0.8cm]n.south west)--([xshift=1*0.8cm]n.south west)--([xshift=1*0.8cm]n.north west);
  \end{tikzpicture}
}
\newsavebox{\JDTnncb}
\sbox{\JDTnncb}{%
  \begin{tikzpicture}[inner sep=0in,outer sep=0in]
    \node (n) {\begin{varwidth}{5cm}{
    \begin{ytableau}
       \overline{\bullet} & \none \\
       \overline{\bullet} & \none \\
       \overline{i} & \none \\
       \overline{i} & \none \\
       \overline{i} & \none \\
    \end{ytableau}}\end{varwidth}};
    \draw[ultra thick,black] ([xshift=0*0.8cm]n.south west)--([xshift=1*0.8cm]n.south west)--([xshift=1*0.8cm]n.north west);
  \end{tikzpicture}
}

\newsavebox{\JDTpnaa}
\sbox{\JDTpnaa}{%
  \begin{tikzpicture}[inner sep=0in,outer sep=0in]
    \node (n) {\begin{varwidth}{5cm}{
    \begin{ytableau}
      \bullet & \none \\
    \end{ytableau}}\end{varwidth}};
    \draw[ultra thick,black] ([xshift=0*0.8cm]n.south west)--([xshift=0*0.8cm]n.north west);
  \end{tikzpicture}
}
\newsavebox{\JDTpnab}
\sbox{\JDTpnab}{%
  \begin{tikzpicture}[inner sep=0in,outer sep=0in]
    \node (n) {\begin{varwidth}{5cm}{
    \begin{ytableau}
       \bullet & \none \\
    \end{ytableau}}\end{varwidth}};
    \draw[ultra thick,black] ([xshift=0*0.8cm]n.south west)--([xshift=0*0.8cm]n.north west);
  \end{tikzpicture}
}
\newsavebox{\JDTpnba}
\sbox{\JDTpnba}{%
  \begin{tikzpicture}[inner sep=0in,outer sep=0in]
    \node (n) {\begin{varwidth}{5cm}{
    \begin{ytableau}
       \overline{i} & \none \\
    \end{ytableau}}\end{varwidth}};
    \draw[ultra thick,black] ([xshift=1*0.8cm]n.south west)--([xshift=1*0.8cm]n.north west);
  \end{tikzpicture}
}
\newsavebox{\JDTpnbb}
\sbox{\JDTpnbb}{%
  \begin{tikzpicture}[inner sep=0in,outer sep=0in]
    \node (n) {\begin{varwidth}{5cm}{
    \begin{ytableau}
       \overline{i} & \none \\
    \end{ytableau}}\end{varwidth}};
    \draw[ultra thick,black] ([xshift=1*0.8cm]n.south west)--([xshift=1*0.8cm]n.north west);
  \end{tikzpicture}
}
\newsavebox{\JDTpnca}
\sbox{\JDTpnca}{%
  \begin{tikzpicture}[inner sep=0in,outer sep=0in]
    \node (n) {\begin{varwidth}{5cm}{
    \begin{ytableau}
       \  & \bullet\overline{i} \\
       \  & \bullet\overline{i} \\
       \  & \  \\
       \  & \  \\
       \  & \  \\
    \end{ytableau}}\end{varwidth}};
    \draw[ultra thick,black] ([xshift=1*0.8cm]n.south west)--([xshift=1*0.8cm]n.north west)--([xshift=2*0.8cm]n.north west);
  \end{tikzpicture}
}
\newsavebox{\JDTpncb}
\sbox{\JDTpncb}{%
  \begin{tikzpicture}[inner sep=0in,outer sep=0in]
    \node (n) {\begin{varwidth}{5cm}{
    \begin{ytableau}
       \  & \  \\
       \  & \  \\
       \  & \  \\
       \overline{i}\bullet & \  \\
       \overline{i}\bullet & \  \\
    \end{ytableau}}\end{varwidth}};
    \draw[ultra thick,black] ([xshift=1*0.8cm]n.south west)--([xshift=1*0.8cm]n.north west)--([xshift=2*0.8cm]n.north west);
  \end{tikzpicture}
}

\newsavebox{\JDTnpaa}
\sbox{\JDTnpaa}{%
  \begin{tikzpicture}[inner sep=0in,outer sep=0in]
    \node (n) {\begin{varwidth}{5cm}{
    \begin{ytableau}
      \overline{\bullet} & \none \\
    \end{ytableau}}\end{varwidth}};
    \draw[ultra thick,black] ([xshift=1*0.8cm]n.south west)--([xshift=1*0.8cm]n.north west);
  \end{tikzpicture}
}
\newsavebox{\JDTnpab}
\sbox{\JDTnpab}{%
  \begin{tikzpicture}[inner sep=0in,outer sep=0in]
    \node (n) {\begin{varwidth}{5cm}{
    \begin{ytableau}
       \overline{\bullet} & \none \\
    \end{ytableau}}\end{varwidth}};
    \draw[ultra thick,black] ([xshift=1*0.8cm]n.south west)--([xshift=1*0.8cm]n.north west);
  \end{tikzpicture}
}
\newsavebox{\JDTnpba}
\sbox{\JDTnpba}{%
  \begin{tikzpicture}[inner sep=0in,outer sep=0in]
    \node (n) {\begin{varwidth}{5cm}{
    \begin{ytableau}
       i & \none \\
    \end{ytableau}}\end{varwidth}};
    \draw[ultra thick,black] ([xshift=0*0.8cm]n.south west)--([xshift=0*0.8cm]n.north west);
  \end{tikzpicture}
}
\newsavebox{\JDTnpbb}
\sbox{\JDTnpbb}{%
  \begin{tikzpicture}[inner sep=0in,outer sep=0in]
    \node (n) {\begin{varwidth}{5cm}{
    \begin{ytableau}
       i & \none \\
    \end{ytableau}}\end{varwidth}};
    \draw[ultra thick,black] ([xshift=0*0.8cm]n.south west)--([xshift=0*0.8cm]n.north west);
  \end{tikzpicture}
}
\newsavebox{\JDTnpca}
\sbox{\JDTnpca}{%
  \begin{tikzpicture}[inner sep=0in,outer sep=0in]
    \node (n) {\begin{varwidth}{5cm}{
    \begin{ytableau}
       \  & \  \\
       \  & \  \\
       \  & \  \\
       \overline{\bullet}i & \  \\
       \overline{\bullet}i & \  \\
    \end{ytableau}}\end{varwidth}};
    \draw[ultra thick,black] ([xshift=0*0.8cm]n.south west)--([xshift=1*0.8cm]n.south west)--([xshift=1*0.8cm]n.north west);
  \end{tikzpicture}
}
\newsavebox{\JDTnpcb}
\sbox{\JDTnpcb}{%
  \begin{tikzpicture}[inner sep=0in,outer sep=0in]
    \node (n) {\begin{varwidth}{5cm}{
    \begin{ytableau}
       \  & i\overline{\bullet} \\
       \  & i\overline{\bullet} \\
       \  & \  \\
       \  & \  \\
       \  & \  \\
    \end{ytableau}}\end{varwidth}};
    \draw[ultra thick,black] ([xshift=0*0.8cm]n.south west)--([xshift=1*0.8cm]n.south west)--([xshift=1*0.8cm]n.north west);
  \end{tikzpicture}
}

\begin{figure}[ht]
\[
\begin{tikzpicture}
\node (begin) at (0,0) {$\bullet i: $};
\node [right = 0.0cm of begin] (aa)  {\scalebox{0.7}{\usebox{\JDTppaa}}};
\node [right = 0.5cm of aa] (ab)  {\scalebox{0.7}{\usebox{\JDTppab}}};
\draw[pil] (aa) -- (ab);
\node [below=of aa.west,anchor=north west] (ba)  {\scalebox{0.7}{\usebox{\JDTppba}}};
\node [below=of ab.west,anchor=north west] (bb)  {\scalebox{0.7}{\usebox{\JDTppbb}}};
\draw[pil] (ba) -- (bb);
\node [below=of ba.west,anchor=north west] (ca)  {\scalebox{0.7}{\usebox{\JDTppca}}};
\node [below=of bb.west,anchor=north west] (cb)  {\scalebox{0.7}{\usebox{\JDTppcb}}};
\draw[pil] (ca) -- (cb);
\end{tikzpicture}
\qquad
\begin{tikzpicture}
\node (begin) at (0,0) {$\overline{\bullet}\overline{i}: $};
\node [right = 0.0cm of begin] (aa)  {\scalebox{0.7}{\usebox{\JDTnnaa}}};
\node [right = 0.5cm of aa] (ab)  {\scalebox{0.7}{\usebox{\JDTnnab}}};
\draw[pil] (aa) -- (ab);
\node [below=of aa.west,anchor=north west] (ba)  {\scalebox{0.7}{\usebox{\JDTnnba}}};
\node [below=of ab.west,anchor=north west] (bb)  {\scalebox{0.7}{\usebox{\JDTnnbb}}};
\draw[pil] (ba) -- (bb);
\node [below=of ba.west,anchor=north west] (ca)  {\scalebox{0.7}{\usebox{\JDTnnca}}};
\node [below=of bb.west,anchor=north west] (cb)  {\scalebox{0.7}{\usebox{\JDTnncb}}};
\draw[pil] (ca) -- (cb);
\end{tikzpicture}
\]
\[
\begin{tikzpicture}
\node (begin) at (0,0) {$\bullet\overline{i}: $};
\node [right = 0.0cm of begin] (aa)  {\scalebox{0.7}{\usebox{\JDTpnaa}}};
\node [right = 0.5cm of aa] (ab)  {\scalebox{0.7}{\usebox{\JDTpnab}}};
\draw[pil] (aa) -- (ab);
\node [below=of aa.west,anchor=north west] (ba)  {\scalebox{0.7}{\usebox{\JDTpnba}}};
\node [below=of ab.west,anchor=north west] (bb)  {\scalebox{0.7}{\usebox{\JDTpnbb}}};
\draw[pil] (ba) -- (bb);
\node [below=of ba.west,anchor=north west] (ca)  {\scalebox{0.7}{\usebox{\JDTpnca}}};
\node [below=of bb.west,anchor=north west] (cb)  {\scalebox{0.7}{\usebox{\JDTpncb}}};
\draw[pil] (ca) -- (cb);
\end{tikzpicture}
\qquad
\begin{tikzpicture}
\node (begin) at (0,0) {$\overline{\bullet}i: $};
\node [right = 0.0cm of begin] (aa)  {\scalebox{0.7}{\usebox{\JDTnpaa}}};
\node [right = 0.5cm of aa] (ab)  {\scalebox{0.7}{\usebox{\JDTnpab}}};
\draw[pil] (aa) -- (ab);
\node [below=of aa.west,anchor=north west] (ba)  {\scalebox{0.7}{\usebox{\JDTnpba}}};
\node [below=of ab.west,anchor=north west] (bb)  {\scalebox{0.7}{\usebox{\JDTnpbb}}};
\draw[pil] (ba) -- (bb);
\node [below=of ba.west,anchor=north west] (ca)  {\scalebox{0.7}{\usebox{\JDTnpca}}};
\node [below=of bb.west,anchor=north west] (cb)  {\scalebox{0.7}{\usebox{\JDTnpcb}}};
\draw[pil] (ca) -- (cb);
\end{tikzpicture}
\]
\caption{Jeu de taquin slides for $\JDT_{i-1}$ on fluctuating tableaux. In each case, the thick line indicates the boundary of the shape immediately before the $\bullet$. One may interpret the $\bullet i$ rules as first sliding $\bullet$'s right one cell past $i$'s, and then sliding $\bullet$'s down past $i$'s. The $\bullet\overline{i}$ rules instead ``slide past'' open cells and move left. The open cells in this case are simply those directly below $\bullet\overline{i}$'s, immediately right of the thick line, and without $\bullet$ or $\overline{i}$.}\label{fig:JDT-rules}
\end{figure}

The following is straightforward to prove by comparing jeu de taquin slides, the tableau-theoretic description of Bender--Knuth involutions, and the properties of local rule diagrams.

\begin{lemma}
  For $i-1 \geq 1$, applying jeu de taquin slides to $\JDT_{i-2}(T)$ results in $\JDT_{i-1}(T)$.
\end{lemma}

\begin{lemma}
  Let $T$ be a length $n$ skew fluctuating tableau.
  \begin{enumerate}[(i)]
    \item $\promotion(T)$ is the result of replacing $\pm 1$'s with $\pm \bullet$'s, using jeu de taquin slides $\JDT_1, \dots, \JDT_{n-1}$,  replacing $\pm \bullet$'s with $\pm (n+1)$'s, and subtracting $1$ from each entry's absolute value.
    \item $\evacuation(T)$ is the result of first replacing $\pm 1$'s with $\pm \bullet_n$'s, sliding them past $\pm n$'s, replacing $\pm 2$'s with $\pm \bullet_{n-1}$'s, sliding them past $\pm n$'s, etc., and finally replacing $\pm \bullet_i$'s with $\pm i$'s.
    \item $\devacuation(T)$ is the result of first replacing $\pm n$'s with $\pm \bullet_1$'s, sliding backwards past $\pm 1$'s, replacing $\pm (n-1)$'s with $\pm \bullet_2$'s, sliding backwards past $\pm 1$'s, etc., and finally replacing $\pm \bullet_i$'s with $\pm i$'s.
  \end{enumerate}
\end{lemma}

\section{Promotion grids and promotion matrices}
\label{sec:prom_stuff}

Hopkins--Rubey \cite[\S4]{Hopkins-Rubey} attached certain decorations to promotion-evacuation diagrams of rectangular $3$-row standard tableaux, which may be encoded as a permutation. We now extend this approach to arbitrary fluctuating tableaux with any number of rows.

\subsection{Local rule grids}

Given a local rule diagram, we may encode it as in \Cref{sec:jdt} as an application of jeu de taquin involving $\bullet$'s or $\overline{\bullet}$'s and $i$'s or $\overline{i}$'s. We decorate the diagram as follows.

\begin{definition}
  The \emph{local rule grid} associated to an $r$-row fluctuating tableau local rule
  \begin{equation}\label{eq:prom_grid}
  \begin{tikzcd}
  \lambda \ar{r}[name=U]{d}
    & \nu \\
  \kappa \uar{c} \ar{r}[swap,name=D]{d}
    & \mu \ar{u}[swap]{c}
  \ar[to path={(U)   (D)}]{}
  \end{tikzcd}
  \end{equation}
  is a $|c| \times |d|$ grid $M$ of intervals in $\{1, \ldots, r-1\}$ defined as follows. First, encode the diagram as a jeu de taquin slide $S \to T$.
  \begin{itemize}
  \item If $c \geq 0$, number the $\bullet$'s in $S$ and $T$ from bottom to top as $\bullet_1, \ldots, \bullet_{c}$.
  \item If $c \leq 0$, number the $\overline{\bullet}$'s in $S$ and $T$ from top to bottom as $\overline{\bullet}_1, \ldots, \overline{\bullet}_{-c}$.
  \item If $d \geq 0$, number the $i$'s in $S$ and $T$ from top to bottom as $i_1, \ldots, i_d$.
  \item If $d \leq 0$, number the $\overline{i}$'s in $S$ and $T$ from bottom to top as $\overline{i}_1, \ldots, \overline{i}_{-d}$.
  \end{itemize}
  Index rows of $S$ and $T$ from $0$ to $r-1$ from the top down. Consider cases based on the signs of $c$ and $d$:
  \begin{itemize}
  \item ($0 \leq c,d$): If $\bullet_a$ swaps vertically with $i_b$, let $M_{ab} = \{j\}$ where $j$ is the row of $\bullet_a$ after the slide.
  \item ($c,d \leq 0$): If $\overline{\bullet}_a$ swaps vertically with $\overline{i}_b$, let $M_{ab} = \{j\}$ where $j$ is the row of $\overline{\bullet}_a$ before the slide.
  \item ($d \leq 0 \leq c$): If $\bullet_a \overline{i}_b$ slides from row $j$ in $S$ to row $k$ in $T$, let $M_{ab} = (j, k]$.
  \item ($c \leq 0 \leq d$): If $\overline{\bullet}_a i_b$ slides from row $k$ in $S$ to row $j$ in $T$, let $M_{ab} = (j, k]$.
  \end{itemize}
  We draw a local rule grid by writing the matrix $M$ in the center of the local rule diagram.
\end{definition}

\begin{example}
    Consider the following local rule diagram, its corresponding jeu de taquin diagram, and the resulting local rule grid:
\newsavebox{\PGaa}
\sbox{\PGaa}{%
  \begin{tikzpicture}[inner sep=0in,outer sep=0in]
    \node (n) {\begin{varwidth}{5cm}{
    \begin{ytableau}
       \none[0] & \none & *(lime)\bullet_2\overline{i}_3 \\
       \none[1] & \none & *(green)\  \\
       \none[2] & \none & *(green)\  \\
       \none[3] & \overline{i}_2 \\
       \none[4] & *(lime)\bullet_1\overline{i}_1 \\
       \none[5] & *(green)\  \\
    \end{ytableau}}\end{varwidth}};
    \draw[ultra thick,black]
      ([xshift=2*0.815cm]n.south west)
      --++(-1*0.815cm, 0)
      --++(0, 2*0.815cm)
      --++(1*0.815cm, 0)
      --++(0, 4*0.815cm);
  \end{tikzpicture}
}
\newsavebox{\PGab}
\sbox{\PGab}{%
  \begin{tikzpicture}[inner sep=0in,outer sep=0in]
    \node (n) {\begin{varwidth}{5cm}{
    \begin{ytableau}
       \none & *(green)\  \\
       \none & *(green)\  \\
       \none & *(lime)\overline{i}_3\bullet_2 \\
       \none & \overline{i}_2 \\
       *(green)\  \\
       *(lime)\overline{i}_1\bullet_1 \\
    \end{ytableau}}\end{varwidth}};
    \draw[ultra thick,black]
      ([xshift=2*0.815cm]n.south west)
      --++(-1*0.815cm, 0)
      --++(0, 2*0.815cm)
      --++(1*0.815cm, 0)
      --++(0, 4*0.815cm);
  \end{tikzpicture}
}
\[
\begin{tikzpicture}
\node (begin) at (0,0)
  {
  \begin{tikzcd}
  10000\overline{1} \ar{r}[name=U]{\overline{3}}
    & 000\overline{1}\overline{1}\overline{1} \\
  0000\overline{1}\overline{1} \uar{2} \ar{r}[swap,name=D]{\overline{3}}
    & 00\overline{1}\overline{1}\overline{1}\overline{2} \ar{u}[swap]{2}
  \ar[to path={(U)  (D)}]{}
  \end{tikzcd}
  };
\node [right = 0.5cm of begin] (a)  {\scalebox{0.7}{\usebox{\PGaa}}};
\node [right = 0.5cm of a] (b)  {\scalebox{0.7}{\usebox{\PGab}}};
\draw[pil] (a) -- (b);
\node [right = 0.5cm of b] (c)
  {$M =
  \begin{pmatrix}
    \{5\} & \varnothing & \varnothing \\
    \varnothing & \varnothing & \{1, 2\} \\
  \end{pmatrix}
  $};
\end{tikzpicture}.
\]
Here, the moving cells have been shaded in light green \textcolor{lime}{$\blacksquare$} and the open cells have been shaded in darker green \textcolor{green}{$\blacksquare$}. The pair $\bullet_2 \overline{i}_3$ moved from row $0$ to row $2$, so $M_{23} = (0, 2] = \{1, 2\}$.

We will generally draw this local rule grid as:
\[
\begin{tikzcd}
  10000\overline{1} \ar{r}[name=U]{\overline{3}}
    & 000\overline{1}\overline{1}\overline{1} \\
  0000\overline{1}\overline{1} \uar{2} \ar{r}[swap,name=D]{\overline{3}}
    & 00\overline{1}\overline{1}\overline{1}\overline{2} \ar{u}[swap]{2}
  \ar[to path={(U) node[midway] {\scalebox{0.7}{$\begin{matrix}
      5 \amsamp \cdot \amsamp \cdot  \\
     \cdot \amsamp \cdot \amsamp \{1,2\} 
  \end{matrix}$}}  (D)}]{}
  \end{tikzcd}
\]
drawing $M$ in the center of its local rule diagram, dropping brackets on singleton sets, and shrinking empty sets to dots.
\end{example}

\begin{definition}\label{def:triangle}
  The \emph{triangular grids} $U$ and $L$ associated to the $r$-row fluctuating tableau diagrams
  \begin{equation}\label{eq:prom_grid_MN}
  \begin{tikzcd}
  \varnothing \ar{r}[name=U]{c}
    & \nu \\
  \ \ar[phantom]{r} \ar[phantom]{r}[swap,name=D]{}
    & \varnothing \ar{u}[swap]{c}
  \ar[to path={(U)   (D)}]{}
  \end{tikzcd}
  \qquad\text{and}\qquad
  \begin{tikzcd}
  \varnothing \ar[phantom]{r}[name=U]{}
    & \  \\
  \kappa \uar{c} \ar{r}[swap,name=D]{c}
    & \varnothing
  \ar[to path={(U)   (D)}]{}
  \end{tikzcd}
  \end{equation}
  are defined as follows. The triangular grid $U$ is an upper triangular grid with $c$ rows and $c$ columns whose entries are elements of $\{0, \ldots, r\}$. Conversely, $L$ is a lower triangular grid that also has $c$ rows and $c$ columns and entries that are elements of $\{0, \ldots, r\}$. 

  If $c =4$, we have 
  \[
U = \begin{matrix}
0 & 1 & 2 & 3  \\
& 0 & 1 & 2  \\
& & 0 & 1  \\
& & & 0  \\
\end{matrix} \quad \text{and} \quad L = \begin{matrix}
 r& & &  \\
 r-1 & r & & \\
 r-2 & r-1 & r & \\
 r-3 & r-2 & r-1 & r \\
\end{matrix}
  \]
  while in general, for $c \geq 0$, we have $U_{ij} = j-i$ whenever $1 \leq i \leq j \leq c$, and $L_{ij} = r -i +j$ whenever $1 \leq j \leq i \leq c$. 
    If $c =-4$, we have 
  \[
U = \begin{matrix}
r & r-1 & r-2 & r-3  \\
& r & r-1 & r-2  \\
& & r & r-1  \\
& & & r  \\
\end{matrix}
\quad \text{and} \quad
L = \begin{matrix}
    0 & & & \\
    1 & 0 &  & \\
    2 & 1 & 0 &  \\
    3 & 2 & 1 & 0
\end{matrix}
  \]
  while in general, for $c \leq 0$, we have $U_{ij} = r-j+i$ whenever $1\leq i \leq j\leq -c$, and $L_{ij} = i -j$ whenever $1 \leq j \leq i \leq -c$.
\end{definition}

\begin{lemma}\label{lem:M.infer}
  Given $\kappa \too{c} \lambda$ and $M$ in \eqref{eq:prom_grid}, we can infer the rest of the local rule diagram. Similarly, the upper triangular grid $U$ in \eqref{eq:prom_grid_MN} uniquely determines $c$ and $\nu$, while the lower triangular grid $L$ in \eqref{eq:prom_grid_MN} uniquely determines $c$ and $\kappa$.
\end{lemma}
\begin{proof}
  For the first claim, we know the location of $\bullet_a$'s or $\overline{\bullet}_a$'s in $S$. If $\bullet_a$ or $\overline{\bullet}_a$ does not move, then $M_{ab} = \varnothing$ for all $b$. If $\bullet_a$ or $\overline{\bullet}_a$ does move, we may infer where it moves by examining $M_{ab}$ for all $b$. Hence we know the locations of the $\pm \bullet$'s in $S$ and $T$, and thereby the locations of the $\pm i$'s in $S$ and $T$. Thus we may infer $\nu$, which lets us infer $\mu$. The other claims are similar but easier.
\end{proof}

\subsection{Promotion grids}

We now glue together some local rule grids and triangular grids to form a single block grid.

\begin{definition}
  Let $\diagram$ be a growth diagram for an $r$-row non-skew fluctuating tableau $T$, for example $\Pdiagram(T)$, $\Ediagram(T)$, $\Eddiagram(T)$, or $\PEdiagram(T)$. The \emph{$\diagram$-grid} of $\diagram$ is the block grid $\PM_{\diagram}$ whose blocks are the local rule grids of $\diagram$ and appropriately-sized triangular grids.
\end{definition}

This definition can be extended to handle skew fluctuating tableaux; however, the definitions of the triangular grids needed become significantly more complicated and we do not currently have an application of the extra generality such a definition would give, so we omit the details here.

\begin{notation}
       We use the abbreviations
    $\PMP(T) = \PM_{\Pdiagram(T)}$,
    $\PME(T) = \PM_{\Ediagram(T)}$,
    $\PMEd(T) = \PM_{\Eddiagram(T)}$, and
    $\PMPE(T) = \PM_{\PEdiagram(T)}$. We will also use the names \emph{promotion grid}, \emph{evacuation grid}, etc., with the obvious meanings.
\end{notation}

\begin{example} \label{ex:prom_grid_PE}
  For the fluctuating tableau $T$ from \Cref{fig:ft-example}, we write the blocks of $\PMPE(T)$ in $\PEdiagram(T)$:
  \[
  \setlength\arraycolsep{2pt}
  \begin{tikzcd}[ampersand replacement=\&,column sep=tiny, scale cd=0.7]
    0000 \rar[""{name=L00r}]
      \& 1100 \rar[""{name=L01r}]
      \& 110\overline{1} \rar[""{name=L02r}]
      \& 2110 \rar[""{name=L03r}]
      \& 2210 \rar[""{name=L04r}]
      \& 2100 \rar[""{name=L05r}]
      \& 2111 \rar[""{name=L06r}]
      \& 1111 \rar[phantom,""{name=L07r}]
      \& \ \\
    \ \rar[phantom,""{name=L10r}]
      \& 0000 \rar[""{name=L11r}] \uar
      \& 000\overline{1} \rar[""{name=L12r}] \uar
      \& 1100 \rar[""{name=L13r}] \uar
      \& 2100 \rar[""{name=L14r}] \uar
      \& 200\overline{1} \rar[""{name=L15r}] \uar
      \& 2100 \rar[""{name=L16r}] \uar
      \& 1100 \rar[""{name=L17r}] \uar
      \& 1111 \rar[phantom,""{name=L18r}]
      \& \ \\
    \ 
      \& \ \rar[phantom,""{name=L21r}]
      \& 0000 \rar[""{name=L22r}] \uar
      \& 1110 \rar[""{name=L23r}] \uar
      \& 2110 \rar[""{name=L24r}] \uar
      \& 210\overline{1} \rar[""{name=L25r}] \uar
      \& 2200 \rar[""{name=L26r}] \uar
      \& 2100 \rar[""{name=L27r}] \uar
      \& 2111 \rar[""{name=L28r}] \uar
      \& 1111 \rar[phantom,""{name=L29r}]
      \& \ \\
    \ 
      \& \ 
      \& \ \rar[phantom,""{name=L32r}]
      \& 0000 \rar[""{name=L33r}] \uar
      \& 1000 \rar[""{name=L34r}] \uar
      \& 10\overline{1}\overline{1} \rar[""{name=L35r}] \uar
      \& 110\overline{1} \rar[""{name=L36r}] \uar
      \& 100\overline{1} \rar[""{name=L37r}] \uar
      \& 1100 \rar[""{name=L38r}] \uar
      \& 1000 \rar[""{name=L39r}] \uar
      \& 1111 \rar[phantom,""{name=L3Ar}]
      \& \ \\
    \ 
      \& \ 
      \& \ 
      \& \ \rar[phantom,""{name=L43r}]
      \& 0000 \rar[""{name=L44r}] \uar
      \& 00\overline{1}\overline{1} \rar[""{name=L45r}] \uar
      \& 100\overline{1} \rar[""{name=L46r}] \uar
      \& 10\overline{1}\overline{1} \rar[""{name=L47r}] \uar
      \& 110\overline{1} \rar[""{name=L48r}] \uar
      \& 100\overline{1} \rar[""{name=L49r}] \uar
      \& 1110 \rar[""{name=L4Ar}] \uar
      \& 1111 \rar[phantom,""{name=L4Br}]
      \& \ \\
    \ 
      \& \ 
      \& \ 
      \& \ 
      \& \ \rar[phantom,""{name=L54r}]
      \& 0000 \rar[""{name=L55r}] \uar
      \& 1100 \rar[""{name=L56r}] \uar
      \& 110\overline{1} \rar[""{name=L57r}] \uar
      \& 211\overline{1} \rar[""{name=L58r}] \uar
      \& 210\overline{1} \rar[""{name=L59r}] \uar
      \& 2210 \rar[""{name=L5Ar}] \uar
      \& 2211 \rar[""{name=L5Br}] \uar
      \& 1111 \rar[phantom,""{name=L5Cr}]
      \& \ \\
    \ 
      \& \ 
      \& \ 
      \& \ 
      \& \ 
      \& \ \rar[phantom,""{name=L65r}]
      \& 0000 \rar[""{name=L66r}] \uar
      \& 000\overline{1} \rar[""{name=L67r}] \uar
      \& 110\overline{1} \rar[""{name=L68r}] \uar
      \& 11\overline{1}\overline{1} \rar[""{name=L69r}] \uar
      \& 2100 \rar[""{name=L6Ar}] \uar
      \& 2110 \rar[""{name=L6Br}] \uar
      \& 1100 \rar[""{name=L6Cr}] \uar
      \& 1111 \rar[phantom,""{name=L6Dr}]
      \& \ \\
    \ 
      \& \ 
      \& \ 
      \& \ 
      \& \ 
      \& \ 
      \& \ \rar[phantom,""{name=L76r}]
      \& 0000 \rar[""{name=L77r}] \uar
      \& 1100 \rar[""{name=L78r}] \uar
      \& 110\overline{1} \rar[""{name=L79r}] \uar
      \& 2110 \rar[""{name=L7Ar}] \uar
      \& 2210 \rar[""{name=L7Br}] \uar
      \& 2100 \rar[""{name=L7Cr}] \uar
      \& 2111 \rar[""{name=L7Dr}] \uar
      \& 1111.
      \arrow[phantom,from=L00r,to=L10r,"{\begin{matrix}0 & 1 \\ \ & 0\end{matrix}}"]
      \arrow[phantom,from=L01r,to=L11r,"{\begin{matrix}\cdot \\ \cdot\end{matrix}}"]
      \arrow[phantom,from=L02r,to=L12r,"{\begin{matrix}\cdot & 2 & \cdot \\ \cdot & \cdot & \cdot\end{matrix}}"]
      \arrow[phantom,from=L03r,to=L13r,"{\begin{matrix}\cdot \\ 1\end{matrix}}"]
      \arrow[phantom,from=L04r,to=L14r,"{\begin{matrix}3 & \cdot \\ \cdot & \cdot\end{matrix}}"]
      \arrow[phantom,from=L05r,to=L15r,"{\begin{matrix}\cdot & \cdot \\ 2 & \cdot\end{matrix}}"]
      \arrow[phantom,from=L06r,to=L16r,"{\begin{matrix}\cdot \\ \cdot\end{matrix}}"]
      \arrow[phantom,from=L07r,to=L17r,"{\begin{matrix}4 & \ \\ 3 & 4\end{matrix}}"]
      \arrow[phantom,from=L11r,to=L21r,"{\begin{matrix}4\end{matrix}}"]
      \arrow[phantom,from=L12r,to=L22r,"{\begin{matrix}\cdot & \cdot & 3\end{matrix}}"]
      \arrow[phantom,from=L13r,to=L23r,"{\begin{matrix}\cdot\end{matrix}}"]
      \arrow[phantom,from=L14r,to=L24r,"{\begin{matrix}\cdot & 2\end{matrix}}"]
      \arrow[phantom,from=L15r,to=L25r,"{\begin{matrix}\cdot & \cdot\end{matrix}}"]
      \arrow[phantom,from=L16r,to=L26r,"{\begin{matrix}1\end{matrix}}"]
      \arrow[phantom,from=L17r,to=L27r,"{\begin{matrix}\cdot & \cdot\end{matrix}}"]
      \arrow[phantom,from=L18r,to=L28r,"{\begin{matrix}0\end{matrix}}"]
      \arrow[phantom,from=L22r,to=L32r,"{\begin{matrix}0 & 1 & 2 \\ \ & 0 & 1 \\ \ & \ & 0\end{matrix}}"]
      \arrow[phantom,from=L23r,to=L33r,"{\begin{matrix}\cdot \\ \cdot \\ \cdot\end{matrix}}"]
      \arrow[phantom,from=L24r,to=L34r,"{\begin{matrix}\cdot & \cdot \\ \cdot & \cdot \\ \cdot & \cdot\end{matrix}}"]
      \arrow[phantom,from=L25r,to=L35r,"{\begin{matrix}\cdot & 3 \\ \cdot & \cdot \\ \cdot & \cdot\end{matrix}}"]
      \arrow[phantom,from=L26r,to=L36r,"{\begin{matrix}\cdot \\ \cdot \\ \cdot\end{matrix}}"]
      \arrow[phantom,from=L27r,to=L37r,"{\begin{matrix}\cdot & \cdot \\ 2 & \cdot \\ \cdot & \cdot\end{matrix}}"]
      \arrow[phantom,from=L28r,to=L38r,"{\begin{matrix}\cdot \\ \cdot \\ 1\end{matrix}}"]
      \arrow[phantom,from=L29r,to=L39r,"{\begin{matrix}4 \\ 3 & 4 \\ 2 & 3 & 4\end{matrix}}"]
      \arrow[phantom,from=L33r,to=L43r,"{\begin{matrix}0\end{matrix}}"]
      \arrow[phantom,from=L34r,to=L44r,"{\begin{matrix}\cdot & \cdot\end{matrix}}"]
      \arrow[phantom,from=L35r,to=L45r,"{\begin{matrix}1 & \cdot\end{matrix}}"]
      \arrow[phantom,from=L36r,to=L46r,"{\begin{matrix}2\end{matrix}}"]
      \arrow[phantom,from=L37r,to=L47r,"{\begin{matrix}\cdot & 3\end{matrix}}"]
      \arrow[phantom,from=L38r,to=L48r,"{\begin{matrix}\cdot\end{matrix}}"]
      \arrow[phantom,from=L39r,to=L49r,"{\begin{matrix}\cdot & \cdot & \cdot\end{matrix}}"]
      \arrow[phantom,from=L3Ar,to=L4Ar,"{\begin{matrix}4\end{matrix}}"]
      \arrow[phantom,from=L44r,to=L54r,"{\begin{matrix}4 & 3 \\ \ & 4\end{matrix}}"]
      \arrow[phantom,from=L45r,to=L55r,"{\begin{matrix}\cdot & 2 \\ \cdot & \cdot\end{matrix}}"]
      \arrow[phantom,from=L46r,to=L56r,"{\begin{matrix}\cdot \\ 3\end{matrix}}"]
      \arrow[phantom,from=L47r,to=L57r,"{\begin{matrix}1 & \cdot \\ \cdot & \cdot\end{matrix}}"]
      \arrow[phantom,from=L48r,to=L58r,"{\begin{matrix}\cdot \\ 2\end{matrix}}"]
      \arrow[phantom,from=L49r,to=L59r,"{\begin{matrix}\cdot & \cdot & \cdot \\ \cdot & \cdot & \cdot\end{matrix}}"]
      \arrow[phantom,from=L4Ar,to=L5Ar,"{\begin{matrix}\cdot \\ \cdot\end{matrix}}"]
      \arrow[phantom,from=L4Br,to=L5Br,"{\begin{matrix}0 \\ 1 & 0\end{matrix}}"]
      \arrow[phantom,from=L55r,to=L65r,"{\begin{matrix}0 & 1 \\ \ & 0\end{matrix}}"]
      \arrow[phantom,from=L56r,to=L66r,"{\begin{matrix}\cdot \\ \cdot\end{matrix}}"]
      \arrow[phantom,from=L57r,to=L67r,"{\begin{matrix}\cdot & 2 \\ \cdot & \cdot\end{matrix}}"]
      \arrow[phantom,from=L58r,to=L68r,"{\begin{matrix}\cdot \\ \cdot\end{matrix}}"]
      \arrow[phantom,from=L59r,to=L69r,"{\begin{matrix}\cdot & \cdot & \cdot \\ 1 & \cdot & \cdot\end{matrix}}"]
      \arrow[phantom,from=L5Ar,to=L6Ar,"{\begin{matrix}3 \\ \cdot\end{matrix}}"]
      \arrow[phantom,from=L5Br,to=L6Br,"{\begin{matrix}\cdot & \cdot \\ 2 & \cdot\end{matrix}}"]
      \arrow[phantom,from=L5Cr,to=L6Cr,"{\begin{matrix}4 \\ 3 & 4\end{matrix}}"]
      \arrow[phantom,from=L66r,to=L76r,"{\begin{matrix}4\end{matrix}}"]
      \arrow[phantom,from=L67r,to=L77r,"{\begin{matrix}\cdot & \cdot\end{matrix}}"]
      \arrow[phantom,from=L68r,to=L78r,"{\begin{matrix}3\end{matrix}}"]
      \arrow[phantom,from=L69r,to=L79r,"{\begin{matrix}\cdot & \cdot & \cdot\end{matrix}}"]
      \arrow[phantom,from=L6Ar,to=L7Ar,"{\begin{matrix}2\end{matrix}}"]
      \arrow[phantom,from=L6Br,to=L7Br,"{\begin{matrix}\cdot & 1\end{matrix}}"]
      \arrow[phantom,from=L6Cr,to=L7Cr,"{\begin{matrix}\cdot & \cdot\end{matrix}}"]
      \arrow[phantom,from=L6Dr,to=L7Dr,"{\begin{matrix}0\end{matrix}}"]
  \end{tikzcd}
  \]
\end{example}

The lemma below follows directly from the definitions.

\begin{lemma}\label{lem:M_prom}
  Consider reading a row of $\PMPE(T)$ either from left to right if the label on the first arrow is positive  or from right to left if the label on the first arrow is negative, skipping empty entries. The result is a sequence of intervals $[0, i_1], (i_1, i_2], \ldots, (i_k, r]$ where $0 \leq i_1 < i_2 < \cdots < i_k \leq r-1$.
\end{lemma}

Since in \Cref{ex:prom_grid_PE}, all nonempty intervals are single elements, the above lemma says all nonempty entries in a row of the promotion-evacuation grid are sequentially $0,1,2,3,4$ or $4,3,2,1,0$, according to whether boxes are added or removed in the first step.

Evacuation grids may be used as alternate encodings of fluctuating tableaux as follows. We will shortly use a similar encoding in the rectangular case, with particularly desirable properties.

\begin{theorem}\label{thm:injective}
  The map $T \mapsto (\PME(T), \type(T))$ is injective.
\end{theorem}
\begin{proof}
  We may infer the block sizes of $\PME(T)$ from the type of $T$.  We may fill $\Ediagram(T)$ diagonal by diagonal by \Cref{lem:M.infer}. Finally, we may read off $T$ from the top row of $\Ediagram(T)$.
\end{proof}

Note that we may infer $\type(T)$ from $\PME(T)$ if we track the blocks of $\PME(T)$. In this sense, we may entirely encode fluctuating tableaux of a fixed shape in terms of the block grid $\PME(T)$.

We now consider the effect on a $\diagram$-grid of oscillizing the growth diagram $\diagram$.  
\begin{lemma}\label{lem:osc_prom}
  If $\diagram$ is a skew fluctuating growth diagram, then
    \[ \PM_{\diagram} = \PM_{\std(\diagram)}. \]
\end{lemma}
\begin{proof}
    One may check that the $\diagram$-grids in \eqref{eq:prom_grid} have been defined so that the $\diagram$-grid agrees with the $\diagram$-grid of its oscillization. It is easy to see that the same is true for \eqref{eq:prom_grid_MN}.
\end{proof}

We now describe the effect of various diagram involutions on $\diagram$-grids. We will need these observations to determine the order of promotion on rectangular fluctuating tableaux.

\begin{definition}
  Given any grid $M$, let $M^\top$ be the \emph{transpose} of $M$ and let $M^\bot$ be the \emph{anti-diagonal transpose} of $M$. 
\end{definition}

Note that the composite of the transpose and anti-diagonal transpose is rotation by $180^\circ$.  Recall that by convention $\tau$ and $\varepsilon$ act on growth diagrams by applying $\tau$ and $\varepsilon$ to each edge and rotating the result by $180^\circ$.

\begin{lemma}\label{lem:other_lemma}
  Suppose $\diagram$ is an $r$-row  fluctuating growth diagram. Then:
  \begin{enumerate}[(i)]
    \item $\PM_{\diagram^\bot} = \PM_{\diagram}^\bot$,
    \item $\PM_{\tau(\diagram)} = \PM_{\diagram}^{\top\bot}$,
    \item $\PM_{\varpi(\diagram)} = r - \PM_{\diagram}$,
    \item $\PM_{\varepsilon(\diagram)} = r - \PM_{\diagram}^{\top\bot}$.
  \end{enumerate}
  Here for a matrix $\mathbf{N}$ we set $(r-\mathbf{N})_{ab} \coloneqq \{r-i : i \in \mathbf{N}_{ab}\}$.
\end{lemma}
\begin{proof}
  For (i), the roles of $\bullet$, $i$ and $S, T$ are interchanged between the two diagrams. The result then follows by examining the symmetry in the definition of the promotion grid.
  
  For (ii), after applying $\tau$ to the local rule, we rotate the diagram $180^\circ$. If we encode the original diagram in jeu de taquin with $S \to T$, the new diagram encoded via jeu de taquin as $S' \to T'$ where $S'$ is the same as $T$ and $T'$ is the same as $S$ after the replacements $\overline{\bullet}_a \leftrightarrow \bullet_{c+1-a}$ and $\overline{i}_b \leftrightarrow i_{d+1-b}$. The indices $\{j\}$ and $(j, k]$ are preserved. Hence $\tau(M)_{ab} = M_{c+1-a, d+1-b} = (M^{\top\bot})_{ab}$.
  
  For (iii), applying $\varpi$ to a local rule diagram has the effect of rotating each individual generalized partition $180^\circ$. Correspondingly, $S''$ is the same as $S$ and $T''$ is the same as $T$ but rotated $180^\circ$ and with the replacements $\bullet_a \leftrightarrow \overline{\bullet}_a$, $i_b \leftrightarrow \overline{i}_b$. Moreover, row indices are reversed according to $j \leftrightarrow r-j$.
  
  For (iv), we compose $\tau$ and $\varpi$.
\end{proof}

The lemma below follows straightforwardly from \Cref{lem:other_lemma}.
\begin{lemma}\label{lem:PM.invs}
  On $r$-row skew fluctuating tableaux, we have the following:
  \begin{enumerate}[(i)]
    \item $\PME(\evacuation(T)) = \PME(T)^\bot$, $\PMEd(\devacuation(T)) = \PMEd(T)^\bot$
    \item 
    \[\begin{aligned}[c]
    \PME(\tau(T))
      &= \PMEd(T)^{\top \bot} \\
    \PME(\varpi(T))
      &= r - \PME(T) \\
    \PME(\varepsilon(T))
      &= r -  \PMEd(T)^{\top \bot} \\
    \end{aligned}
    \qquad
    \begin{aligned}[c]
    \PMEd(\tau(T))
      &= \PME(T)^{\top \bot} \\
    \PMEd(\varpi(T))
      &= r - \PMEd(T) \\
    \PMEd(\varepsilon(T))
      &= r - \PME(T)^{\top \bot} \\
    \end{aligned}\]
    \item $\PMPE(\varpi(T)) = r - \PMPE(T)$
  \end{enumerate}
\end{lemma}

\subsection{Promotion matrices}
Next we encode $\PMPE(T)$ in a square block matrix.

\begin{definition}\label{def:PM}
  The \textit{promotion matrix} of an $r$-row fluctuating tableau $T$ of length $n$ is the $n \times n$ block matrix $\PM(T)$ whose upper triangle is $\PME(T)$ and whose lower triangle is $\PMEd(\devacuation\evacuation(T))$. Entries on the main diagonal are the union of the entries from the main diagonals of the two triangles, reduced modulo $r$ (and hence are always $0$).
\end{definition}

We visualize $\PM(T)$ by ``wrapping around'' $\PEdiagram(T)$ to form a square matrix:

\[
  \begin{tikzcd}[column sep=small, row sep=small]
    \ 
      & \ 
      & \lambda^{00} \ar[""{name=L00r}]{rr}
      & \ 
      & \lambda^{01} \ar[""{name=L01r}]{rr}
      & \ 
      & \cdots \ar[""{name=L0nl}]{rr}
      & \ 
      & \lambda^{0n} \rar[equals,-]
      & \  \\
    \ \rar[equals,-]
      & \lambda^{00} \ar[phantom,""{name=L0nr}]{rr} \ar[dashed,-]{ur}
      & \ 
      & \ 
      & \ 
      & \ 
      & \ 
      & \ 
      & \ 
      & \  \\
    \ 
      & \ 
      & \  \ar[phantom,""{name=L11l}]{rr}
      & \ 
      & \lambda^{11} \ar[""{name=L11r}]{rr} \ar{uu}
      & \ 
      & \cdots \ar[""{name=L1nl}]{rr} \ar{uu}
      & \ 
      & \lambda^{1n} \ar{uu} \rar[equals,-]
      & \ \\
    \ \rar[equals,-]
      & \lambda^{10} \ar[""{name=L1nr}]{rr} \ar{uu}
      & \ 
      & \lambda^{11} \ar[phantom,""{name=L1n1r}]{rr} \ar[dashed,-]{ur}
      & \ 
      & \ 
      & \ 
      & \ 
      & \ 
      & \ \\
    \ 
      & \ 
      & \ 
      & \ 
      & \  \ar[phantom,""{name=Lddotsl}]{rr}
      & \ 
      & \ddots \ar[""{name=Lddotsr}]{rr} \ar{uu}
      & \ 
      & \vdots \ar{uu} \rar[equals,-]
      & \ \\
    \ \rar[equals,-]
      & \vdots \ar{uu} \ar[""{name=Lvdots1r}]{rr}
      & \ 
      & \vdots \ar{uu} \ar[""{name=Lvdots2r}]{rr}
      & \ 
      & \ddots \ar[phantom,""{name=Lddotsr2}]{rr} \ar[dashed,-]{ur}
      & \ 
      & \ 
      & \ 
      & \ \\
    \ 
      & \ 
      & \ 
      & \ 
      & \ 
      & \ 
      & \ \ar[phantom,""{name=Lnnl}]{rr}
      & \ 
      & \lambda^{nn} \ar{uu} \rar[equals,-]
      & \ \\
    \ \rar[equals,-]
      & \lambda^{n0} \ar[""{name=Lnnr}]{rr} \ar{uu}
      & \ 
      & \lambda^{n1} \ar[""{name=Lnn1r}]{rr} \ar{uu}
      & \ 
      & \cdots \ar[""{name=Lcdotsr}]{rr} \ar{uu}
      & \ 
      & \lambda^{nn} \ar[dashed,-]{ur}
      & \ 
      & \ \\
    \arrow[phantom,from=L00r,to=L11l,"U^{1}"]
    \arrow[phantom,from=L01r,to=L11r,"M^{12}"]
    \arrow[phantom,from=L0nl,to=L1nl,"M^{1n}"]
    \arrow[phantom,from=Lddotsl,to=L11r,"U^{2}"]
    \arrow[phantom,from=L1nl,to=Lddotsr,"M^{2n}"]
    \arrow[phantom,from=Lddotsr,to=Lnnl,"U^{n}"]
    \arrow[phantom,from=L0nr,to=L1nr,"L^{1}"]
    \arrow[phantom,from=L1nr,to=Lvdots1r,"M^{21}"]
    \arrow[phantom,from=L1n1r,to=Lvdots2r,"L^{2}"]
    \arrow[phantom,from=Lvdots1r,to=Lnnr,"M^{n1}"]
    \arrow[phantom,from=Lvdots2r,to=Lnn1r,"M^{n2}"]
    \arrow[phantom,from=Lddotsr2,to=Lcdotsr,"L^{n}"]
  \end{tikzcd}.
\]
Here, $U^{i}$ is upper triangular and $L^{i}$ is lower triangular, as in \Cref{def:triangle}, and they are combined in $\PM(T)$. We will refer to the square grid constructed from $U^i$ and $L^i$ as $M^{ii}$ for consistency with the indexing of the other blocks of $\PM(T)$. Note that the entries in $\PMPE(T)$ cyclically increment or decrement when reading across any given row in the sense of \Cref{lem:M_prom}.

\begin{example} \label{ex:prom_mx}
  For the fluctuating tableau $T$ in \Cref{fig:ft-example}, $\type(T) = (2, -1, 3, 1, -2, 2, -1)$ and the promotion matrix is:
  \[
   \arraycolsep=3.2pt
  \PM(T) = 
  \left(
  \begin{array}{cc|c|ccc|c|cc|cc|c}
    0 & 1 & \cdot & \cdot & 2 & \cdot & \cdot & 3 & \cdot & \cdot & \cdot & \cdot \\
    3 & 0 & \cdot & \cdot & \cdot & \cdot & 1 & \cdot & \cdot & 2 & \cdot & \cdot \\
    \hline
    \cdot & \cdot & 0 & \cdot & \cdot & 3 & \cdot & \cdot & 2 & \cdot &\cdot & 1 \\
    \hline
    \cdot & \cdot & \cdot & 0 & 1 & 2 & \cdot & \cdot & \cdot & \cdot & 3 & \cdot \\
    2 & \cdot & \cdot & 3 & 0 & 1 & \cdot & \cdot & \cdot & \cdot & \cdot & \cdot \\
    \cdot & \cdot & 1 & 2 & 3 & 0 & \cdot & \cdot & \cdot & \cdot & \cdot & \cdot \\
    \hline
    \cdot & 3 & \cdot & \cdot & \cdot & \cdot & 0 & \cdot & \cdot & 1 & \cdot & 2 \\
    \hline
    1 & \cdot & \cdot & \cdot & \cdot & \cdot & \cdot & 0 & 3 & \cdot & 2 & \cdot \\
    \cdot & \cdot & 2 & \cdot & \cdot & \cdot & \cdot & 1 & 0 & \cdot & \cdot & 3 \\
    \hline
    \cdot & 2 & \cdot&  \cdot & \cdot & \cdot & 3 & \cdot & \cdot & 0 & 1 & \cdot \\
    \cdot & \cdot & \cdot & 1 & \cdot & \cdot & \cdot & 2 & \cdot & 3 & 0 & \cdot \\
    \hline
    \cdot & \cdot & 3 & \cdot & \cdot & \cdot & 2 & \cdot & 1 & \cdot & \cdot & 0
  \end{array}
  \right).
  \]
\end{example}

By \Cref{lem:PM.invs}(iii), we have
\begin{equation}\label{eq:PM.pi}
  \PM(\varpi(T)) = r - \PM(T).
\end{equation}
The effects of the other fundamental involutions in general are more complex. They can however be easily described in the case that evacuation and dual evacuation coincide.

\begin{lemma}\label{lem:evac.devac}
  The following are equivalent on $r$-row  fluctuating tableaux:
  \begin{enumerate}[(i)]
    \item $\evacuation(T) = \devacuation(T)$,
    \item $\PM(\evacuation(T)) = \PM(T)^\bot$,
    \item $\PM(\devacuation(T)) = \PM(T)^\bot$,
    \item $\PM(\tau(T)) = \PM(T)^{\top \bot}$,
    \item $\PM(\varepsilon(T)) = r - \PM(T)^{\top \bot}$.
  \end{enumerate}
\end{lemma}

\begin{proof} 
    By \eqref{eq:pediagram.sum} and \Cref{thm:injective}, we have (i) if and only if
  \begin{equation}\label{eq:evac.devac.1}
    \PMPE(T) = \PME(T) \mathbin\Vert \PMEd(T),
  \end{equation}
  where $\mathbin\Vert$ denotes concatenating the grids and identifying the column where equality holds.
 By \Cref{lem:invs.PEEd}, we may replace $T$ in \eqref{eq:evac.devac.1} with $\tau(T)$ or $\varepsilon(T)$ while remaining equivalent to (i). Using \Cref{lem:PM.invs}, we obtain
  \begin{align*}
    \PME(\tau(T)) \mathbin\Vert \PMEd(\tau(T))
      &= \PMEd(T)^{\top \bot} \mathbin\Vert \PME(T)^{\top \bot} \\
      &= (\PME(T) \mathbin\Vert \PMEd(T))^{\top \bot} \\
      &= \PM(T)^{\top \bot}.
  \end{align*}
  In this way, (i), (iv), and (v) are equivalent. The argument for (ii) and (iii) is essentially identical.
\end{proof}

\begin{lemma}\label{lem:E.eps}
  The following are equivalent on $r$-row skew fluctuating tableaux:
  \begin{enumerate}[(i)]
    \item $\evacuation(T) = \varepsilon(T)$,
    \item $\devacuation(T) = \varepsilon(T)$,
    \item $\PM(T) = r - \PM(T)^\top$.
  \end{enumerate}
\end{lemma}

\begin{proof}
  We have seen the equivalence of (i) and (ii) in \Cref{lem:E.Ed.epsilon}. For the rest,   by \eqref{eq:pediagram.sum} we have (iii) if and only if
    \[ \PMEd(\devacuation \circ \evacuation(T)) = r - \PME(T)^\top. \]
 Take the anti-diagonal transpose of both sides. Now
    \[ \PMEd(\devacuation \circ \evacuation(T))^\bot = \PMEd(\evacuation(T)) \]
  by \Cref{lem:PM.invs}, and likewise
    \[ (r - \PME(T)^\top)^\bot = r - \PME(T)^{\top \bot} = \PMEd(\varepsilon(T)). \]
  Note that,  by \Cref{thm:injective}, $\PMEd(\evacuation(T)) = \PMEd(\varepsilon(T))$ is equivalent to (i), completing the proof.
\end{proof}

\subsection{Reduced promotion matrices}\label{sec:reduced_M}

We now describe an alternate encoding of the promotion matrices $\PM(T)$. This encoding directly corresponds to a crystal-theoretic interpretation given in \Cref{sec:crystals_promotion}. 

\begin{definition}\label{def:PMr}
  The \textit{reduced promotion matrices} of an $r$-row  fluctuating tableau $T$ of length $n$ are the $n \times n$ matrices $\PMr^i(T)$ (for $0 \leq i \leq r-1$) with entries in $\mathbb{Z}_{\geq 0}$ defined by
    \[ \PMr^i(T)_{uv} = \#\{\text{entries in the block $M^{uv}$ of $\PM(T)$ containing $i$}\}. \]
\end{definition}

We have the following direct characterization of $\PMr(T)^i_{uv}$ off the main diagonal.

\begin{proposition}\label{prop:PMr_sum}
  Suppose we have a local rule diagram:
  \begin{center}
  \begin{tikzcd}
  \lambda \ar{r}[name=U]{}
    & \nu \\
  \kappa \uar{} \ar{r}[swap,name=D]{}
    & \mu \ar{u}[swap]{}
  \ar[to path={(U) node[midway] {$\PMr^i$}  (D)}]{}
  \end{tikzcd}
  \end{center}
  where $\lambda = \kappa + \mathbf{e}_A$, $\nu = \mu + \mathbf{e}_B$,  and $\PMr^i$ are the reduced promotion matrix entries. Then the $\PMr^i$ are uniquely characterized by
  \begin{equation}\label{eq:PMr_sum}
    \mathbf{e}_B = \mathbf{e}_A + \sum_{i=1}^{r-1} \PMr^i (\mathbf{e}_{i+1} - \mathbf{e}_i).
  \end{equation}
\end{proposition}

\begin{proof}
  Consider the jeu de taquin slides associated with this local rule diagram. Recall that the rows of fluctuating tableaux are indexed top-down from $0$ to $r-1$. Note that all elements of $A$ and $B$ have the same sign. If $A, B$ consist of positive numbers, the $\bullet$'s begin in the rows indexed by $\{a -1 : a \in A \}$ and end in the rows indexed by $\{b - 1 : b \in B\}$. If $A, B$ consist of negative numbers, the $\overline{\bullet}$'s begin in the rows indexed by $\{-a -1 : a \in A \}$ and end in the rows indexed by $\{-b -1 : b \in B\}$. In the former case, $\PMr^i$ counts the number of times a $\bullet$ slides from row $i-1$ to row $i$, which effectively removes $\mathbf{e}_i$ from $\mathbf{e}_A$ and adds $\mathbf{e}_{i+1}$ to $\mathbf{e}_B$. In the latter case, $\PMr^i$ counts the number of times a $\overline{\bullet}$ slides from row $i$ to row $i-1$, which effectively removes $\mathbf{e}_{\overline{i+1}} = -\mathbf{e}_{i+1}$ from $\mathbf{e}_A$ and adds $\mathbf{e}_{\overline{i}} = -\mathbf{e}_i$ to $\mathbf{e}_B$. The net effect is the same in either case, and \eqref{eq:PMr_sum} follows by summing. Uniqueness follows from the fact that the $\mathbf{e}_{i+1} - \mathbf{e}_i$ are linearly independent.
\end{proof}

The reduced promotion matrix analogue of \Cref{thm:injective} holds.

\begin{theorem}\label{thm:injective.reduced}
  The map $T \mapsto (\PMr^1(T), \ldots, \PMr^{r-1}(T), \type(T))$ is injective.
\end{theorem}

\begin{proof}
  It suffices to show that \Cref{lem:M.infer} remains true given only the number of entries of each type in the promotion grids rather than the full promotion grids. We know the positions of the $\pm \bullet$'s and $\pm i$'s in $S$ as well as the number of times some $\pm \bullet$ slides past each particular row. We may hence infer the final row of the bottommost $\bullet$ or topmost $\overline{\bullet}$, and iteratively we may infer all final locations of $\pm \bullet$'s. The arguments for the other diagrams are similar.
\end{proof}

\begin{example}\label{ex:PMr}
  For the fluctuating tableau $T$ from \Cref{fig:ft-example}, we have:
  \[
  \begin{array}{cc}
    \PMr^0(T) 
  =
  \begin{pmatrix}
    2 & 0 & 0 & 0 & 0 & 0 & 0 \\
    0 & 1 & 0 & 0 & 0 & 0 & 0 \\
    0 & 0 & 3 & 0 & 0 & 0 & 0 \\
    0 & 0 & 0 & 1 & 0 & 0 & 0 \\
    0 & 0 & 0 & 0 & 2 & 0 & 0 \\
    0 & 0 & 0 & 0 & 0 & 2 & 0 \\
    0 & 0 & 0 & 0 & 0 & 0 & 1 \\
  \end{pmatrix}
  &
  \PMr^1(T) =
  \begin{pmatrix}
    1 & 0 & 0 & 1 & 0 & 0 & 0 \\
    0 & 0 & 0 & 0 & 0 & 0 & 1 \\
    0 & 1 & 2 & 0 & 0 & 0 & 0 \\
    0 & 0 & 0 & 0 & 0 & 1 & 0 \\
    1 & 0 & 0 & 0 & 1 & 0 & 0 \\
    0 & 0 & 1 & 0 & 0 & 1 & 0 \\
    0 & 0 & 0 & 0 & 1 & 0 & 0 \\
  \end{pmatrix} \\[0.8in]
  \PMr^2(T) =
  \begin{pmatrix}
    0 & 0 & 1 & 0 & 0 & 1 & 0 \\
    0 & 0 & 0 & 0 & 1 & 0 & 0 \\
    1 & 0 & 2 & 0 & 0 & 0 & 0 \\
    0 & 0 & 0 & 0 & 0 & 0 & 1 \\
    0 & 1 & 0 & 0 & 0 & 1 & 0 \\
    1 & 0 & 0 & 0 & 1 & 0 & 0 \\
    0 & 0 & 0 & 1 & 0 & 0 & 0 \\
  \end{pmatrix}
  &
  \PMr^3(T) =
  \begin{pmatrix}
    1 & 0 & 0 & 0 & 1 & 0 & 0 \\
    0 & 0 & 1 & 0 & 0 & 0 & 0 \\
    0 & 0 & 2 & 0 & 0 & 1 & 0 \\
    1 & 0 & 0 & 0 & 0 & 0 & 0 \\
    0 & 0 & 0 & 0 & 1 & 0 & 1 \\
    0 & 0 & 0 & 1 & 0 & 1 & 0 \\
    0 & 1 & 0 & 0 & 0 & 0 & 0 \\
  \end{pmatrix} \\
  \end{array}.
  \]
  These matrices should be compared to the matrix $\PM(T)$ from \Cref{ex:prom_mx}.
\end{example}

\section{Promotion permutations}\label{sec:promperm}

We begin by transforming promotion matrices into promotion functions. We then show they have particularly nice properties in the rectangular case and prove our main result, \Cref{thm:prom_perms}.

\begin{definition}\label{def:prom_fcn}
  The \emph{promotion functions} of an $r$-row fluctuating tableau $T$ of type $\underline{c}$ are partial functions defined as follows. Let $t = \sum_i |c_i|$, so that $\PM(T)$ is $t \times t$. For all $0 \leq i \leq r - 1$, each row of $\PM(T)$ has at most one entry containing $i$. Set
  \begin{align*}
    \prom_i(T) \colon [t] &\rightharpoonup [t], \\
    \prom_i(T)(a) &= b \qquad\text{whenever }i \in \PM(T)_{ab}.
  \end{align*}
  For symmetry, we define $\prom_r(T) \coloneqq \prom_0(T)$.
\end{definition}

\begin{example}\label{ex:prom_fncts}
  For the fluctuating tableau $T$ from \Cref{fig:ft-example} with promotion matrix found in \Cref{ex:prom_mx}, all the promotion functions are permutations. Specifically, we have
  \begin{align*}
    \prom_0(T) &= \prom_4(T) = \id \\
    \prom_1(T) &= (1\ 2\ 7\ 10\ 11\ 4\ 5\ 6\ 3\ 12\ 9\ 8) \\
    \prom_2(T) &= (1\ 5)(2\ 10)(3\ 9)(4\ 6)(7\ 12)(8\ 11) \\
    \prom_3(T) &= (1\ 8\ 9\ 12\ 3\ 6\ 5\ 4\ 11\ 10\ 7\ 2). \\
  \end{align*}
\end{example}

\begin{remark}\label{rem:prom.conditions}
  Note that $\prom_0(T) = \id = \prom_r(T)$ are always trivial; they are included for overall consistency. Since the diagonal of $\PM(T)$ is all zeros, all other $\prom_i(T)$ are fixed-point free.  Note that $\PM(T)$ can be recovered from the collection $(\prom_i(T))_{i=1}^{r-1}$ of all promotion functions. We will shortly see that, in the rectangular case, each $\prom_i(T)$ is a permutation. In the rectangular case, we will therefore refer to promotion functions as \emph{promotion permutations}.
\end{remark}

We have the following relation between the promotion functions of $T$ and the promotion functions of $\promotion(T)$.

\begin{lemma}\label{lem:prom.P}
  We have
    \[ \prom_i(T)(u+|c_1|) = v+|c_1| \qquad\Longleftrightarrow\qquad \prom_i(\promotion(T))(u) = v \]
  for all $1 \leq u, v \leq t-|c_1|$ and all $0 \leq i \leq r$.
\end{lemma}
\begin{proof}
By construction, the promotion-evacuation diagram of $\promotion(T)$ is obtained from that of $T$ by cutting off the top row. Hence, the promotion matrix is obtained by cutting away the top $c_1$ rows and leftmost $c_1$ columns. The lemma follows.    
\end{proof}

It is well-known that the $n$th power of promotion on rectangular standard tableaux with $n$ cells is the identity. Moreover, the effect of evacuation on such rectangular standard tableaux is the \textit{reverse-complement} (see, e.g., \cite{Haiman} for both of these facts), which in our terminology using lattice words is phrased as $L(\evacuation(T)) = \varepsilon(L(T))$. These properties extend to rectangular fluctuating tableaux; for an example, see \Cref{fig:full_orbit}.

\begin{theorem}\label{thm:ft.prom_evac}
  Let $T$ be a rectangular fluctuating tableau of length $n$. Then
  \begin{enumerate}[(a)]
    \item $\promotion^n(T) = T$,
    \item $L(\evacuation(T)) = \varepsilon(L(T))$.
  \end{enumerate}
\end{theorem}
\begin{proof}
  In terms of tableaux, (b) is equivalent to $\evacuation(T) = \varepsilon(T)$.
  By \Cref{lem:E.Ed.epsilon}, (b) implies (a).
  
  As for (b), we may use the $\toggle_i$ involutions and \Cref{lem:BK.toggle.PEEd} to reduce to the case when $T$ is (transpose) semi-standard. In this case, (b) is a well-known consequence of the RSK algorithm; see, e.g., \cite[Lemma~3.1]{Patrias-Pechenik} for details. An alternate crystal-theoretic proof is given in \Cref{sec:crystals-involutions}.
\end{proof}

We obtain the following as a corollary.
\begin{corollary}\label{cor:516_517}
  All of the equivalent conditions in \Cref{lem:evac.devac} and \Cref{lem:E.eps} hold for rectangular fluctuating tableaux.
\end{corollary}

The following is the main result of the present work. Let $\sigma = (1\,2\,\cdots\,t)$ be the long cycle, and let $w_0 = (1, t)(2, t-1)\cdots$ be the longest element in the symmetric group $\mathfrak{S}_t$.

\begin{theorem}\label{thm:prom_perms}
  Let $T$ be an $r$-row rectangular fluctuating tableau of type $(c_1, \ldots, c_n)$ where $|c_1| + \cdots + |c_n| = t$. Then for all $0 \leq i \leq r$:
  \begin{enumerate}[(i)]
    \item $\prom_i(T)$ is a permutation,
    \item $\prom_i(T) = \prom_{r-i}(T)^{-1}$,
    \item $\prom_i(\promotion(T)) = \sigma^{-|c_1|} \prom_i(T) \sigma^{|c_1|}$,
    \item $\prom_i(\evacuation(T)) = w_0 \prom_{r-i}(T) w_0$.
  \end{enumerate}
\end{theorem}
\begin{proof}
By \Cref{cor:516_517}, we have all the equivalent conditions of \Cref{lem:evac.devac} and \Cref{lem:E.eps}.
  \Cref{lem:E.eps}(iii) implies (i) and (ii). \Cref{lem:prom.P} implies (iii). Finally, (iv) follows from \Cref{lem:evac.devac}(v) and the fact that, for any matrix $M$, we have $M^{\top\bot} = w_0 M w_0$, where $w_0$ is viewed as a permutation matrix.
\end{proof}

\begin{corollary}\label{cor:FPF}
 If $T$ is a rectangular fluctuating tableau with $r$ rows and $1 \leq i \leq r-1$, then $\prom_i(T)$ is a fixed-point free permutation.
  Moreover if $r$ is even, then $\prom_{r/2}(T)$ is a fixed-point free involution.
\end{corollary}
\begin{proof}
    This is an immediate consequence of (i) and (ii) in the above theorem and \Cref{rem:prom.conditions}.
\end{proof}

We now describe how to obtain promotion permutations directly from a tableau without reference to promotion matrices. 

\begin{proposition}\label{prop:promi_alt_def}
    Let $T \in \ft(r,n)$ be a  fluctuating tableau and $1\leq i\leq r-1$. Then $\prom_i(T)(b) \equiv |a|+b-1 \pmod n$ if and only if $a$ is the unique value that crosses the boundary between rows $i$ and $i+1$ in the application of jeu de taquin promotion to  $\promotion^{b-1}(\std(T))$.
\end{proposition}
\begin{proof}
    First, suppose $T$ is of oscillating type.
    Consider repeatedly performing promotion on $T$ via jeu de taquin with the convention that we do not relabel entries after promotion. Then $\PMPE(T)$ records an $i$ in row $a$, column $b$, when letter $a$ or $\overline{a}$ crosses the boundary between rows $i$ and $i+1$ during the $b$th application of promotion to $T$. Since promotion actually decreases the absolute value of all entries by $1$, the result follows in this case.

    If $T$ is not of oscillating type, we may instead consider $\osc(T)$. By \Cref{lem:osc_prom}, $\PMPE(T) = \PMPE(\osc(T))$, so the theorem follows from the previous case.
\end{proof}

We end this section by describing properties of promotion permutations that determine the entries of the corresponding tableau. We will use the following lemma.
\begin{lemma}\label{lem:promi_osc}
Given a fluctuating tableau $T$ of $r$ rows and $0\leq i\leq r$,
$\prom_i(T)=\prom_i(\std(T))$.
\end{lemma}
\begin{proof}
\Cref{lem:osc_prom} directly implies that $\PM(T)=\PM(\std(T))$, so the conclusion follows.
\end{proof}

\begin{definition}
    Let $\pi \in \mathfrak{S}_n$ be a permutation. A number $i \in [n]$ is an \emph{antiexcedance} of $\pi$ if $\pi^{-1}(i) > i$. We write $\Aexc(\pi)$ for the set of antiexcedances of $\pi$. We say $i \in [n]$ is an \emph{excedance} of $\pi$ if it is neither an antiexcedance nor a fixed point.
\end{definition} 

\begin{theorem}
\label{thm:antiexcedance}
Let $T \in \ft(r)$ be a rectangular fluctuating tableau. Suppose $1 \leq i \leq r-1$.

Then $a$ is an antiexcedance of $\prom_i(T)$ if and only if either $a$ appears in the top $i$ rows of $\osc(T)$ or $\overline{a}$ appears in the bottom $r-i$ rows of $\osc(T)$. In particular, if $T$ is a standard tableau, 
then the antiexcedances of $\prom_i(T)$ are exactly the numbers in the first $i$ rows of $T$.
\end{theorem}
\begin{proof}
By \Cref{lem:promi_osc}, we may assume $T = \osc(T)$.
We start by characterizing the excedances of $\prom_i(T)$ instead of the antiexcedances. Since $\prom_i(T)$ is fixed-point free by \Cref{cor:FPF}, the antiexcedances are the complementary set to the excedances.

If we write a permutation $\pi \in \mathfrak{S}_n$ as a permutation matrix by placing $1$ in each position $(a,\pi(a))$, then the excedances of $\pi$ are exactly the values $\pi(a)$ such that $(a,\pi(a))$ appears strictly above the main diagonal.

The upper triangle of $\PM(T)$ is $\PME(T)$. Consider repeatedly performing promotion on $T$ via jeu de taquin with the convention that we do not relabel entries after promotion and we do not replace $\bullet$ or $\overline{\bullet}$ by a new number at the end of promotion. Then $\PME(T)$ records an $i$ in row $a$, column $b$, when letter $a$ or $\overline{a}$ crosses the boundary between rows $i$ and $i+1$ during the $b$th application of promotion to $T$. Note that during jeu de taquin, unbarred entries only move weakly up, while barred entries only move weakly down. But the entry $a$ must eventually exit out the top if it exists. Otherwise, $\overline{a}$ exists and must exit out the bottom. Hence, $\PME(T)$ will have an $i$ in row $a$ if any only if $a$ appears strictly below row $i$ in $T$ or if $\overline{a}$ appears weakly above row $i$ in $T$. Thus, $a$ is an excedance of $\prom_i(T)$ if and only if $a$ fails the hypotheses of the theorem. Thus, $a$ is an antiexcedance of $\prom_i(T)$ if and only if either $a$ appears in the top $i$ rows of $T$ or $\overline{a}$ appears in the bottom $r-i$ rows of $T$.
\end{proof}

\begin{corollary}\label{cor:uniquely_determined}
    Suppose $T\in\ft(r,\underline{c})$ is rectangular. Then $T$ is uniquely determined by its type together with the promotion permutations $(\prom_i(T))_{i=1}^{\lfloor r/2 \rfloor}$.
\end{corollary}
\begin{proof}
    This follows by combining \Cref{thm:prom_perms} with \Cref{thm:antiexcedance} after oscillizing the tableau.
\end{proof}

\def\x{1cm}
\ytableausetup{boxsize=0.98cm}
\newsavebox{\proz}
\sbox{\proz}{%
  \begin{tikzpicture}[inner sep=0in,outer sep=0in]
    \node (n) {\begin{varwidth}{5cm}{
    \begin{ytableau}
      \none & *(light-gray)1 & 4\,\overline{12} \\
      \none & *(light-gray)2 & 7\,\overline{9} \\
      \none & *(light-gray)5\,\overline{8}\,10 \\
      {\overline{3} \, 6}   & *(light-gray)11 \\
    \end{ytableau}}\end{varwidth}};
    \draw[ultra thick,black] ([xshift=\x]n.south west)--([xshift=\x]n.north west);
  \end{tikzpicture}
}

\newsavebox{\proa}
\sbox{\proa}{%
  \begin{tikzpicture}[inner sep=0in,outer sep=0in]
    \node (n) {\begin{varwidth}{5cm}{
    \begin{ytableau}
      \none & *(light-gray)1 & 3\,\overline{11} \\
      \none & *(light-gray)4 & 6\,\overline{8} \\
      \none & *(light-gray)9 \\
      {\overline{2} \, 5 \, \overline{7} \atop  10}   & *(light-gray)12 \\
    \end{ytableau}}\end{varwidth}};
    \draw[ultra thick,black] ([xshift=\x]n.south west)--([xshift=\x]n.north west);
  \end{tikzpicture}
}

\newsavebox{\prob}
\sbox{\prob}{%
  \begin{tikzpicture}[inner sep=0in,outer sep=0in]
    \node (n) {\begin{varwidth}{5cm}{
    \begin{ytableau}
      \none & *(light-gray)2 & 5\,\overline{10} \\
      \none & *(light-gray)3 \, \overline{7}\,8  \\
      \none & *(light-gray)11 \\
     { \overline{1}\,4\,\overline{6} \atop 9} & *(light-gray)12 \\
    \end{ytableau}}\end{varwidth}};
    \draw[ultra thick,black] ([xshift=\x]n.south west)--([xshift=\x]n.north west);
  \end{tikzpicture}
}

\newsavebox{\proc}
\sbox{\proc}{%
  \begin{tikzpicture}[inner sep=0in,outer sep=0in]
    \node (n) {\begin{varwidth}{5cm}{
    \begin{ytableau}
      \none & *(light-gray)1 & 4\,\overline{12} \\
      \none & *(light-gray)2 & 7\,\overline{9} \\
      \none & *(light-gray)3\,\overline{6}\,10 \\
      \overline{5}\,8 & *(light-gray)11 \\
    \end{ytableau}}\end{varwidth}};
    \draw[ultra thick,black] ([xshift=\x]n.south west)--([xshift=\x]n.north west);
  \end{tikzpicture}
}

\newsavebox{\prodd}
\sbox{\prodd}{%
  \begin{tikzpicture}[inner sep=0in,outer sep=0in]
    \node (n) {\begin{varwidth}{5cm}{
    \begin{ytableau}
      \none & *(light-gray)1 & 3\,\overline{11} \\
      \none & *(light-gray)2 & 6\,\overline{8} \\
      \overline{5}\,7 & *(light-gray)11 \\
      \overline{4}\,10 & *(light-gray)12 \\
    \end{ytableau}}\end{varwidth}};
    \draw[ultra thick,black] ([xshift=\x]n.south west)--([xshift=\x]n.north west);
  \end{tikzpicture}
}

\newsavebox{\proe}
\sbox{\proe}{%
  \begin{tikzpicture}[inner sep=0in,outer sep=0in]
    \node (n) {\begin{varwidth}{5cm}{
    \begin{ytableau}
      \none & *(light-gray)1 & 2\,\overline{10} \\
      \none & *(light-gray)5\,\overline{7}\,8 \\
      \overline{4}\,6 & *(light-gray)11 \\
      \overline{3}\,9 & *(light-gray)12 \\
    \end{ytableau}}\end{varwidth}};
    \draw[ultra thick,black] ([xshift=\x]n.south west)--([xshift=\x]n.north west);
  \end{tikzpicture}
}

\newsavebox{\prof}
\sbox{\prof}{%
  \begin{tikzpicture}[inner sep=0in,outer sep=0in]
    \node (n) {\begin{varwidth}{5cm}{
    \begin{ytableau}
      \none & *(light-gray)1 & \none \\
      \none & *(light-gray){4\,\overline{6}\,7 \atop \overline{9}\,10}  \\
      \overline{3}\,5 & *(light-gray)11 \\
      \overline{2}\,8 & *(light-gray)12 \\
    \end{ytableau}}\end{varwidth}};
    \draw[ultra thick,black] ([xshift=\x]n.south west)--([xshift=\x]n.north west);
  \end{tikzpicture}
}

\newsavebox{\prog}
\sbox{\prog}{%
  \begin{tikzpicture}[inner sep=0in,outer sep=0in]
    \node (n) {\begin{varwidth}{5cm}{
    \begin{ytableau}
      \none & *(light-gray)3 & \none \\
      \none & *(light-gray)6 \, \overline{8} \, 9 \\
      {\overline{2}\,4 \atop \overline{5}\,7} & *(light-gray)10 \\
      \overline{1}\,11 & *(light-gray)12 \\
    \end{ytableau}}\end{varwidth}};
    \draw[ultra thick,black] ([xshift=\x]n.south west)--([xshift=\x]n.north west);
  \end{tikzpicture}
}

\newsavebox{\proh}
\sbox{\proh}{%
  \begin{tikzpicture}[inner sep=0in,outer sep=0in]
    \node (n) {\begin{varwidth}{5cm}{
    \begin{ytableau}
      \none & *(light-gray)2 & 5\,\overline{12} \\
      \none & *(light-gray)3\,\overline{7}\,8 \\
      \overline{4}\,6 & *(light-gray)9 \\
      \overline{1}\,10 & *(light-gray)11 \\
    \end{ytableau}}\end{varwidth}};
    \draw[ultra thick,black] ([xshift=\x]n.south west)--([xshift=\x]n.north west);
  \end{tikzpicture}
}

\newsavebox{\proi}
\sbox{\proi}{%
  \begin{tikzpicture}[inner sep=0in,outer sep=0in]
    \node (n) {\begin{varwidth}{5cm}{
    \begin{ytableau}
      \none & *(light-gray)1 & 4\,\overline{12} \\
      \none & *(light-gray)2 & 7\,\overline{11} \\
      \none & *(light-gray)5\,\overline{6}\,8 \\
      \overline{3}\,9 & *(light-gray)10 \\
    \end{ytableau}}\end{varwidth}};
    \draw[ultra thick,black] ([xshift=\x]n.south west)--([xshift=\x]n.north west);
  \end{tikzpicture}
}

\newsavebox{\proj}
\sbox{\proj}{%
  \begin{tikzpicture}[inner sep=0in,outer sep=0in]
    \node (n) {\begin{varwidth}{5cm}{
    \begin{ytableau}
      \none & *(light-gray)1 & 3\,\overline{11} \\
      \none & *(light-gray)4 & 6\,\overline{8} \\
      \none & *(light-gray)9 \\
      {\overline{2}5\overline{7} \atop 10} & *(light-gray)12 \\
    \end{ytableau}}\end{varwidth}};
    \draw[ultra thick,black] ([xshift=\x]n.south west)--([xshift=\x]n.north west);
  \end{tikzpicture}
}

\newsavebox{\prok}
\sbox{\prok}{%
  \begin{tikzpicture}[inner sep=0in,outer sep=0in]
    \node (n) {\begin{varwidth}{5cm}{
    \begin{ytableau}
      \none & *(light-gray)2 & 5\,\overline{10} \\
      \none & *(light-gray)3  \\
      \overline{4}\,6 & *(light-gray)8\,\overline{9}\,11 \\
      \overline{1}\,7 & *(light-gray)12 \\
    \end{ytableau}}\end{varwidth}};
    \draw[ultra thick,black] ([xshift=\x]n.south west)--([xshift=\x]n.north west);
  \end{tikzpicture}
}

\newsavebox{\prol}
\sbox{\prol}{%
  \begin{tikzpicture}[inner sep=0in,outer sep=0in]
    \node (n) {\begin{varwidth}{5cm}{
    \begin{ytableau}
      \none & *(light-gray)1 & 4\,\overline{12} \\
      \none & *(light-gray)2 & 7\,\overline{9} \\
      \none & *(light-gray)5\,\overline{8}\,10 \\
      \overline{3}\,6 & *(light-gray)11 \\
    \end{ytableau}}\end{varwidth}};
    \draw[ultra thick,black] ([xshift=\x]n.south west)--([xshift=\x]n.north west);
  \end{tikzpicture}
}

\begin{figure}[htbp]
\[
\begin{tikzpicture}
\node (begin) at (0,0) {$S = $};
\node [right = -0.65cm of begin] (0)  {\scalebox{0.8}{\usebox{\proz}}};
\node [right = 0.5cm of 0] (1) {\scalebox{0.8}{\usebox{\proa}}};
\node [right = 0.5cm of 1] (2) {\scalebox{0.8}{\usebox{\prob}}};
\node [right = 0.5cm of 2] (3) {\scalebox{0.8}{\usebox{\proc}}};
\node [below = 0.5cm of 3] (4) {\scalebox{0.8}{\usebox{\prodd}}};
\node [below = 0.5cm of 2] (5) {\scalebox{0.8}{\usebox{\proe}}};
\node [below = 0.5cm of 1] (6) {\scalebox{0.8}{\usebox{\prof}}};
\node [below = 0.5cm of 0] (7) {\scalebox{0.8}{\usebox{\prog}}};
\node [below = 0.5cm of 7] (8) {\scalebox{0.8}{\usebox{\proh}}};
\node [below = 0.5cm of 6] (9) {\scalebox{0.8}{\usebox{\proi}}};
\node [below = 0.5cm of 5] (10) {\scalebox{0.8}{\usebox{\proj}}};
\node [below = 0.5cm of 4] (11) {\scalebox{0.8}{\usebox{\prok}}};
\node [right = -0.65cm of 11] (end) {$=\promotion^{-1}(S)$};
\draw[pil] (0) -- (1) node[midway,above] {$\promotion$};
\draw[pil] (1) -- (2) node[midway,above] {$\promotion$};
\draw[pil] (2) -- (3) node[midway,above] {$\promotion$};
\draw[pil] (3) -- (4) node[midway,right] {$\promotion$};
\draw[pil] (4) -- (5) node[midway,above] {$\promotion$};
\draw[pil] (5) -- (6) node[midway,above] {$\promotion$};
\draw[pil] (6) -- (7) node[midway,above] {$\promotion$};
\draw[pil] (7) -- (8) node[midway,left] {$\promotion$};
\draw[pil] (8) -- (9) node[midway,above] {$\promotion$};
\draw[pil] (9) -- (10) node[midway,above] {$\promotion$};
\draw[pil] (10) -- (11) node[midway,above] {$\promotion$};
\end{tikzpicture}
\]
\caption{The full promotion orbit of $S \coloneqq \std(T)$, where $T$ is the fluctuating tableau from \Cref{fig:ft-example}.}\label{fig:full_orbit}
\end{figure}
\ytableausetup{boxsize=.85cm}

\begin{example}\label{ex:prom_orbit}
For $T$ the fluctuating tableau from \Cref{fig:ft-example},
\Cref{fig:full_orbit} gives the full promotion orbit of $S \coloneqq \std(T)$, which we use to illustrate \Cref{prop:promi_alt_def} and \Cref{thm:antiexcedance}. 
\Cref{prop:promi_alt_def} allows us to write the promotion permutations of $S$ by tracking entries that move between rows in applications of promotion. For example, $\prom_1(T)(1) = 2$ because in the first promotion of \Cref{fig:full_orbit} the label $2$ moves from row $2$ to row $1$ (and then becomes a $1$). Similarly, $\prom_1(T)(2) = 7$ because in the second promotion, the label $6$ moves from row $2$ to row $1$ (and then becomes $5$). More interestingly, $\prom_1(T)(3) = 12$ because in the third promotion, the label $\overline{10}$ moves from row $1$ to row $2$ (and then becomes $\overline{9}$). Finally, we observe that $\prom_2(T)(1) = 5$ because in the first promotion, the label $5$ moves from row $3$ to row $2$ (and then becomes $4$). The reader may enjoy recomputing the rest of \Cref{ex:prom_fncts} by analyzing \Cref{fig:full_orbit} and applying \Cref{prop:promi_alt_def}.

The antiexcedances of $\prom_1(T)$ are $1,3,4,8,9$. Note that $1$ and $4$ are the positive entries in row $1$ of $S$, while $\overline{3}$, $\overline{8}$, and $\overline{9}$ are the negative entries in rows $2$, $3$, and $4$ of $S$. The antiexcedances of $\prom_2(T)$ are $1,2,3,4,7,8$. Note that $1,2,4$ and $7$ are the positive entries in rows $1$ and $2$, while $\overline{3}$ and $\overline{8}$ are the negative entries in rows  $3$, and $4$. Finally, the antiexcedances of $\prom_3(T)$ are $1,2,3,4,5,7,10$. Note that $1,2,4,5,7,10$ are the positive entries in rows $1$, $2$, and $3$, while $\overline{3}$ is the negative entry in row $4$.

\end{example}

\begin{remark}\label{remark:wedonotknow} 
  For $r> 2$, we do not know an intrinsic characterization of the tuples $(\prom_i(T))_{i=0}^r$ that can arise for $T$ an $r$-row rectangular fluctuating tableau. Necessary conditions include the fixed point conditions in \Cref{cor:FPF}, the row increasing conditions from \Cref{lem:M_prom}, the symmetry \Cref{thm:prom_perms}(ii), and the nesting of the antiexcedance sets 
  \begin{align*}
        \{ a \in \Aexc(\prom_i(T)) : a > 0\} &\subseteq \{ a \in \Aexc(\prom_{i+1}(T)) : a > 0\}, \\
        \{ a \in \Aexc(\prom_{i+1}(T)) : a < 0\} &\subseteq \{ a \in \Aexc(\prom_{i}(T)) : a < 0\}
  \end{align*}
  from \Cref{thm:antiexcedance}. See \Cref{remark:see} for further related discussion.
\end{remark}

\section{Dihedral models of promotion and evacuation}\label{sec:dihedral}

Recall that $\promotion$ and $\evacuation$ generate an infinite dihedral action on skew fluctuating tableaux (see \Cref{lem:dihedral}). By \Cref{thm:ft.prom_evac}, $\promotion^n = \id$ on rectangular fluctuating tableaux of length $n$, so there is a non-obvious action of the dihedral group of order $2n$ on such tableaux by promotion and evacuation. In this section, we use the preceding constructions to give a pictorial model of this dihedral action where $\promotion$ corresponds to rotation and $\evacuation$ corresponds to reflection.

Let
  \[ \langle \rot, \refl \mid \rot^n = \refl^2 = 1, \rot\circ\refl = \refl\circ\rot^{-1}\rangle \]
be the dihedral group of symmetries of the regular $n$-gon, where $\rot$ acts as rotation by $2\pi/n$ and $\refl$ acts as reflection through a fixed axis which passes through the midpoint of an edge. For a given type $\underline{c} = (c_1, \ldots, c_n)$ with $|c_1| + \cdots + |c_n| = t$, this group acts on the set of tuples of maps
  \[ (f_i \colon [t] \to [t])_{i=0}^r \]
by
\begin{align*}
  \rot^k \cdot (f_i)
    &\coloneqq (\sigma^{-|c_1| - \cdots - |c_k|} \circ f_i \circ \sigma^{|c_1| + \cdots + |c_k|}) \qquad\text{if }1 \leq k \leq n \\
  \refl \cdot (f_i)
    &\coloneqq (w_0 \circ f_i \circ w_0)_{i=0}^{r-1}.
\end{align*}

Combining \Cref{thm:injective} and \Cref{thm:prom_perms} gives the following.

\begin{corollary}\label{cor:dihedral.1}
  The inclusion
    \[ \Phi(T) \coloneqq (\prom_i(T))_{i=0}^r \]
  from $r$-row rectangular fluctuating tableaux of length $n$ and type $\underline{c}$ to tuples of permutations intertwines promotion and evacuation with the action of the dihedral group of order $2n$:
  \begin{align*}
    \Phi(\promotion(T))  &= \rot \cdot \Phi(T) \\
    \Phi(\evacuation(T)) &= \refl \cdot \Phi(T). \\
  \end{align*}
\end{corollary}

We may encode  $\Phi(T)$ in terms of a diagram on a disk as in the following example. Place $t$ vertices on the circumference in $n$ equally-spaced groups of size $|c_1|, \ldots, |c_n|$ in clockwise order. Place a directed edge from $a$ to $b$ with color $i$ if $\prom_i(T)(a) = b$. Promotion corresponds to counterclockwise rotation by $2 \pi / n$ and evacuation corresponds to reflection across the diameter passing midway between the $1$st and $n$th groups.

\begin{example}\label{ex:dihedral}
  For the fluctuating tableau $T$ from \Cref{fig:ft-example}, the diagrams of the promotion permutations of $T$ are as follows:
  \begin{center}
    \includegraphics[scale=0.35, trim={0.4cm 2.2cm 10.5cm 0.45cm},clip]{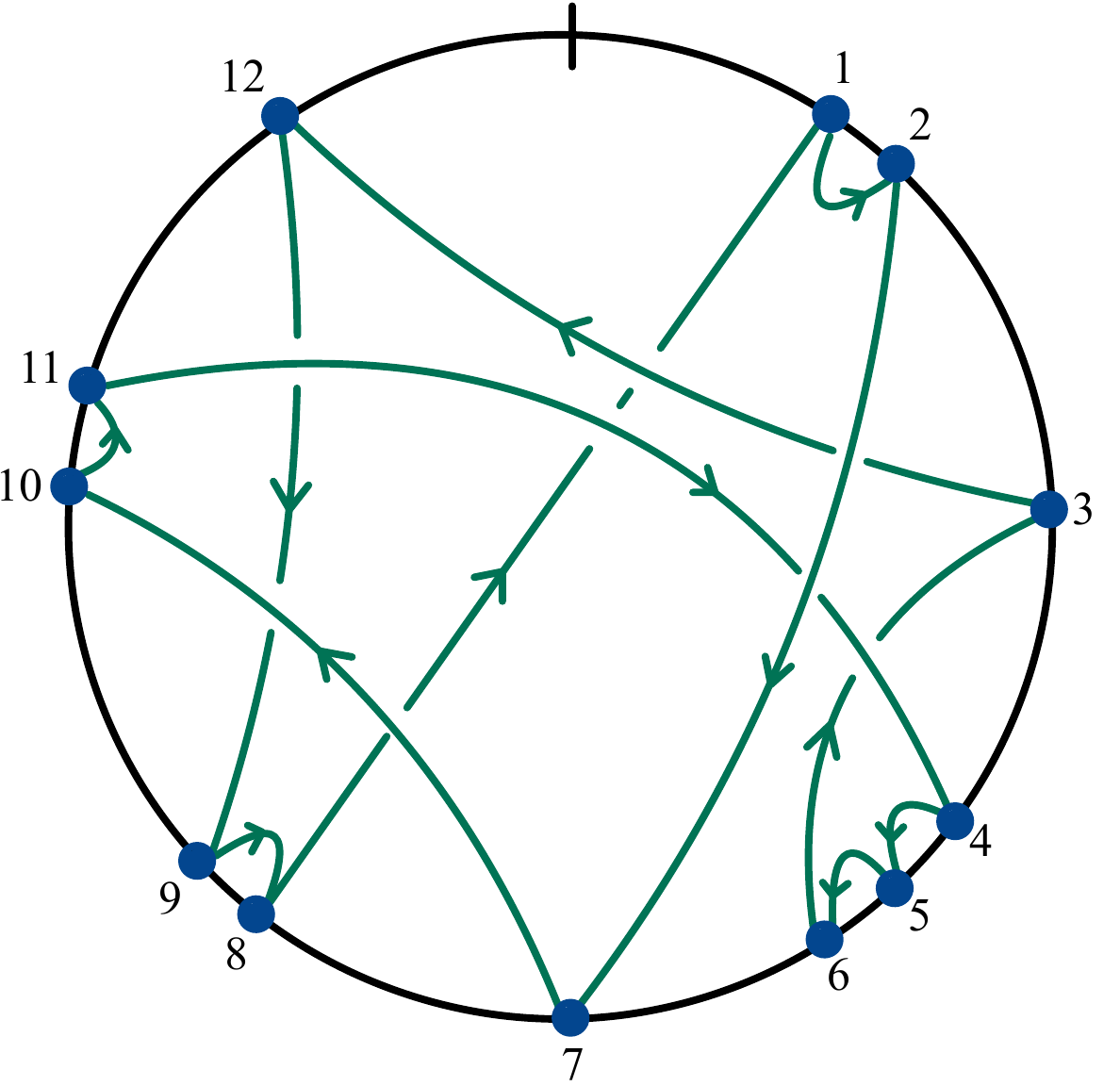}
    \includegraphics[scale=0.35, trim={0.3cm 2.2cm 10.5cm 0.45cm},clip]{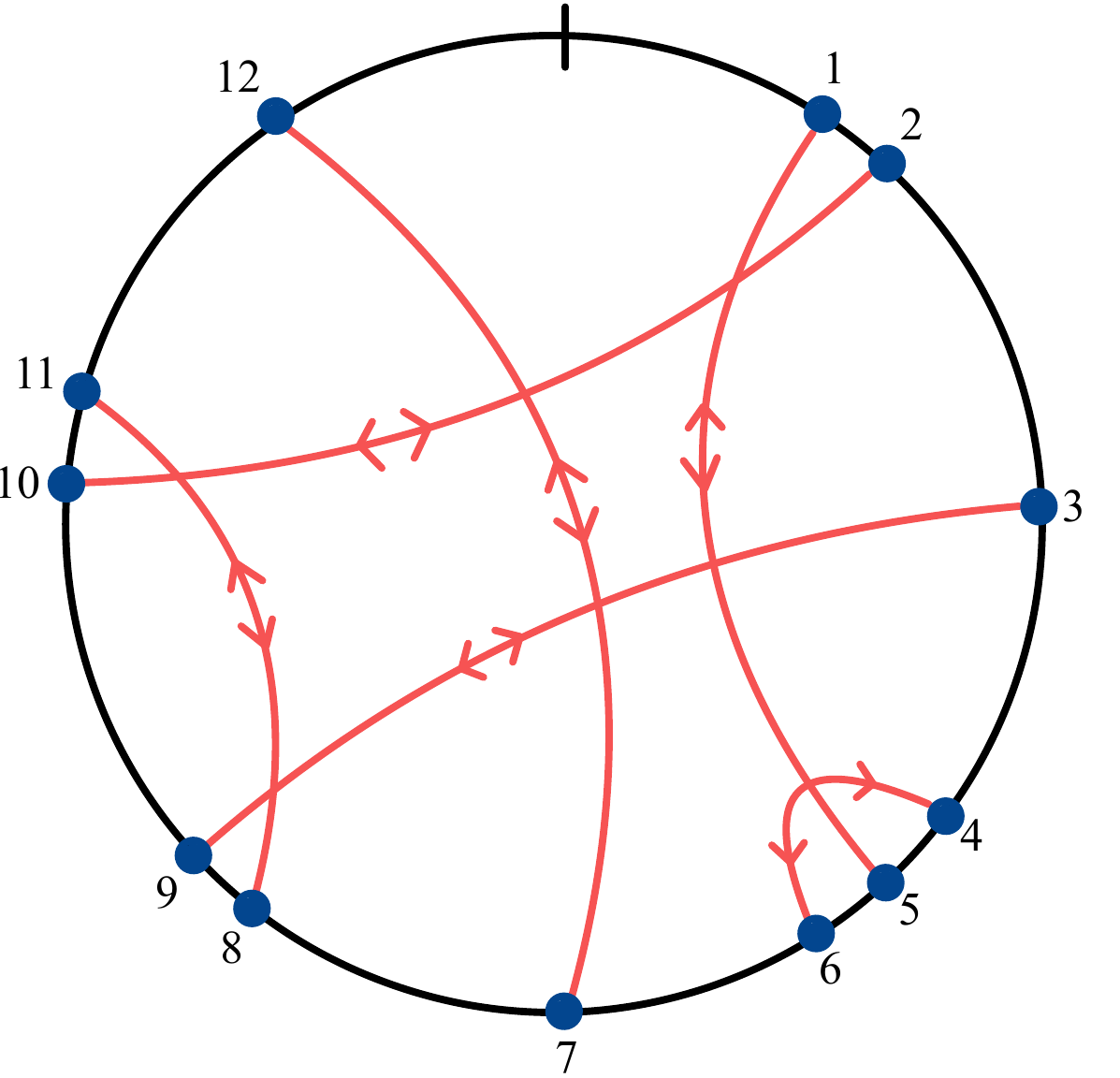}
    \includegraphics[scale=0.35, trim={0.4cm 2.2cm 10.5cm 0.45cm},clip]{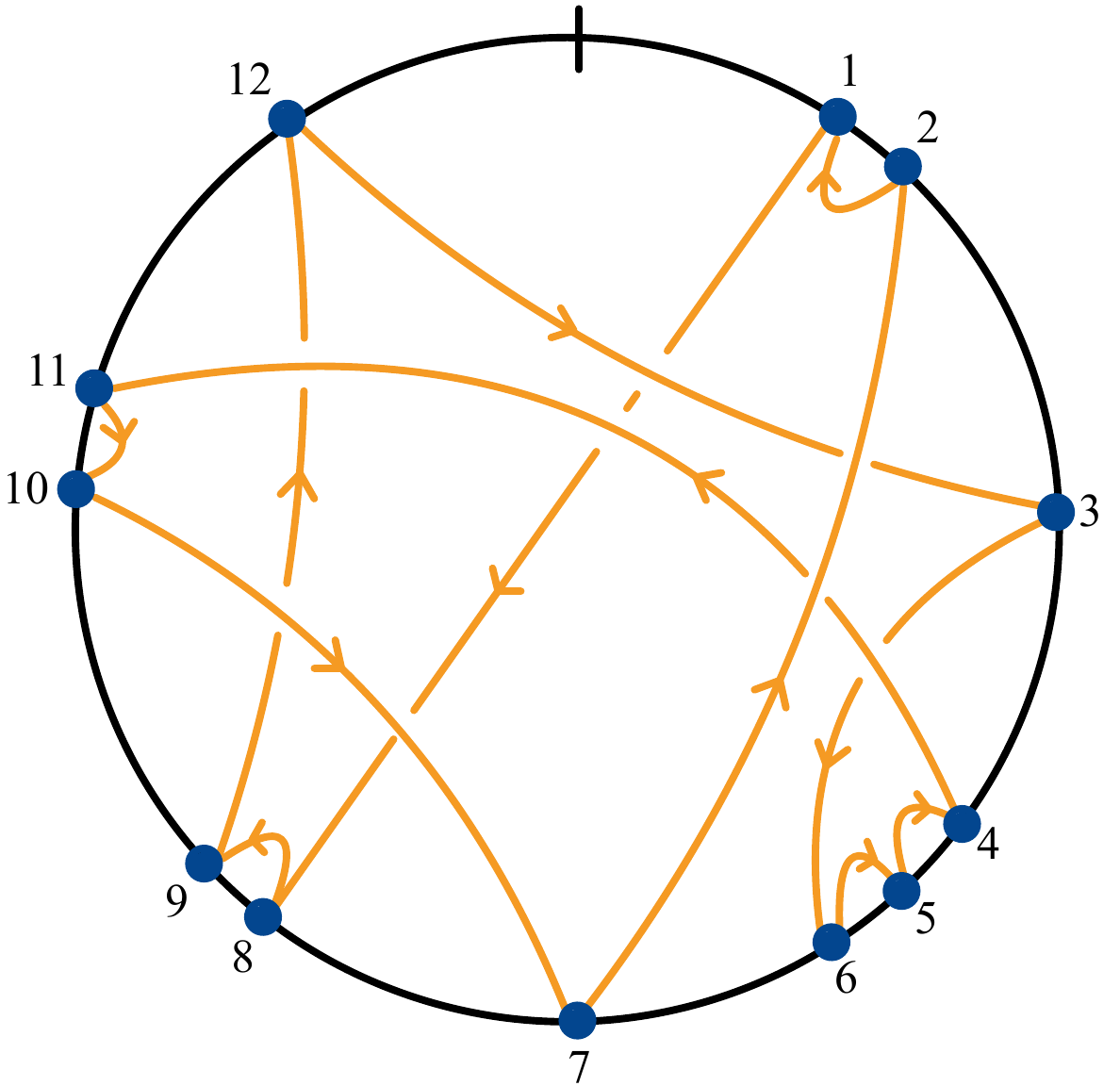}.
  \end{center}
  Here we have drawn the $\prom_i$ diagrams separately for clarity and included only $i=1, 2, 3$ from left to right, though they rotate and reflect simultaneously. The action of $\promotion$ corresponds to counterclockwise rotation by an angle of $2\pi/7$, while that of $\evacuation$ corresponds to reflection across the vertical axis. The fundamental involution $\varpi$ corresponds to reversing the direction of the arrows, $\varepsilon$ agrees with $\evacuation$, and the time reversal involution $\tau$ corresponds to both reflecting the diagrams and reversing the direction of the arrows.
\end{example}

To reduce the number of vertices and edges, one could choose to combine the vertices in each group into a single vertex, and then additionally delete any loops formed by this process. This simplification is related to the reduced promotion matrices of \Cref{sec:reduced_M}.

\begin{remark}\label{remark:see}
This graphical model allows one to directly ``see'' the dihedral action of promotion and evacuation on rectangular fluctuating tableaux, something that has long been desired by combinatorialists even in the standard case. 
  When $r=2$, the map from rectangular fluctuating tableaux to fixed-point free involutions given by $T \mapsto \prom_1(T)$ is injective. Indeed, after combining vertices as described above the image is precisely the \emph{non-crossing matchings}. For $r=3$, a very similar model for semistandard tableaux was considered by Hopkins and Rubey \cite{Hopkins-Rubey}. 

For representation-theoretic purposes, it would be nice to have such a diagrammatic model that naturally extends the dihedral action to a full $\mathfrak{S}_n$-action.
Recall from \Cref{remark:wedonotknow} that we do not have a characterization of promotion permutations in general.
Therefore, for general $r$, we do not know how to use this model 
to define a natural $\mathfrak{S}_n$-action that extends promotion and evacuation on the span of tableaux. For $r=3$, the model can be enriched to a web basis for $\SL_3$ \cite{Kuperberg}, which carries a full $\mathfrak{S}_n$-action (see also, \cite{Petersen-Pylyavskyy-Rhoades, Patrias-Pechenik}). In \cite{Four-row-paper}, we similarly use this model with $r=4$ as a starting point to construct a web basis for $\SL_4$, which again has a full $\mathfrak{S}_n$-action with $\sigma$ acting by $\rot$ and $w_0$ acting by $\refl$. See \Cref{fig:sl4web} for an example.
\end{remark}

\begin{figure}[hbtp]
    \centering
    \includegraphics{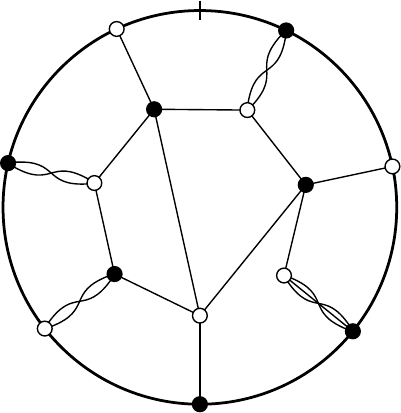}
    \caption{The $\SL_4$-web corresponding to the fluctuating tableau $T$ from \Cref{fig:ft-example}, as constructed in our companion paper \cite{Four-row-paper}.}
    \label{fig:sl4web}
\end{figure}

\section{Connections to crystals}\label{sec:crystals}

We now give a representation-theoretic interpretation of the promotion functions from \Cref{sec:promperm} and reduced promotion matrices from \Cref{sec:prom_stuff} by relating fluctuating tableaux to vertices in crystal graphs. Crystal bases were first introduced by Kashiwara \cite{Kashiwara-Crystals} (see also, Lusztig \cite{Lusztig-Crystals}) and are a combinatorial tool to study representations of quantum groups associated to Lie algebras. See \cite{Hong-Kang,Bump-Schilling} for textbook accounts.

\subsection{Crystal basics}

We largely follow \cite{Bump-Schilling}. Let $\Phi$ be a \emph{root system} for a Euclidean space with positive-definite inner product $\langle -, -\rangle$, $\Phi^+$ a choice of \emph{positive roots}, $\Lambda \supset \Phi$ a \emph{weight lattice}, and $\Sigma = \{\alpha_i : i \in I\} \subset \Phi^+$ a set of \emph{simple roots} with index set $I$ as in \cite[\S2.1]{Bump-Schilling}. We require only finite seminormal Kashiwara crystals, which simplifies the exposition.

\begin{definition}
  A (finite, seminormal, Kashiwara) \emph{crystal} is a nonempty finite set $\mathcal{B}$ together with maps
  \begin{align*}
    e_i, f_i \colon &\mathcal{B} \to \mathcal{B} \sqcup \{0\} \qquad \forall i \in I, \\
    \wt \colon &\mathcal{B} \to \Lambda
  \end{align*}
  where $0 \not\in \mathcal{B}$ is an auxiliary element and which satisfy the following axioms.
  \begin{enumerate}
    \item[(A1)] If $x, y \in \mathcal{B}$, then $e_i(x) = y \Leftrightarrow f_i(y) = x$ and
    \[ \wt(y) = \wt(x) + \alpha_i. \]
    \item[(A2)] For $x \in \mathcal{B}$, the \emph{$i$-root string} through $x$ is the collection $\{e_i^k(x), f_i^k(x) : k \geq 0\} \cap \mathcal{B}$. If $x$ is the maximal element in its $i$-root string, i.e.~$e_i(x) = 0$, we require the $i$-string length, i.e.~the maximal $k$ such that $f_i^k(x) \neq 0$, to be $\langle \wt(x), \alpha_i^\vee\rangle$ where $\alpha^\vee \coloneqq 2\alpha/\langle \alpha, \alpha\rangle$.
  \end{enumerate}
\end{definition}

In type $A$, axiom (A2) says that a basis element $x \in \mathcal{B}$ of weight $\wt(x) = \mu$ which is maximal in its $i$-root string has $i$-string length $\mu_i - \mu_{i+1} \geq 0$. In particular, if $x$ is of \textit{highest weight}, meaning $e_i(x) = 0$ for all $i \in I$, then $\mu_1 \geq \cdots \geq \mu_r$ and $\mu$ is \textit{dominant}.

\begin{definition}
  If $\mathcal{B}$ is a crystal, the \textit{crystal graph} of $\mathcal{B}$ is the labeled, directed graph with vertex set $\mathcal{B}$ and edges $x \too{i} y$ whenever $y = f_i(x) \in \mathcal{B}$.
\end{definition}

The maps $e_i$ may be recovered from the crystal graph and (A1), so the crystal graph together with the weight function $\wt$ entirely encodes the crystal. The edges of $\mathcal{B}$ are oriented ``downwards'' with respect to the partial order on $\Lambda$. One may infer $\wt$ from the weights of the set of highest weight elements,
  \[ \hw(\mathcal{B}) \coloneqq \{x \in \mathcal{B} : \forall i \in I, e_i(x) = 0\}. \]

A \textit{strict morphism} of crystals is a partially defined map $\mathcal{B} \to \mathcal{C}$ which preserves $\wt$ and sends $i$-root strings bijectively onto $i$-root strings. See \cite[\S4.5]{Bump-Schilling} for the formal, general definition. We will only encounter strict morphisms.

\begin{remark}
 For the sake of completeness we briefly describe the correspondence between crystals and representation theory. We will not discuss the technical assumptions necessary to make this correspondence precise.
  \begin{itemize}
    \item Connected components of $\mathcal{B}$ $\longleftrightarrow$ irreducible decomposition of $W$,
    \item Highest weight elements of $\mathcal{B}$ $\longleftrightarrow$ dominant weights of the irreducible components of $W$,
    \item Weight generating function of $\mathcal{B}$ $\longleftrightarrow$ character of $W$,
    \item Tensor product of crystals $\longleftrightarrow$ tensor product of representations,
    \item Isomorphic crystals $\longleftrightarrow$ isomorphic representations.
  \end{itemize}
  In particular, connected crystals coming from representation theory are entirely determined by their unique highest weight element, and $\mathcal{B}$ is uniquely determined by $\hw(\mathcal{B})$.
\end{remark}

The following crystals correspond to the $\GL_r(\mathbb{C})$-representations $\bigwedge^k V$ and $\bigwedge^k V^*$; see \cite[\S2.2]{Bump-Schilling}. These are the representations that are important for fluctuating tableaux (see \Cref{sec:GL_r_rep_theory}). 

\ytableausetup{boxsize=0.6cm}
\begin{example}
  In type $A_{r-1}$, the crystal graph $\mathcal{B}(V)$ of the $r$-dimensional standard representation $V$ of $\GL_r(\mathbb{C})$ is:
    \[ \boxed{1} \too{1} \boxed{2} \too{2} \cdots \too{r-1} \boxed{r}, \qquad\text{where }\wt\left(\boxed{i}\right) = \mathbf{e}_i. \]
  The unique highest weight element is $\boxed{1}$ of weight $\mathbf{e}_1$.
\end{example}

\begin{example}
  The crystal graph $\mathcal{B}(V^*)$ of the dual $V^*$ is:
    \[ \boxed{\overline{r}} \too{r-1} \boxed{\overline{r-1}} \too{r-2} \cdots \too{1} \boxed{\overline{1}}, \qquad\text{where }\wt\left(\boxed{\overline{i}}\right) = -\mathbf{e}_i. \]
  The unique highest weight element is $\boxed{\overline{r}}$ of weight $-\mathbf{e}_r$.
\end{example}

\begin{example}
  The crystal graph of the exterior power $\bigwedge^k V$ of the standard representation of $\GL_r(\mathbb{C})$ has vertex set:
    \[ \mathcal{B}\left(\bigwedge\nolimits^k V\right) = \left\{\ytableaushort{{j_1},{j_2},{\vdots},{j_k}}: 1 \leq j_1 < j_2 < \dots < j_k \leq r \right\}. \]
  Given a column $C \in \mathcal{B}(\bigwedge^k V)$, we can apply $f_i$ to $C$ by changing $i$ into $i+1$ as long as $C$ does not already contain $i+1$. The weight of $C$ is $\sum_{j \in C} \mathbf{e}_j$. The unique highest weight element is the column consisting of $1, 2, \ldots, k$ with weight $\omega_k \coloneqq \mathbf{e}_1 + \cdots + \mathbf{e}_k$.
\end{example}

\begin{example}
  The crystal graph of the dual exterior power $\bigwedge^k V^*$ of the standard representation of $\GL_r(\mathbb{C})$ has vertex set:
    \[ \mathcal{B}\left(\bigwedge\nolimits^k V^*\right) = \left\{\ytableaushort{{\overline{j_k}},{\vdots},{\overline{j_2}},{\overline{j_1}}}: 1 \leq j_1 < j_2 < \dots < j_k \leq r \right\}. \]
  Given a column $\overline{C} \in \mathcal{B}(\bigwedge^k V^*)$, we can apply $f_i$ to $\overline{C}$ by changing $\overline{i+1}$ into $\overline{i}$ as long as $\overline{C}$ does not already contain $\overline{i}$. The weight of $\overline{C}$ is $\sum_{\overline{j} \in \overline{C}} -\mathbf{e}_j$. The unique highest weight element is the column consisting of $\overline{r}, \overline{r-1}, \ldots, \overline{r-k+1}$ with weight $\overline{\omega}_k = -\mathbf{e}_r - \cdots - \mathbf{e}_{r-k+1}$.
\end{example}

\begin{example}\label{ex:crystal_wedge2}
  When $r=4, k=2$, the crystal graph $\mathcal{B}(\bigwedge^2 V)$ is:
  \begin{center}
  \begin{tikzcd}[row sep=tiny]
    \ 
      & \ 
      & \ytableaushort{{1}, {4}} \ar{dr}{1} \\
    \ytableaushort{{1}, {2}} \rar{2}
      & \ytableaushort{{1}, {3}} \ar{ur}{3} \ar{dr}[swap]{1}
      & \ 
      & \ytableaushort{{2}, {4}} \rar{2}
      & \ytableaushort{{3}, {4}} \\
    \ 
      & \ 
      & \ytableaushort{{2}, {3}} \ar{ur}[swap]{3}
  \end{tikzcd}.
  \end{center}
\end{example}

In general, the \textit{dual} of a crystal graph is obtained by reversing the arrows and negating the weights; see \cite[Def.~2.20]{Bump-Schilling}.

\subsection{Tensor products of crystals}

Crystals have the following combinatorial tensor product due to Kashiwara \cite[Prop.~6]{Kashiwara-Crystals}. For a crystal $\mathcal{B}$ and $x \in \mathcal{B}$, set
  \[ \varphi_i(x) \coloneqq \max\{k \in \mathbb{Z}_{\geq 0} \mid f_i^k(x) \neq 0\}\qquad\text{and}\qquad \varepsilon_i(x) \coloneqq \max\{k \in \mathbb{Z}_{\geq 0} \mid e_i^k(x) \neq 0\}. \]
Hence, $\varphi_i(x) + \varepsilon_i(x)$ is the length of the $i$-string through $x$.

\begin{definition}
  Suppose $\mathcal{B}, \mathcal{C}$ are crystals for $\Phi$. The \textit{tensor product} $\mathcal{B} \otimes \mathcal{C}$ is the crystal with vertex set $\{x \otimes y : x \in \mathcal{B}, y \in \mathcal{C}\}$, weight function $\wt(x \otimes y) \coloneqq \wt(x) + \wt(y)$, and maps
  \begin{equation}\label{eq:tensor_rule}
  \begin{split}
    f_i(x \otimes y) &\coloneqq
    \begin{cases}
        f_i(x) \otimes y & \text{if $\varphi_i(x) > \varepsilon_i(y)$},\\
        x \otimes f_i(y) & \text{if $\varphi_i(x) \le \varepsilon_i(y)$},
    \end{cases} \\
    e_i(x \otimes y) &\coloneqq
    \begin{cases}
        e_i(x) \otimes y & \text{if $\varphi_i(x) \ge \varepsilon_i(y)$},\\
        x \otimes e_i(y) & \text{if $\varphi_i(x) < \varepsilon_i(y)$}.
    \end{cases}
  \end{split}
  \end{equation}
\end{definition}

\begin{proposition}[See {\cite[Prop.~2.29]{Bump-Schilling}}]
  If $\mathcal{B}$ and $\mathcal{C}$ are crystals, then  $\mathcal{B} \otimes \mathcal{C}$ is a crystal.
\end{proposition}

\begin{remark}
  We use the original tensor product convention due to Kashiwara, which is compatible with the lattice words from \Cref{sec:lattice}. In combinatorics (e.g.~\cite[Def.~2.3]{Bump-Schilling}), it is common to write the tensor factors in the opposite order.
\end{remark}

See \Cref{ex:crystal_tensor.1} and \Cref{ex:crystal_tensor.2} for examples of tensor products.

\subsection{Crystals and fluctuating tableaux}\label{sec:crystals.tableaux}

Recall the notation of \Cref{sec:GL_r_rep_theory}, in particular the tensor products of exterior powers $\bigwedge^{\underline{c}} V$ from \eqref{eq:V_c_wedge}. The bracketing rule gives a crystal structure on the tensor product:
  \[ \mathcal{B}\left(\bigwedge\nolimits^{\underline{c}} V\right) \coloneqq \mathcal{B}\left(\bigwedge\nolimits^{c_1} V\right) \otimes \cdots \otimes \mathcal{B}\left(\bigwedge\nolimits^{c_n} V\right). \]

To apply $e_i$ or $f_i$ to $w_1 \otimes \cdots \otimes w_n \in \mathcal{B}(\bigwedge^{\underline{c}} V)$ where each $w_j \in \mathcal{A}_r$ (where $1 \leq i \leq r-1$), we apply the following combinatorial algorithm (\emph{bracketing rule}).
\begin{enumerate}[(1)]
  \item Place $[$ below each letter containing $i$ but not $i+1$, or $\overline{i+1}$ but not $\overline{i}$.
  \item Place $]$ below each letter containing $i+1$ but not $i$, or $\overline{i}$ but not $\overline{i+1}$.
  \item Match brackets from the inside out.
  \item Now $f_i$ acts on the letter with the left-most unmatched $[$ by replacing the $i$ with $i+1$ or the $\overline{i+1}$ with $\overline{i}$.
  \item Similarly $e_i$ acts on the letter with the right-most unmatched $]$ by replacing the $i+1$ with $i$ or the $\overline{i}$ with $\overline{i+1}$.
\end{enumerate}

\begin{example}\label{ex:crystal_tensor.1}
  Over $\GL_3(\mathbb{C})$, consider $\mathcal{B}(V)^{\otimes 8} = \mathcal{B}(\bigwedge^{\underline{c}} V)$ where $\underline{c} = (1, \ldots, 1) = (1^8)$. We have:
  \begin{center}
  \begin{tikzcd}[row sep=0em, column sep=0.2em]
    f_2( 
      & 3 \ar[otimes]{r}
      & 2 \ar[otimes]{r}
      & 1 \ar[otimes]{r}
      & 2 \ar[otimes]{r}
      & 2 \ar[otimes]{r}
      & 3 \ar[otimes]{r}
      & 3 \ar[otimes]{r}
      & 1
      & )
      & = 3 \ar[otimes]{r}
      & f_2(2) \ar[otimes]{r}
      & 1 \ar[otimes]{r}
      & 2 \ar[otimes]{r}
      & 2 \ar[otimes]{r}
      & 3 \ar[otimes]{r}
      & 3 \ar[otimes]{r}
      & 1 \\
    \ 
      & {]}
      & {[}
      & \ 
      & {[} \arrow[rrr, start anchor=base, end anchor=base, no head, yshift=-1.5em, decorate, decoration={brace, mirror}]
      & {[} \arrow[r, start anchor=base, end anchor=base, no head, yshift=-1em, decorate, decoration={brace, mirror}]
      & {]}
      & {]} 
      & \ 
      & \
      & = 3 \ar[otimes]{r}
      & 3 \ar[otimes]{r}
      & 1 \ar[otimes]{r}
      & 2 \ar[otimes]{r}
      & 2 \ar[otimes]{r}
      & 3 \ar[otimes]{r}
      & 3 \ar[otimes]{r}
      & 1 \\
  \end{tikzcd}.
  \end{center}
\end{example}

\begin{example}\label{ex:crystal_tensor.2}
  Over $\GL_4(\mathbb{C})$, consider
    \[ \mathcal{B}\left(\bigwedge\nolimits^2 V \otimes V^* \otimes \bigwedge\nolimits^3 V \otimes V \otimes \bigwedge\nolimits^2 V^* \otimes \bigwedge\nolimits^2 V \otimes V^*\right), \]
  which is $\mathcal{B}(\bigwedge^{\underline{c}} V)$ for $\underline{c} = (2, \overline{1}, 3, 1, \overline{2}, 2, \overline{1})$. We have:
  \begin{center}
  \begin{tikzcd}[row sep=0em, column sep=0.2em]
    f_1( 
      & \{12\} \ar[otimes]{r}
      & \overline{4} \ar[otimes]{r}
      & \{134\} \ar[otimes]{r}
      & 2 \ar[otimes]{r}
      & \{\overline{3}\overline{2}\} \ar[otimes]{r}
      & \{34\} \ar[otimes]{r}
      & \overline{1}
      & )
      & = 0 \\
    \ 
      & \ 
      & \ 
      & {[} \arrow[r, start anchor=base, end anchor=base, no head, yshift=-1em, decorate, decoration={brace, mirror}]
      & {]}
      & {[} \arrow[rr, start anchor=base, end anchor=base, no head, yshift=-1em, decorate, decoration={brace, mirror}]
      & \ 
      & {]} \\
  \end{tikzcd}.
  \end{center}
  Here, we have written columns $C$ or $\overline{C}$ as sets for brevity. The root strings in $\mathcal{B}(\bigwedge^k V)$ are each of length $1$, so we only need to consider single brackets in each tensor factor here. This basis element is simultaneously highest and lowest weight, i.e.~all $f_i, e_i$ send it to $0$. This example is based on the fluctuating tableau from \Cref{fig:ft-example}.
\end{example}

We may drop the tensor product symbols and identify elements of $\mathcal{B}(\bigwedge^{\underline{c}} V)$ with words whose letters are elements of $\mathcal{A}_r$, which appeared in \Cref{sec:lattice} and \Cref{sec:involutions}. Under this identification, as we now show, highest-weight elements of $\mathcal{B}(\bigwedge^{\underline{c}} V)$ are precisely the lattice words of fluctuating tableaux of type $\underline{c}$ whose weight and final shape coincide. This perspective gives a crystal-theoretic proof of \Cref{thm:ft.irreps}.

\begin{proposition}
  Let $\underline{c} = (c_1, \ldots, c_n)$ with $c_i \in \{0, \pm 1, \ldots, \pm r\}$. Then:
    \[ \hw\left(\mathcal{B}\left(\bigwedge\nolimits^{\underline{c}} V\right)\right) = L(\ft(r, n, \underline{c})). \]
\end{proposition}
\begin{proof}
  Recall that lattice words $L$ of fluctuating tableaux are characterized by the inequalities \eqref{eq:lattice_inequalities}. Concretely, in each prefix of $L$, we require the number of $a$'s minus the number of $\overline{a}$'s to be weakly greater than the number of $b$'s minus the number of $\overline{b}$'s for all $1 \leq a \leq b \leq r$. We may equivalently restrict these conditions to the case where $a=i, b=i+1$ for $1 \leq i \leq r-1$.

  Now, consider a highest-weight word $L$. In order for $e_i$ to result in $0$, the bracketing rule must result in no unmatched $]$'s. Equivalently, in each prefix of the sequence of brackets, there must be at least as many $[$'s as $]$'s. The $[$'s arise precisely from letters containing ($i$ but not $i+1$) or $(\overline{i+1}$ but not $\overline{i}$). The $]$'s arise precisely from letters containing $i+1$ but not $i$ or $\overline{i}$ but not $\overline{i+1}$. It is straightforward that these conditions are equivalent to the conditions on lattice words.
\end{proof}

The same argument shows that the \textit{lowest} weight elements of $\mathcal{B}(\bigwedge^{\underline{c}} V)$ are precisely the \textit{reverse} lattice words, namely those where in every \textit{suffix} the number of $i$'s minus the number of $\overline{i}$'s is at \textit{most} the number of $i+1$'s minus the number of $\overline{i+1}$'s. In particular, rectangular fluctuating tableaux are special in the following sense.

\begin{corollary}\label{cor:isolatedVertices}
  The isolated vertices of $\mathcal{B}(\bigwedge^{\underline{c}} V)$, i.e.~the simultaneous highest and lowest weight elements, are precisely the rectangular fluctuating tableaux of type $\underline{c}$.
\end{corollary}

\subsection{Crystals, oscillization, and switch}

We may interpret oscillization from \Cref{sec:oscillization} as arising from a crystal morphism. First, we have a map
  \[ \mathcal{B}\left(\bigwedge\nolimits^k V\right) \to \mathcal{B}(V)^{\otimes k} \]
given by sending $\{i_1 < \cdots < i_k\}$ to $i_1 \otimes \cdots \otimes i_k$. There is an analogous dual notion sending $\{\overline{i_1} < \cdots < \overline{i_k}\}$ to $\overline{i_1} \otimes \cdots \otimes \overline{i_k}$ with highest weight $\{\overline{r}, \cdots, \overline{r-k+1}\} \mapsto \overline{r} \otimes \cdots \otimes \overline{r-k+1}$. We leave the proof of the following to the reader.

\begin{lemma}
  Let $\underline{c} = (c_1, \ldots, c_n)$ with $c_i \in \{0, \pm 1, \ldots, \pm r\}$ and let $\std(\underline{c})$ be the sequence where $c_i$ is replaced by $|c_i|$ copies of $1$ if $c_i \geq 0$ and $\overline{1}$ if $c_i \leq 0$. The oscillization map
    \[ \std \colon \ft(r, \lambda, \underline{c})) \to \ft(r, \lambda, \std(\underline{c})) \]
  is the inclusion on highest-weight elements induced by the crystal inclusion
    \[ \mathcal{B}\left(\bigwedge\nolimits^{\underline{c}} V\right) \to \mathcal{B}\left(\bigwedge\nolimits^{\std(\underline{c})} V\right). \]
\end{lemma}

The $\toggle$ maps from \Cref{def:toggle} similarly arise from the crystal isomorphisms
\begin{align*}
  \mathcal{B}\left(\bigwedge\nolimits^c V\right) &\too{\sim} \mathcal{B}\left(\det(V) \otimes \bigwedge\nolimits^{r-c} V^*\right) \\
  \{i_1 < \cdots < i_c\} &\mapsto \{1, \ldots, r\} \otimes \{\overline{j_1} < \cdots < \overline{j_{r-c}}\}  \\
  \mathcal{B}\left(\bigwedge\nolimits^c V^*\right) &\too{\sim} \mathcal{B}\left(\det(V^*) \otimes \bigwedge\nolimits^{r-c} V\right) \\
  \{\overline{i_1} < \cdots < \overline{i_c}\} &\mapsto \{\overline{r}, \ldots,  \overline{1}\} \otimes \{j_1 < \cdots < j_{r-c}\} \\
\end{align*}
where $\{i_1, \ldots, i_c\} \sqcup \{j_1, \ldots, j_{r-c}\} = [r]$. In terms of lattice words, we may freely add or remove the letters $\{1, \dots, r\}$ or $\{\overline{r}, \dots, \overline{1}\}$ without materially altering the combinatorics of the preceding sections.

\subsection{Crystals and promotion}\label{sec:crystals_promotion}

We now describe crystal-theoretic interpretations of the Bender--Knuth involutions (following Lenart \cite{Lenart}) and promotion (following Pfannerer--Rubey--Westbury \cite{Pfannerer-Rubey-Westbury}). We then give a different crystal-theoretic ``balance point'' description of promotion (\Cref{prop:first_balance}), as well as a crystal-theoretic interpretation of reduced promotion matrices (\Cref{thm:Mbar_crystal}).

Although $\mathcal{B} \otimes \mathcal{C} \cong \mathcal{C} \otimes \mathcal{B}$, the naive map $x \otimes y \mapsto y \otimes x$ does not respect the bracketing rule and is not generally a morphism of crystals. Henriques--Kamnitzer \cite{Henriques.Kamnitzer} introduced \textit{crystal commutors}
  \[ \sigma_{\mathcal{B}, \mathcal{C}} \colon \mathcal{B} \otimes \mathcal{C} \to \mathcal{C} \otimes \mathcal{B}. \]
These are natural involutive crystal isomorphisms that make the category of crystals into a \textit{coboundary category}. There is a corresponding action of the \textit{$n$-fruit cactus group} on $n$-fold tensor products of crystals. See \cite[\S2.4]{Lenart} for details. We will not require the specifics of these constructions until \Cref{sec:crystals-involutions}.

Lenart \cite{Lenart} interpreted the local rules of van Leeuwen \cite{vanLeeuwen} in terms of crystal commutors. Specifically, we have the following.

\begin{theorem}[{\cite[Thm.~4.4]{Lenart}}]\label{thm:Lenart}
  The bijection on highest weight elements induced by
    \[ \sigma_{\mathcal{B}(\bigwedge\nolimits^{\underline{c}} V), \mathcal{B}\left(\bigwedge\nolimits^{\underline{d}} V\right)} \colon \mathcal{B}\left(\bigwedge\nolimits^{\underline{c}} V\right) \otimes \mathcal{B}\left(\bigwedge\nolimits^{\underline{d}} V\right) \too{\sim} \mathcal{B}\left(\bigwedge\nolimits^{\underline{d}} V\right) \otimes \mathcal{B}\left(\bigwedge\nolimits^{\underline{c}} V\right) \]
  is given by the $|\underline{c}| \times |\underline{d}|$ growth diagram
  \begin{equation*}
  \begin{tikzcd}
    \lambda \rar{\underline{d}}
      & \nu \\
    \varnothing \uar{\underline{c}} \rar[swap]{\underline{d}}
      & \mu \uar[swap]{\underline{c}} \\
  \end{tikzcd}.
  \end{equation*}
\end{theorem}
\begin{proof}
One may reduce to the transpose semistandard case using the $\toggle$ involutions and the naturality of the crystal commutor. In that setting, the theorem is a direct restatement of \cite[Thm.~4.4]{Lenart} into our terminology. 
\end{proof}

In particular, promotion has the following crystal-theoretic description, as observed in \cite[\S4]{Pfannerer-Rubey-Westbury}. Recall we may identify a fluctuating tableau $T$ with its lattice word $L(T)$. 

\begin{corollary}\label{cor:crystalPromotion}
  Let $\underline{c} = (c_1, \underline{d})$ and $\underline{c}' = (\underline{d}, c_1)$. The bijection on highest weight elements induced by:
    \[ \sigma_{\mathcal{B}\left(\bigwedge\nolimits^{c_1} V\right), \mathcal{B}\left(\bigwedge\nolimits^{\underline{d}} V\right)} \colon \mathcal{B}\left(\bigwedge\nolimits^{\underline{c}} V\right) \to \mathcal{B}\left(\bigwedge\nolimits^{\underline{c}'} V\right) \]
  is given by promotion $\promotion$ on fluctuating tableaux of type $\underline{c}$.
\end{corollary}

\begin{remark}\label{rem:other_types}
Let $G$ be a Lie group with a representation $U$ and corresponding crystal graph $\mathcal{B}$. A more general version of \Cref{cor:isolatedVertices} (cf.\ \cite{Westbury}) states that a basis of $\Inv_{G}(U^{\otimes n})$ is indexed by the isolated vertices in the crystal $\mathcal{B}^{\otimes n}$, which are often identified with various kinds of tableaux. 
In particular, 
\begin{itemize}
    \item for the adjoint representation of $\GL_r$, these vertices are \emph{Stembridge's alternating tableaux} \cite{Stembridge:rational};
    \item  for the vector representation of 
$\Sp_{2r}$ these are \emph{Sundaram's $r$-symplectic oscillating tableaux} \cite{Sundaram};
\item for the vector representation of $\SO(2r+1)$ these are \emph{vacillating tableaux} \cite{Jagenteufel};
\item and the spin representation of $\Spin(2r+1)$ these are \emph{$r$-fans of Dyck paths} \cite{Pappe-Pfannerer-Schilling-Simone}.
\end{itemize}
  
\Cref{cor:crystalPromotion} motivates a definition of promotion on such kinds of tableaux using the crystal commutor $\sigma_{\mathcal{B},\mathcal{B}^{\otimes (n-1)}}$. A description of promotion in terms of local rules using \Cref{thm:Lenart} is then possible whenever the representation $U$ can be embedded in a tensor product of minuscule representations. In the definition from Equation~\eqref{eq:local_rule.sort} of the local rules, $\sort$ needs to be replaced with the function that maps a weight to the unique dominant representative in its Weyl group orbit.
All of the families of tableaux listed above may be embedded in the set of fluctuating tableaux using identifications coming from virtual crystal morphisms. Hence, promotion on fluctuating tableaux extends promotion on all of the other families listed above.
See \cite{Westbury, Pfannerer-Rubey-Westbury, Pappe-Pfannerer-Schilling-Simone, Henrickson.Stokke.Wiebe} for further discussion of promotion in other Lie types.
\end{remark}

It will be convenient to $0$-index some lattice words, contrary to our earlier convention.

\begin{proposition}\label{prop:prom_raising}
  Let $w_0 \ldots w_{n-1} \in \ft(n, r, \underline{c})$. To compute $\promotion(w_0 \ldots w_{n-1})$, do the following.
  \begin{enumerate}[(i)]
    \item Delete $w_0$.
    \item Apply raising operators $e_i$ to $w_1 \ldots w_{n-1}$ to reach a highest weight element $w_1' \ldots w_{n-1}'$.
    \item Append the unique element $w_n' \in \mathcal{B}(\bigwedge\nolimits^{c_1} V)$ such that the weight of $w_0 \ldots w_{n-1}$ agrees with the weight of $w_1' \ldots w_{n-1}' w_n'$.
  \end{enumerate}
\end{proposition}

\begin{proof}
\Cref{cor:crystalPromotion} explains the relation between promotion of fluctuating tableaux and crystal commutors. The rest of the proposition is a special case of the description from \cite[Cor.~4.19]{Pfannerer-Rubey-Westbury} of crystal commutors on highest weight elements of tensor products of crystals.
\end{proof}

The intermediate steps of the ``raising algorithm'' in \Cref{prop:prom_raising} do not obviously correspond to the intermediate steps in the computation of promotion via Bender--Knuth involutions as in \Cref{ex:running.BK_P}. Nonetheless, using the bracketing rule, we may give a more explicit combinatorial description of this crystal-theoretic promotion algorithm when $|c_1| = 1$ .

\begin{definition}\label{def:balance}
  Let $w_0 \ldots w_{n-1} \in \ft(n, r, \underline{c})$ be the lattice word of a fluctuating tableau. An \emph{$i$-balance point} is an index $0 \leq j \leq n-1$ such that, in the subword $w_0 \ldots w_j$, the number of $i$'s minus the number of $\overline{i}$'s equals the number of $i+1$'s minus the number of $\overline{i+1}$'s. We call the difference $\# i - \# \overline{i} - \#(i+1) + \#\overline{i+1}$ the \emph{slack} of the the index $j$ with respect to $i$. The index $j$ is an $i$-balance point if and only if its slack with respect to $i$ is $0$.
\end{definition}

\begin{proposition}\label{prop:first_balance}
Let $w \coloneqq w_0 \ldots w_{n-1} \in \ft(n, r, \underline{c})$ with $|c_1| = 1$.
  Suppose the raising operators in \Cref{prop:prom_raising} acting on $w_1 \ldots w_{n-1}$ and resulting in the highest weight element $w_1' \ldots w_{n-1}'$ are, in order, $e_{i_1}, \ldots, e_{i_k}$, acting on positions $j_1, \ldots, j_k$. 

  Then we have the following:
  \begin{itemize}
    \item If $c_1 = 1$, then $i_1, \ldots, i_k = 1, 2, \ldots, k$ and $w_n' = k+1$.
    \item If $c_1 = -1$, then $i_1, \ldots, i_k = r-1, r-2, \ldots, r-k$ and $w_n' = \overline{r-k}$.
  \end{itemize}
  
  For convenience, set $j_0 \coloneqq 0$ and $i_{k+1} \coloneqq k+1$ if $c_1 = 1$ and $i_{k+1} \coloneqq r-k-1$ if $c_1 = -1$.
  Then, for each $1 \leq h \leq k$,
  $j_h$ is the first $h$-balance point of $w$ weakly after $j_{h-1}$. There is no $(k+1)$-balance point of $w$ weakly after $j_k$. In particular, the sequence $1 \leq j_1 \leq \cdots \leq j_k \leq n-1$ weakly increases, and in the rectangular case, $k=r-1$.
\end{proposition}
\begin{proof}
  Let $w = w_0 w_1 \ldots w_{n-1} \in \ft(n, r, \underline{c})$. We assume $c_1 = 1$, the case $c_1 = -1$ being similar. Consider applying $e_a$, for some $1 \leq a \leq r-1$,to $w$ by writing $[_a$'s and $]_a$'s below appropriate letters. Since $w$ is a highest weight word, every $]_a$ is matched with a $[_a$. In particular, the number of $[_a$'s in any prefix is at least as great as the number of $]_a$'s. The difference between these two numbers is the slack, as in \Cref{def:balance}. Since $w_0 = 1$ has $[_1$, it is matched in $w$ with the $]_1$ at the first $1$-balance point of $w$, say $j_1 \geq 1$, if it exists. Cutting away $w_0$, the brackets $[_a$ and $]_a$ for $w' \coloneqq w_1 \ldots w_{n-1}$ with $a>1$ are unchanged from those in $w$, hence fully matched, and the only raising operator which may possibly apply to $w'$ is $e_1$. Indeed, $e_1$ applies to $w'$ precisely at position $j_1$, which we may assume exists, since otherwise we are done.

  Let $w''$ be the result of applying $e_1$ to $w'$ at $w_{j_1}$. We have two cases.
  \begin{itemize}
    \item Suppose $w_{j_1}$ has $2 \in w_{j_1}$, $1 \not\in w_{j_1}$, and is decorated with $]_1$ and $[_2$ in $w$. Applying $e_1$ results in $1$, which is decorated with $[_1$ in $w''$.
    \item Suppose $w_{j_1}$ has $\overline{1} \in w_{j_1}$, $\overline{2} \not\in w_{j_1}$, and is decorated with $]_1$ and $[_2$ in $w$. Applying $e_1$ results in $\overline{2}$, which is decorated with $[_1$ and $]_2$ in $w''$.
  \end{itemize}
  In particular, we see that $w''$ has no unmatched $]_1$'s or $]_a$'s for $a > 2$, since $w$ has none. As for $a=2$, the $[_2$'s and $]_2$'s in $w''$ are the same as those in $w$ except that the $[_2$ at $w_{j_1}$ has been deleted and possibly replaced with $]_2$. Since $e_2$ applies to the rightmost unmatched $]_2$, in either case we see that $e_2$ applies to $w''$ at the $]_2$ matched with the $[_2$  of $w_{j_1}$ in $w$ if it exists, which is at $j_2$, the first $2$-balance point weakly after $j_1$. Continuing in this way, we see inductively that $i_1, \ldots, i_k$, $j_1, \ldots, j_k$, and $k$ are as described.

  In the rectangular case, for each $1 \leq a \leq r-1$, the definition of rectangularity gives that there is an $a$-balance point weakly after the last $(a-1)$-balance point, so the necessary balance points always exist.
\end{proof}

When applying the raising algorithm in \Cref{prop:prom_raising} when $|c_1| = 1$, \Cref{prop:first_balance} shows that there is a unique crystal raising path. When $|c_1|>1$, however, there are generally multiple possible paths. Nonetheless, we may oscillize the $\bigwedge^{c_1} V$ and repeatedly apply \Cref{prop:first_balance}, which gives a deterministic calculation, as in the following example.

\begin{example}
  Let $w = \{12\}\overline{4}\{134\}2\{\overline{3}\overline{2}\}\{34\}\overline{1}$ be the lattice word of the fluctuating tableau from \Cref{fig:ft-example}, where $r=4$. Oscillizing the first letter and applying each of \Cref{prop:prom_raising} and \Cref{prop:first_balance} twice yields
  
  \begin{minipage}{0.5\textwidth}
  \begin{align*}
    &\blue{\{12\}}\overline{4}\{134\}2\{\overline{3}\overline{2}\}\{34\}\overline{1} \\
    \too{\std} & \blue{1}2\overline{4}\{134\}2\{\overline{3}\overline{2}\}\{34\}\overline{1} \\
    \too{\text{cut}} & \blue{2}\overline{4}\{134\}2\{\overline{3}\overline{2}\}\{34\}\overline{1} \\
    \too{e_1} & 1\overline{4}\blue{\{134\}}2\{\overline{3}\overline{2}\}\{34\}\overline{1} \\
    \too{e_2} & 1\overline{4}\{124\}2\blue{\{\overline{3}\overline{2}\}}\{34\}\overline{1} \\
    \too{e_3} & 1\overline{4}\{124\}2\{\overline{4}\overline{2}\}\{34\}\overline{1} \\
    \too{\text{app.}} & 1\overline{4}\{124\}2\{\overline{4}\overline{2}\}\{34\}\overline{1}\blue{4} \\
  \end{align*}
  \end{minipage}
  \begin{minipage}{0.5\textwidth}
  \begin{align*}
    & \blue{1}\overline{4}\{124\}2\{\overline{4}\overline{2}\}\{34\}\overline{1}4 \\
    \too{\text{cut}} & \overline{4}\{124\}\blue{2}\{\overline{4}\overline{2}\}\{34\}\overline{1}4 \\
    \too{e_1} & \overline{4}\{124\}1\{\overline{4}\overline{2}\}\blue{\{34\}}\overline{1}4 \\
    \too{e_2} & \overline{4}\{124\}1\{\overline{4}\overline{2}\}\{24\}\overline{1}\blue{4} \\
    \too{e_3} & \overline{4}\{124\}1\{\overline{4}\overline{2}\}\{24\}\overline{1}3 \\
    \too{\text{app.}} & \overline{4}\{124\}1\{\overline{4}\overline{2}\}\{24\}\overline{1}3\blue{4} \\
    \too{\std^{-1}} & \overline{4}\{124\}1\{\overline{4}\overline{2}\}\{24\}\overline{1}\blue{\{34\}}. \\
  \end{align*}
  \end{minipage}
  Here, we have highlighted in blue each tensor factor that is being acted on.

  Consequently, $\promotion(\{12\}\overline{4}\{134\}2\{\overline{3}\overline{2}\}\{34\}\overline{1}) = \overline{4}\{124\}1\{\overline{4}\overline{2}\}\{24\}\overline{1}\{34\}$, in agreement with the calculation from \Cref{ex:PEdiagram_running} by Bender--Knuth involutions.
  
  Alternatively, we could cut away $\{12\}$ entirely and apply the raising algorithm in \Cref{prop:prom_raising} directly.
  The corresponding crystal operations occur in a crystal isomorphic to \Cref{ex:crystal_wedge2}, where two possible paths could be taken. The analogue of the path from the oscillization calculation is the path that applies $e_2, e_3, e_1, e_2$ in that order to $w_1 \ldots w_6 = \overline{4}\{134\}2\{\overline{3}\overline{2}\}\{34\}\overline{1}$ at positions $2, 4, 3, 5$, respectively.
  
  By the next result, we may use the this path to determine the non-diagonal entries in the top rows of the reduced promotion matrices $\PMr^1(T)$, $\PMr^2(T)$, and $\PMr^3(T)$ from \Cref{ex:PMr}. For example, the $e_2$ is applied at positions $2$ and $5$, resulting in $1$'s in those columns (indexed starting at $0$) of the top row of $\PMr^2(T)$ and $0$'s elsewhere.
\end{example}

We now give a crystal-theoretic interpretation of reduced promotion matrices, strengthening the link between crystals and promotion permutations.

\begin{theorem}\label{thm:Mbar_crystal}
  Let $T$ be an $r$-row fluctuating tableau of length $n$. Fix $1 \leq j \leq n-1$. 
  
  Then $\PMr^i(T)_{u,u+j}$ (with the column index $u+j$ taken modulo $n$) is the number of times the raising operator $e_i$ is applied at index $j$ of the lattice word $L(\promotion^{u-1}(T))$ when computing the promotion of $\promotion^{u-1}(T)$ by the raising algorithm in \Cref{prop:prom_raising}. 
\end{theorem}
\begin{proof}
  It suffices to consider a single local rule diagram as in \Cref{prop:PMr_sum}:
    \begin{center}
  \begin{tikzcd}
  \lambda \ar{r}[name=U]{}
    & \nu \\
  \kappa \uar{} \ar{r}[swap,name=D]{}
    & \mu \ar{u}[swap]{}
  \ar[to path={(U) node[midway] {$\PMr^i$}  (D)}]{}
  \end{tikzcd}
  \end{center}
  where $\lambda = \kappa + \mathbf{e}_A$, $\nu = \mu + \mathbf{e}_B$, $\mu = \kappa + \mathbf{e}_U$, $\nu = \lambda + \mathbf{e}_V$  and $\PMr^i$ are the reduced promotion matrix entries.
  
    Now, the top edge corresponds to the letter $V$ in a crystal word and the bottom edge corresponds to the letter $U$. If $e_i$ is applied to $V$ a total of $m_i$ times during the crystal raising algorithm, then
     \begin{equation}\label{eq:last_one}
         \mathbf{e}_V = \mathbf{e}_U + \sum_{i=1}^{r-1} m_i (\mathbf{e}_{i+1} - \mathbf{e}_i).
     \end{equation}

 The entries of $\PMr^i$ are uniquely determined from $A$ and $B$ by \eqref{eq:PMr_sum} in \Cref{prop:PMr_sum}.
   Since $\mathbf{e}_B - \mathbf{e}_A = \mathbf{e}_V - \mathbf{e}_U$, 
  the theorem then follows by the uniqueness in \Cref{prop:PMr_sum}, comparing \eqref{eq:PMr_sum} to \eqref{eq:last_one}.
\end{proof}

\subsection{Crystals and a fundamental involution}\label{sec:crystals-involutions}

The involution $\varepsilon$ from \Cref{sec:involutions} may be interpreted in terms of crystals using Lusztig's involution. We sketch this connection here.

\textit{Lusztig's involution} \cite{Lusztig-Crystals-II} is a certain involution $\eta$ on crystals that sends elements of weight $\alpha$ to elements of weight $\rev(\alpha)$; see \cite[\S2.2]{Henriques.Kamnitzer} or \cite[\S4.1]{Pfannerer-Rubey-Westbury} for details. In particular, $\eta$ acts on $\mathcal{B}\left(\bigwedge\nolimits^c V\right)$ by ``complementing'' elements, i.e.~by sending $\pm S \subseteq [r]$ to $\pm \{r+1-s : s \in S\}$. Lusztig's involution interchanges the unique highest and lowest weight elements of a connected crystal.

\begin{theorem}\label{thm:eta-epsilon}
    On lattice words of $r$-row fluctuating tableaux, we have:
    \begin{equation}\label{eq:eta-epsilon}
        \operatorname{\eta} \circ \evacuation =  \operatorname{\varepsilon} = \devacuation \circ \operatorname{\eta}.
    \end{equation}
\end{theorem}

\begin{proof}
    Henriques--Kamnitzer \cite{Henriques.Kamnitzer} defined an action of the \textit{$n$-fruit cactus group} on $n$-fold tensor products of crystals. The $n$-fruit cactus group is a certain group generated by elements $s_{p, q}$ for intervals $[p, q] \subseteq [n]$. By \cite[p.207]{Henriques.Kamnitzer}, the action can be defined on words by:
      \[ s_{p, q}(w_1 \cdots w_n) = w_1 \cdots w_{p-1} \eta(\eta(w_q)\cdots\eta(w_p)) w_{q+1} \cdots w_n. \]
    The involution $\varepsilon$ reverses words and complements letters. Thus,
      \[ s_{p, q}(w_1 \cdots w_n) = w_1 \cdots w_{p-1}\eta(\varepsilon(w_p\cdots w_q))w_{q+1} \cdots w_n. \]
    In particular, $s_{1, n} = \operatorname{\eta} \circ \operatorname{\varepsilon}$ on $\mathcal{B}\left(\bigwedge\nolimits^{\underline{c}} V\right)$.
    
    The special case $s_{p, p+1}$ corresponds to acting on the $p$th and $(p+1)$st factors of $\bigwedge\nolimits^{\underline{c}} V$ by the crystal commutors from \Cref{sec:crystals_promotion}, fixing the other factors. By \cite[Lem.~4.2]{Pfannerer-Rubey-Westbury} and Lenart's \Cref{thm:Lenart}, $s_{1, n}(L) = \evacuation(L)$ when $L$ is the lattice word of a fluctuating tableau of length $n$. Hence $\evacuation = \operatorname{\eta} \circ \operatorname{\varepsilon}$, so $\operatorname{\eta} \circ \evacuation = \operatorname{\varepsilon}$. Using \Cref{lem:invs.PEEd}(iii), we also have
      \[ \devacuation = \operatorname{\varepsilon} \circ \evacuation \circ \operatorname{\varepsilon} = \operatorname{\varepsilon} \circ \operatorname{\eta}. \]
    The theorem follows.
\end{proof}

Since Lusztig's involution is the identity on isolated vertices, which are precisely the rectangular fluctuating tableaux, an alternate proof of \Cref{thm:ft.prom_evac} is as
an immediate corollary of \eqref{eq:eta-epsilon}.

\section*{Acknowledgements}
This project began during the 2021 BIRS Dynamical Algebraic Combinatorics program hosted at UBC Okanagan, and we are very grateful for the excellent research environment provided there. At that conference, Sam Hopkins and Martin Rubey introduced us to the notion of $\prom_1$ for rectangular standard tableaux. We also wish to thank Rebecca Patrias and Anne Schilling for their helpful comments. We thank Sami Assaf for the phrase `numbers in boxes with rules' from her FPSAC 2018 talk, which we used in the introduction. Part of this work was done at NDSU, for whose hospitality we are very thankful. We are grateful for the resources provided at ICERM, where this paper was completed. We thank the anonymous referees for their helpful comments.

\bibliographystyle{amsalphavar}
\bibliography{fluctuating}
\end{document}